\documentclass[preprint]{elsarticle}
\usepackage{graphicx}
\usepackage{natbib,amssymb}
\usepackage{amsmath}
\usepackage{amsthm}
\usepackage[T1]{fontenc}
\usepackage{epsfig}
\usepackage{epstopdf}
\usepackage{fancybox}
\usepackage{float}
\usepackage{pifont}
\usepackage{fancyhdr}
\usepackage{ulem}
\usepackage{url}
\usepackage{color}
\usepackage{placeins}

\def \RR {\mathbb{R}}

\def \NN {\mathbb{N}}
\def \ZZ {\mathbb{Z}}
\def \L {\mathcal{L}}
\def \N {\mathcal{N}}

\def\EE {\mathbb{E}}

\newtheorem{prop}{Proposition}[section]
\newtheorem{lem}[prop]{Lemma}
\newtheorem{cond}[prop]{Condition}
\newtheorem{defi}{Definition}[section]

\newtheorem{rem}[defi]{Remark}

\newcommand{\argmin}{\mathop{\mathrm{argmin}}}
\newcommand{\diag}{\mathop{\mathrm{diag}}}
\newcommand{\var}{\mathop{\mathrm{var}}}
\newcommand{\cov}{\mathop{\mathrm{cov}}}

\title{Asymptotic analysis of the role of spatial sampling for covariance parameter estimation of Gaussian processes}
              
\tnotetext[t1]{A supplementary material is attached in the electronic version of the article.}

\author[fb1,fb2]{Fran\c cois Bachoc\corref{cor1}}

\cortext[cor1]{Corresponding author: Fran\c cois Bachoc \\
CEA-Saclay, DEN, DM2S, STMF, LGLS, F-91191 Gif-Sur-Yvette, France \\
Phone: +33 (0) 1 69 08 97 91 \\
Email: francois.bachoc@cea.fr}                                                

\address[fb1]{CEA-Saclay, DEN, DM2S, STMF, LGLS, F-91191 Gif-Sur-Yvette, France \\}
\address[fb2]{Laboratoire de Probabilit\'es et Mod\`eles Al\'eatoires, Universit\'e Paris Diderot, \\
75205 Paris cedex 13}

\begin{document}

\begin{abstract}
   
Covariance parameter estimation of Gaussian processes is analyzed in an asymptotic framework.
The spatial sampling is a randomly perturbed regular grid and its deviation from the perfect regular grid is controlled by a single scalar regularity parameter.
Consistency and asymptotic normality are proved for the Maximum Likelihood
and Cross Validation estimators of the covariance parameters. The asymptotic covariance matrices
of the covariance parameter estimators are deterministic functions of the regularity parameter.
By means of an exhaustive study of the asymptotic covariance matrices,
it is shown that the estimation is improved when the regular grid is strongly perturbed.
Hence, an asymptotic confirmation is given to the commonly admitted fact that using groups of observation points with small spacing
is beneficial to covariance function estimation.
Finally, the prediction error, using a consistent estimator of the covariance parameters, is analyzed in details.
\end{abstract}


\maketitle

\section{Introduction}  \label{section: introduction}

In many areas of science that involve measurements or data acquisition, one often has
to answer the question of how the set of experiments should be designed \citep{DAE}.
It is known that in many situations, an irregular, or even random, spatial sampling is preferable to a regular one. Examples of
these situations are found in many fields. For numerical integration, Gaussian quadrature rules generally yield irregular grids \citep[ch.4]{NRASC}.
The best known low-discrepancy sequences for quasi-Monte Carlo methods (van der Corput, Halton, Sobol, Faure, Hammersley,...) are not regular either
\citep{RNGQMCM}.
In the compressed sensing domain, it has been shown that one can
recover a signal very efficiently, and at a small cost, by using random measurements \citep{NOSRFRPUES}. 

In this paper, we are focused on the role of spatial sampling for meta-modeling. Meta-modeling is particularly
relevant for the analysis of complex computer models \citep{TDACE}. We will address the case of Kriging models,
which consist in interpolating the values of a Gaussian random field given observations at a finite set of observation points.
Kriging has become a popular method for a
large range of applications, such as numerical code approximation \citep{DACE,TDACE} and calibration \citep{CCMMO} or global optimization \citep{EGOEBBF}.

One of the main issues regarding Kriging is the choice of the covariance function
for the Gaussian process.
Indeed, a Kriging model yields an unbiased predictor with minimal variance and a correct predictive variance only if the correct covariance
function is used. The most common practice is to statistically estimate the covariance function, from a set
of observations of the Gaussian process, and to plug \citep[ch.6.8]{ISDSTK} the estimate in the Kriging equations.
Usually, it is assumed that the covariance function belongs to a given parametric family (see \cite{RGRFCF} for a review of classical families).
In this case, the estimation boils down to estimating the corresponding covariance parameters.

The spatial sampling, and particularly its degree of regularity,
plays an important role for the covariance function estimation.
In chapter 6.9 of \cite{ISDSTK}, it is shown that adding three observation points with small spacing
to a one-dimensional regular grid of twenty points dramatically improves the estimation in two ways. First, it enables to detect
without ambiguities that a Gaussian covariance model is poorly adapted, when the true covariance function is Mat\'ern $\frac{3}{2}$.
Second, when the Mat\'ern model is used for estimation, it subsequently improves the estimation of the smoothness parameter.
It is shown in \cite{SSDUIAF} that the optimal samplings, for maximizing the log of the determinant of the Fisher information matrix,
averaged over a Bayesian prior on the true covariance parameters, contain closely spaced points. Similarly, in the geostatistical community,
it is acknowledged that adding sampling crosses, that are small crosses of observation points making the different input quantities vary
slightly, enables a better identification of the small scale behavior of the random field, and therefore a better overall estimation of its covariance
function \citep{Jeannee2008geostatistical}.
The common conclusion of the three examples we have given is that irregular samplings, in the sense that they contain at least pairs of observation points
with small spacing, compared to the average density of observation points in the domain, work better for covariance function estimation
than regular samplings, that is samplings with evenly spaced points. This conclusion has become a commonly admitted fact in the Kriging literature.

In this paper, we aim at confirming this fact in an asymptotic framework.
Since exact finite-sample results are generally not reachable
and not meaningful as they are specific to the situation, asymptotic theory
is widely used to give approximations of the estimated covariance parameter distribution.

The two most studied asymptotic frameworks are the increasing-domain and fixed-domain asymptotics \citep[p.62]{ISDSTK}.
In increasing-domain asymptotics, a minimal spacing
exists between two different observation points, so that the infinite sequence of observation points is unbounded. In fixed-domain asymptotics, the sequence
is dense in a bounded domain.

In fixed-domain asymptotics, significant results are available concerning the estimation of the covariance function, and its influence on
Kriging predictions. In this asymptotic framework, two types of covariance parameters can be distinguished:
microergodic and non-microergodic covariance parameters. Following the definition
in \cite{ISDSTK}, a covariance parameter is microergodic if two covariance functions are orthogonal whenever they differ for it (as in \cite{ISDSTK}, we say that
two covariance functions are orthogonal if the two underlying Gaussian measures are orthogonal). Non-microergodic covariance parameters cannot be consistently
estimated, but have no asymptotic influence on Kriging predictions \citep{AEPRFMCF,BELPUICF,UAOLPRFUISOS,IEAEIMBG}. On the contrary,
there is a
fair amount of literature on consistently estimating microergodic covariance parameters, notably using the Maximum Likelihood (ML) method. Consistency has been proved for
several models \citep{APMLEDGP,MLEPUSSS,ESCMSGRFM,IEAEIMBG,FDASMTGRF,CSSVGRF}.
Microergodic covariance parameters have an asymptotic influence on predictions, as shown in \cite[ch.5]{MCSLMMNA}.

Nevertheless, the fixed-domain asymptotic framework is not well adapted to study the influence of the
irregularity of the spatial sampling on covariance parameter estimation. Indeed, we would like to compare sampling techniques
by inspection of the asymptotic distributions
of the covariance parameter estimators. In fixed-domain asymptotics, when an asymptotic distribution is proved for ML \citep{APMLEDGP,MLEPUSSS,FDAPTMLE},
it turns out that it is independent of the dense sequence of observation points.
This makes it impossible to compare the effect of spatial sampling on covariance parameter estimation using fixed-domain asymptotics techniques.

The first characteristic of increasing-domain asymptotics is that, as shown in subsection \ref{subsection: inf_hyp_miss},
all the covariance parameters have strong asymptotic influences on predictions.
The second characteristic is that all the covariance parameters
(satisfying a very general identifiability assumption) can be consistently estimated,
and that asymptotic normality generally holds \citep{UANMLE,MLEMRCSR,ADREMLE}.
Roughly speaking, increasing-domain asymptotics is characterized by a vanishing dependence between
observations from distant observation points.
As a result, a large sample size
gives more and more information about the covariance structure. 
Finally, we show that the asymptotic variances of the covariance parameter estimators strongly depend on the spatial sampling. This is why we address
the increasing-domain asymptotic framework
to study the influence of the spatial sampling on the covariance parameter estimation.

We propose a sequence of random spatial samplings of size $n \in \NN$.
The regularity of the spatial sampling sequence is characterized by a regularity parameter $  \epsilon \in [0,\frac{1}{2})$.
$\epsilon=0$ corresponds to a regular grid, and the
irregularity is increasing with $\epsilon$.
We study the ML estimator, and also a Cross Validation (CV) estimator \citep{PACHGP,KCVMSD}, for which, to the best of our knowledge, no asymptotic results
are yet available in the literature. For both estimators, we prove an asymptotic normality result for the estimation, with a $\sqrt{n}$ convergence, and an
asymptotic covariance matrix which is a deterministic function of $\epsilon$.
The asymptotic normality yields, classically, approximate confidence intervals for finite-sample estimation.
Then, carrying out an exhaustive analysis of the asymptotic covariance matrix,
for the one-dimensional Mat\'ern model, we show that large values of the regularity parameter $\epsilon$ always yield an improvement
of the ML estimation. We also show that ML has a smaller asymptotic variance than CV, which is expected since we address the well-specified case here,
in which the true covariance function does belong to the parametric set used for estimation. Thus, our general conclusion is a confirmation
of the aforementioned results in the literature: using a large regularity parameter $\epsilon$ yields groups of observation points
with small spacing, which improve the ML estimation, which is the preferable method to use.

The rest of the article is organized as follows. In section \ref{section: context}, we introduce the random sequence of observation points, that is parameterized
by the regularity parameter $\epsilon$. We also present the ML and CV estimators. In section \ref{section: consistency_asymptotic_normality},
we give the asymptotic normality results. In section \ref{section: numerical_study}, we carry out an exhaustive study of the asymptotic covariance matrices
for the Mat\'ern model in dimension one. In section \ref{section: analysis_prediction}, we analyze the Kriging prediction for the asymptotic
framework we consider.
In sections \ref{section: appendix_proof_normality}, \ref{section: appendix_proof_echange_limite_derivee} and
\ref{section: proof_prediction_independente_estimation},
we prove the results of sections \ref{section: consistency_asymptotic_normality}, \ref{section: numerical_study}
and \ref{section: analysis_prediction}.
In section \ref{section: appendix_technical_results}, we state and prove several technical results. Finally, section
\ref{section: Toeplitz} is dedicated to the one-dimensional case, with $\epsilon=0$.
We present an efficient calculation of the asymptotic variances for ML and CV
and of the second derivative
of the asymptotic variance of ML, at $\epsilon=0$, using properties of Toeplitz matrix sequences.

\section{Context} \label{section: context}

\subsection{Presentation and notation for the spatial sampling sequence}

We consider a stationary Gaussian process $Y$ on $\RR^d$. We denote $\Theta = [\theta_{inf},\theta_{sup}]^p$.
The subset $\Theta$ is compact in $(- \infty,\infty)^p$. The covariance function of $Y$
is $K_{\theta_0}$ with
$ \theta_{inf} < \left(\theta_0\right)_i < \theta_{sup}$, for $1 \leq i \leq p$.
$K_{\theta_0}$ belongs to a parametric model
$\{ K_{\theta} , \theta \in \Theta \}$, with $K_{\theta}$ a stationary covariance function.

We shall assume the following condition for the parametric model, which is satisfied in all the most classical cases, and especially for the Mat\'ern model that
we will analyze in detail in sections \ref{section: numerical_study} and \ref{section: analysis_prediction}.

\begin{cond}  \label{cond: Ktheta}
 \begin{itemize}
  \item For all $\theta \in \Theta$, the covariance function $K_{\theta}$ is stationary.
\item The covariance function $K_{\theta}$ is three times differentiable with respect to $\theta$.
For all $q \in \left\{0,...,3 \right\}$, $i_1,...,i_q \in  \left\{ 1,...,p \right\}$, there exists $C_{i_1,...,i_q} < + \infty$ so that for all $\theta \in \Theta$,
$t \in \RR^d$,
\begin{equation} \label{eq: controleCov}
\left| \frac{\partial}{\partial \theta_{i_1}}...\frac{\partial}{\partial \theta_{i_q}}  K_{\theta}\left(t\right) \right| \leq \frac{C_{i_1,...,i_q}}{ 1+|t|^{d+1} },
\end{equation}
\end{itemize}
where $|t|$ is the Euclidean norm of $t$.
We define the Fourier transform of a function $h: \RR^d \to \RR$ by $\hat{h}(f) = \frac{1}{(2 \pi)^d} \int_{\RR^d} h(t) e^{- \mathrm{i} f \cdot t } dt$,
where $\mathrm{i}^2 = -1$.
Then, for all $\theta \in \Theta$, the covariance function $K_{\theta}$ has a Fourier transform $\hat{K}_{\theta}$ that is continuous and bounded.
 \begin{itemize}
  \item For all $\theta \in \Theta$, $K_{\theta}$ satisfies
\[
K_{\theta}(t) = \int_{\RR^d} \hat{K}_{\theta}(f)  e^{ \mathrm{i} f \cdot t } df. 
\]
\item $\hat{K}_{\theta}\left(f\right): (\RR^d \times \Theta) \to \RR$ is continuous and positive.
 \end{itemize}

\end{cond}

Let $\NN^*$ be the set of positive integers. We denote by $\left(v_i\right)_{i \in \NN^*}$ a sequence of deterministic points in $(\NN^*)^d$ so that for all
$N \in \NN^*$, $\left\{v_i,1\leq i \leq N^d\right\} = \{ 1,...,N \}^d$.
$Y$ is observed at the points $v_i + \epsilon X_i$, $1 \leq i \leq n$, $n \in \NN^*$, with $ \epsilon \in [0,\frac{1}{2})$
and $X_i \sim_{iid} \mathcal{L}_X$. $\mathcal{L}_X$ is a symmetric probability distribution with support $S_X \subset [-1,1]^d$, and
with a positive probability density function on $S_X$. Two remarks can be made on this sequence of observation points:

\begin{itemize}
 \item This is an increasing-domain asymptotic context. The condition $ \epsilon \in [0,\frac{1}{2})$ ensures a minimal spacing between two
distinct observation points.
\item The observation sequence we study is random, and the parameter $\epsilon$ is a regularity parameter. $\epsilon = 0$ corresponds to a regular
observation grid, while when $\epsilon$ is close to $\frac{1}{2}$, the observation set is irregular
and, more specifically, groups of observation points with small spacing appear.
Examples of observation sets are given in figure \ref{fig: grilles_perturbees}, with $d=2$, $n=8^2$, and different values of $\epsilon$.
\end{itemize}

\begin{figure}[]
\centering
 \hspace*{-2cm}

\begin{tabular}{c c c}
\includegraphics[width=5cm,angle=0]{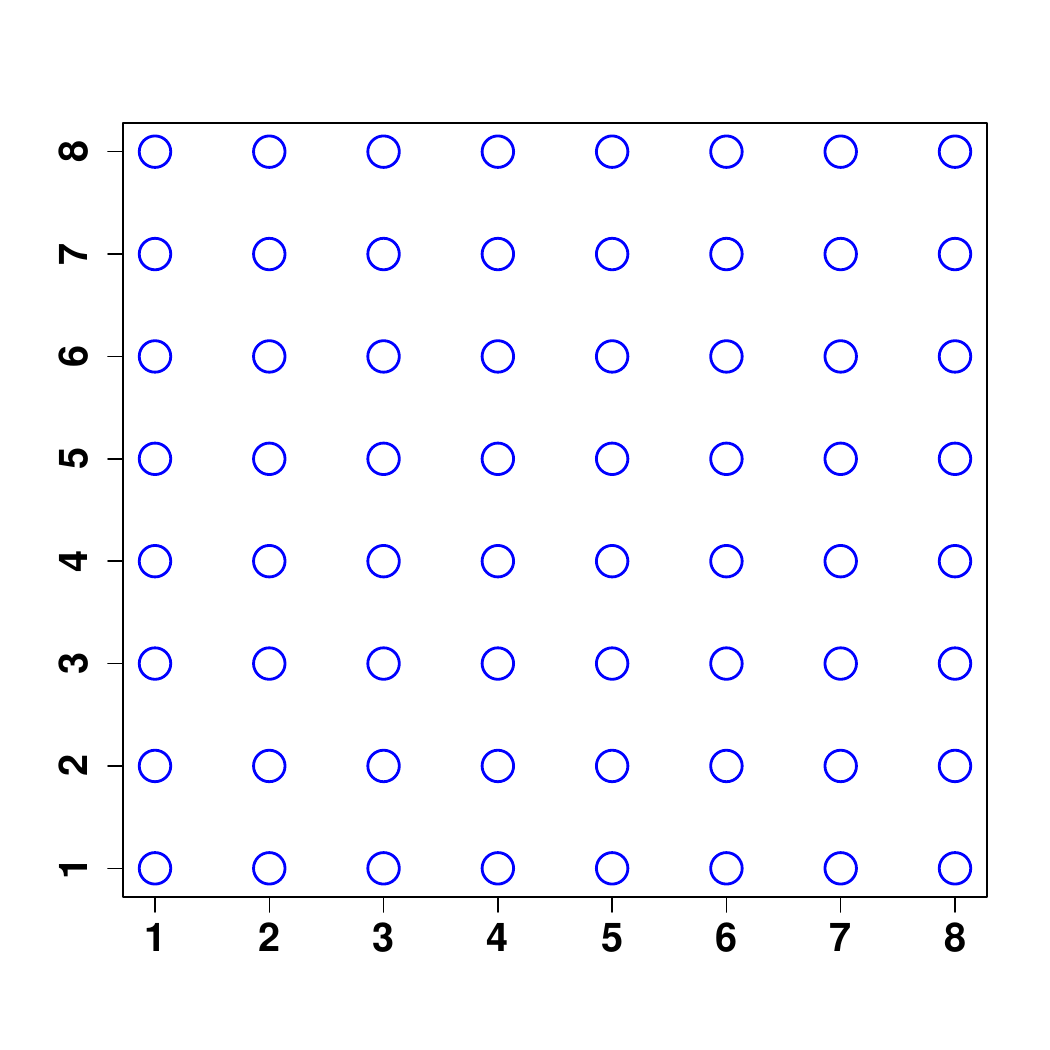} &
\includegraphics[width=5cm,angle=0]{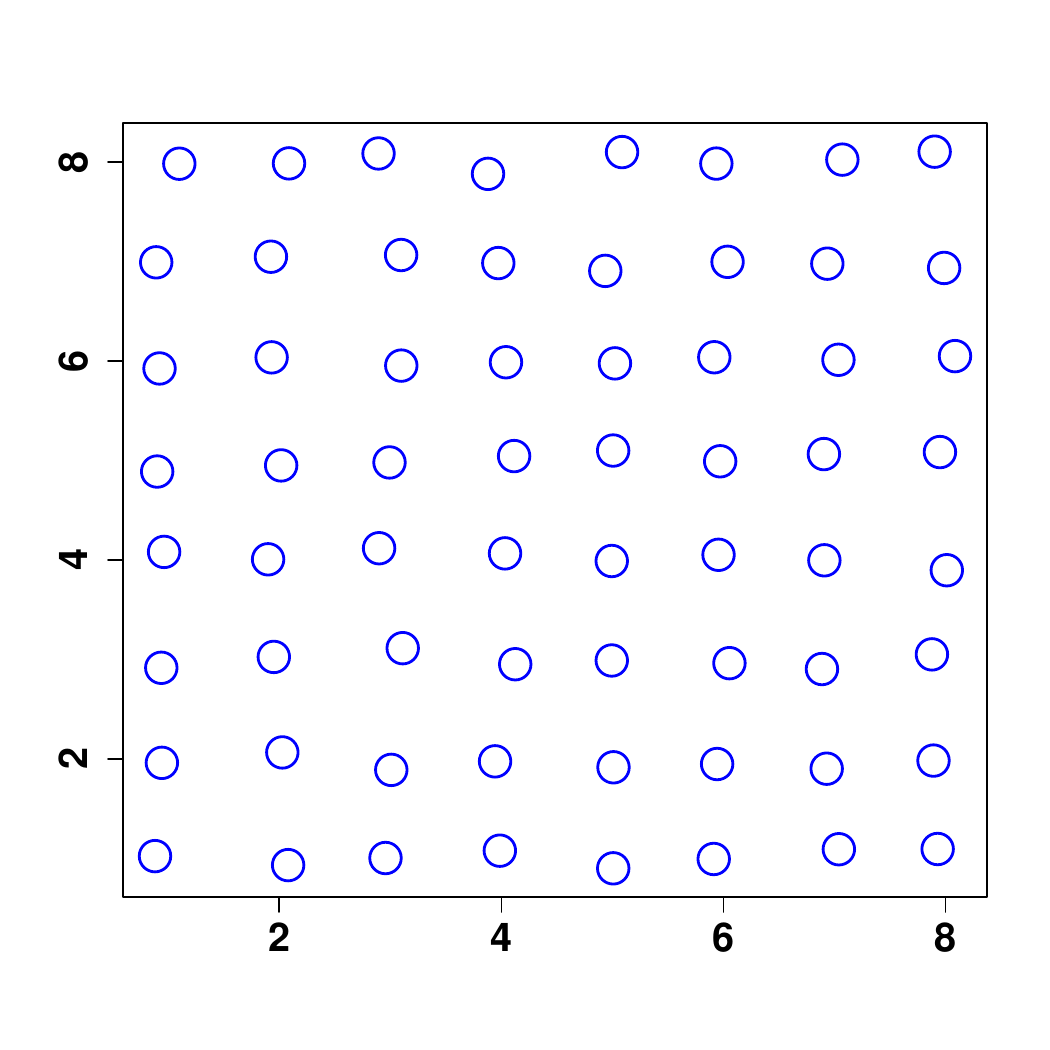} &
\includegraphics[width=5cm,angle=0]{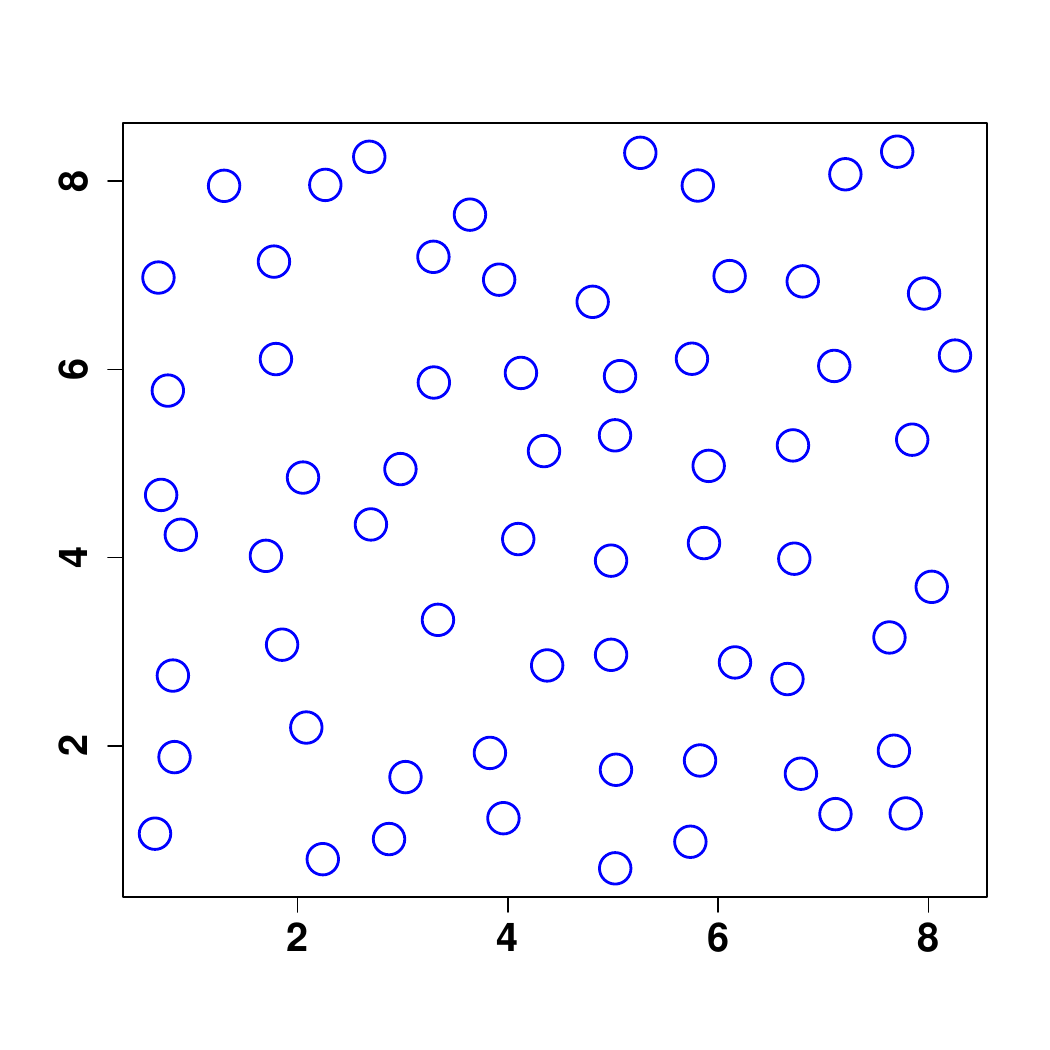}
\end{tabular}
\caption{Examples of three perturbed grids. The dimension is $d=2$ and the number of observation points is $n=8^2$.
From left to right, the values of the regularity parameter are $0$, $\frac{1}{8}$
and $\frac{3}{8}$. $\epsilon = 0$ corresponds to a regular
observation grid, while, when $\epsilon$ is close to $\frac{1}{2}$, the observation set is highly irregular.}
\label{fig: grilles_perturbees}
\end{figure}

Naturally, the parameterized irregularity that we consider here is governed by a single scalar parameter, and
thus cannot be representative of all the notions of spatial sampling irregularity that exist. We can already discuss two examples
for this. First, because the observation points are linked to a regular grid, there is a limit for the number of points
that can form a cluster with very small spacing. For instance in dimension 2, this limit number is $4$ (for $4$ points
stemming from perturbations of elementary squares of the form $(i,j),(i+1,j),(i,j+1),(i+1,j+1)$).
Second, still because the random sampling relies on a regular grid,
it remains rather regular at a global level, as one can observe in figure \ref{fig: grilles_perturbees}.
In \cite{Lahiri2003central} and \cite{Lahiri2004asymptotic}, an asymptotic framework is studied in which the observation points
are $iid$ realizations stemming from a common probability distribution, not necessarily uniform. This purely random sampling
solves the two limitations we have mentioned: the sampling is globally not uniform, and clusters of arbitrary numbers of points can appear.  
Nevertheless, the spatial sampling family we propose is simple to interpret, and is parameterized by a scalar regularity parameter, which enables us to compare the irregularities
of two samplings stemming from it with no ambiguity. Furthermore, the conclusions of section \ref{section: numerical_study} obtained from the perturbed regular grid
are of interest, since they confirm that using groups of observation points with small spacing is beneficial to estimation, compared to
evenly spaced points.

We denote $C_{S_X} = \left\{ t_1-t_2,t_1 \in S_X, t_2 \in S_X \right\}$, the set of all possible differences between two points
in $S_X$.
We denote, for $n \in \NN^*$, $X = \left(X_1,...,X_n\right)^t$, where we do not write explicitly the dependence on $n$ for clarity.
$X$ is a random variable with distribution ${\mathcal{L}}_X^{\otimes n}$.
We also denote $x = \left(x_1,...,x_n\right)^t$, an element of $\left(S_X\right)^n$, as a realization of $X$. 

We define the $n \times n$ random matrix $R_{\theta}$ by $\left(R_{\theta}\right)_{i,j} = K_{\theta} \left( v_i-v_j + \epsilon\left(X_i-X_j\right) \right)$. We
do not write explicitly the dependence of $R_{\theta}$ with respect to
$X$, $\epsilon$ and $n$. 
We shall denote, as a simplification, $R := R_{\theta_0}$.
We define the random vector $y$ of size $n$ by $y_i = Y\left(v_i + \epsilon X_i\right)$. We do not write explicitly the dependence of $y$ with respect to $X$, $\epsilon$ and $n$.

We denote as in \cite{TCMR}, for a real $n \times n$ matrix $A$, $|A|^2 = \frac{1}{n} \sum_{i,j=1}^n A_{i,j}^2$ and $||A||$ the largest singular value of $A$.
$|.|$ and $||.||$ are norms and $||.||$ is a matrix norm. We denote by $\phi_i\left(M\right)$, $1 \leq i \leq n$, the eigenvalues of a symmetric matrix
$M$. We denote, for two sequences of square matrices $A$ and $B$, depending on $n \in \NN^*$, $A \sim B$ if $|A-B| \to_{n \to + \infty} 0$ and $||A||$ and $||B||$
are bounded with respect to $n$. For a square matrix $A$, we denote by $\diag\left(A\right)$ the matrix obtained by setting to $0$ all non diagonal elements of $A$.

Finally, for a sequence of real random variables $z_n$, we denote $z_n \to_p 0$ and $z_n = o_p\left(1\right)$
when $z_n$ converges to zero in probability.

\subsection{ML and CV estimators}

We denote $L_{\theta} := \frac{1}{n} \left\{ \log\left(\det\left(R_{\theta}\right)\right) + y^t R_{\theta}^{-1} y \right\}$ the modified
opposite log-likelihood, where we
do not write explicitly the dependence on $X$, $Y$, $n$ and $\epsilon$. We denote by $\hat{\theta}_{ML}$ the Maximum Likelihood estimator, defined by
\begin{equation} \label{eq: thetaML}
 \hat{\theta}_{ML} \in \argmin_{\theta \in \Theta} L_{\theta},
\end{equation}
where we do not write explicitly the dependence of $\hat{\theta}_{ML}$ with respect to $X$, $Y$, $\epsilon$ and $n$.

\begin{rem}  \label{rem: definition_ML}
The ML estimator in \eqref{eq: thetaML} is actually not uniquely defined, since the likelihood function of \eqref{eq: thetaML}
can have more than one global
minimizer. Nevertheless, the convergence results of $\hat{\theta}_{ML}$, as $n \to + \infty$, hold when $\hat{\theta}_{ML}$ is any random variable
belonging to the set of the
global minimizers of the likelihood of \eqref{eq: thetaML}, regardless of the value chosen in this set. Furthermore, it can be shown that,
with probability converging to one,
as $n \to \infty$ (see the proof of proposition \ref{prop: normaliteEstimateur} in section \ref{section: appendix_technical_results}),
the likelihood function has a unique global minimizer. To define a measurable function
$\hat{\theta}_{ML}$ of $Y$ and $X$, belonging to the set of the minimizers of the likelihood, one possibility is the
following. For a given realization of $Y$ and $X$, let $\mathcal{K}$ be the set of the minimizers of the likelihood. Let
$\mathcal{K}_0 = \mathcal{K}$ and, for $0 \leq k \leq p-1$, $\mathcal{K}_{k+1}$ is the subset of $\mathcal{K}_k$
whose elements have their coordinate $k+1$ equal to $\min{ \left\{ \tilde{\theta}_{k+1},\tilde{\theta}  \in \mathcal{K}_k \right\} }$.
Since, $\mathcal{K}$ is compact
(because the likelihood function is continuous with respect to $\theta$ and defined on the compact set $\Theta$),
the set $\mathcal{K}_{p}$ is composed of a unique element, that we define as $\hat{\theta}_{ML}$, which is a measurable function of
$X$ and $Y$. The same remark can be made for the Cross Validation estimator of \eqref{eq: thetaCV}.
\end{rem}

When the increasing-domain asymptotic sequence of observation points is deterministic, it is shown
in \cite{MLEMRCSR} that $\hat{\theta}_{ML}$ converges to a centered Gaussian random vector. The asymptotic covariance matrix is the inverse of the
Fisher information matrix.
For fixed $n$, the Fisher information matrix is the $p \times p$
matrix with element $(i,j)$ equal to
$ \frac{1}{2} {\rm{Tr}}\left( R_{\theta_0}^{-1} \frac{\partial R_{\theta_0}}{\partial \theta_i} R_{\theta_0}^{-1} \frac{\partial R_{\theta_0}}{\partial \theta_j} \right)  $.
Since the literature has not addressed
yet the asymptotic distribution of $\hat{\theta}_{ML}$ in increasing-domain asymptotics with random observation points, we give complete proofs about it in section
\ref{section: appendix_proof_normality}. Our techniques are original and
not specifically oriented towards ML contrary to \cite{UANMLE,MLEMRCSR,ADREMLE}, so that they allow us to
address the asymptotic distribution of the CV estimator in the same fashion.

The CV estimator is defined by
\begin{equation}  \label{eq: thetaCV}
\hat{\theta}_{CV} \in \argmin_{\theta \in \Theta} \sum_{i=1}^n \{ y_{i} - \hat{y}_{i,\theta} \left(y_{-i}\right) \}^2,
\end{equation} 
where, for $1 \leq i \leq n$, $\hat{y}_{i,\theta} \left(y_{-i}\right)  := \EE_{\theta|X} \left( y_i | y_1,...,y_{i-1},y_{i+1},...,y_n \right)$ is the Kriging Leave-One-Out prediction
of $y_i$ with covariance parameters $\theta$. $\EE_{\theta|X}$ denotes the expectation with respect to the distribution of $Y$
with the covariance function $K_{\theta}$, given $X$.

The CV estimator selects the covariance parameters according to the criterion of the point-wise prediction errors. This criterion
does not involve the Kriging predictive variances.
Hence, the CV estimator of \eqref{eq: thetaCV} cannot estimate a covariance parameter impacting only on the variance of the Gaussian process.
Nevertheless, all the
classical parametric models $\{K_{\theta}, \theta \in \Theta\}$
satisfy the decomposition $ \theta = \left( \sigma^2 , \tilde{\theta}^t \right)^t $ and $\{K_{\theta}, \theta \in \Theta\} =
\{ \sigma^2 \tilde{K}_{\tilde{\theta}}, \sigma^2 > 0 , \tilde{\theta} \in \tilde{\Theta} \}$, with $\tilde{K}_{\tilde{\theta}}$ a correlation function.
Hence, in this case, $\tilde{\theta}$ would be estimated by \eqref{eq: thetaCV}, and $\sigma^2$ would be estimated by the equation
$\hat{\sigma}_{CV}^2 (\tilde{\theta}) = \frac{1}{n} \sum_{i=1}^n \frac{\{y_i - \hat{y}_{i,\tilde{\theta}} \left(y_{-i}\right) \}^2}{ c_{i,-i,\tilde{\theta}}^2}$,
where $c_{i,-i,\tilde{\theta}}^2 := \var_{\tilde{\theta}|X}\left(y_i|y_1,...,y_{i-1},y_{i+1},...,y_n\right)$ is the Kriging Leave-One-Out predictive variance for $y_i$
with covariance parameters $\sigma^2 = 1$ and $\tilde{\theta}$. $\var_{\tilde{\theta}|X}$ denotes the variance
with respect to the distribution of $Y$ with the covariance function
$K_{\theta}$, $\theta = (1,\tilde{\theta}^t)^t$, given $X$. To summarize, the general CV procedure we study is a two-step procedure. In a first step, the correlation parameters are
selected according to a mean square error criterion. In a second step, the global variance parameter is selected, so that the predictive variances are
adapted to the Leave-One-Out prediction errors. Here, we address the first step, so we focus on the CV estimator
defined in \eqref{eq: thetaCV}.

The criterion \eqref{eq: thetaCV} can be computed with a single matrix inversion, by means of
virtual LOO formulas (see e.g \cite[ch.5.2]{SS} for the zero-mean case addressed here, and \cite{CVKUN} for the universal Kriging case).
These virtual LOO formulas yield
\[
 \sum_{i=1}^n \left\{ y_{i} - \hat{y}_{i,\theta} \left(y_{-i}\right) \right\}^2 = y^t R_{\theta}^{-1} \diag\left(R_{\theta}^{-1}\right)^{-2} R_{\theta}^{-1} y,
\]
which is useful both in practice (to compute quickly the LOO errors) and in the proofs on CV. We then define
\[
CV_{\theta} := \frac{1}{n} y^t R_{\theta}^{-1} \diag\left(R_{\theta}^{-1}\right)^{-2} R_{\theta}^{-1} y
\]
as the CV criterion, where we do not write explicitly the
dependence on $X$, $n$, $Y$ and $\epsilon$. Hence we have, equivalently to $\eqref{eq: thetaCV}$,
$\hat{\theta}_{CV} \in \argmin_{\theta \in \Theta} CV_{\theta}$.

Since the asymptotic covariance matrix of the ML estimator is the inverse of the Fisher information matrix, this estimator should be used when $\theta_0 \in \Theta$ holds,
which is the case of interest here. However, in practice, it is likely that the true covariance function of the Gaussian process does not belong to the
parametric family used for the estimation. In \cite{Bachoc2013cross}, it is shown that CV
is more efficient than ML in this case, provided that the sampling is not too regular.
Hence, the CV estimator is relevant in practice, which is a reason for studying it in the case $\theta_0 \in \Theta$
addressed here.

In the sequel, we call the case $\theta_0 \in \Theta$ the well-specified case, and we call the case where the true covariance function of $Y$ does not
belong to $\{ K_{\theta}, \theta \in \Theta \}$ the misspecified case, or the case of model misspecification.

Hence, we aim at studying the influence of the spatial sampling on CV as well as on ML, in the well-specified case addressed here.
Furthermore, since it is expected that ML performs better than CV in the well-specified case, we are interested in quantifying this fact.

\section{Consistency and asymptotic normality}  \label{section: consistency_asymptotic_normality}

Proposition \ref{prop: consistance_ML} addresses the consistency of the ML estimator. The only assumption on the parametric family of covariance functions
is an identifiability assumption. Basically, for a fixed $\epsilon$, there should not exist two distinct covariance parameters so that the two associated covariance functions
are the same, on the set of inter-point distances covered by the random spatial sampling. The identifiability assumption is clearly minimal.

\begin{prop} \label{prop: consistance_ML}
Assume that condition \ref{cond: Ktheta} is satisfied.

For $\epsilon = 0$, if there does not exist $\theta \neq \theta_0$ so that
$K_{\theta}\left(v\right) = K_{\theta_0}\left(v\right)$ for all $v \in \ZZ^d$, then the ML estimator is consistent.

For $\epsilon \neq 0$, we denote $D_{\epsilon} = \cup_{v \in \ZZ^d \backslash 0} \left(v + \epsilon C_{S_X}\right)$,
with $C_{S_X} = \left\{ t_1-t_2,t_1 \in S_X, t_2 \in S_X \right\}$.
Then, if there does not exist $\theta \neq \theta_0$ so that
$K_{\theta} = K_{\theta_0}$ a.s. on $D_{\epsilon}$, according to the Lebesgue measure on $D_{\epsilon}$, and
$K_{\theta}\left(0\right) = K_{\theta_0}\left(0\right)$, then the
ML estimator is consistent.
 \end{prop}

In proposition \ref{prop: normaliteML}, we address the asymptotic normality of ML. The convergence rate is $\sqrt{n}$, as in a classical $iid$ framework, and we
prove the existence of a deterministic asymptotic covariance matrix of $\sqrt{n} \hat{\theta}_{ML}$ which depends only on the regularity parameter $\epsilon$.
In proposition \ref{prop: info>0_ML}, we prove that this asymptotic covariance matrix is positive, as long as the different derivative functions
with respect to $\theta$ at $\theta_0$ of the covariance function are non redundant on the set of inter-point distances covered by the random spatial sampling.
This condition is minimal, since when these derivatives are redundant, the Fisher information matrix is singular for all finite sample-size $n$ and its kernel is
independent of $n$. 

\begin{prop} \label{prop: normaliteML}
Assume that condition \ref{cond: Ktheta} is satisfied.

For $1 \leq i,j \leq p$, define the sequence of $n \times n$ matrices, indexed by $n \in \NN^*$,
$M_{ML}^{i,j}$ by,
$M_{ML}^{i,j} := \frac{1}{2} R^{-1} \frac{\partial R}{\partial \theta_i} R^{-1} \frac{\partial R}{\partial \theta_j}$.
Then, the random trace $\frac{1}{n} {\rm{Tr}}\left( M_{ML}^{i,j} \right)$
has an almost sure limit as $n \to +\infty$.

We thus define the $p \times p$ deterministic matrix $\Sigma_{ML}$ so that $\left(\Sigma_{ML}\right)_{i,j}$
is the almost sure limit, as $n \to + \infty$, of
$\frac{1}{n} {\rm{Tr}}\left( M_{ML}^{i,j} \right)$.

Then, if $\hat{\theta}_{ML}$ is consistent and if $\Sigma_{ML}$ is positive, then
\[
 \sqrt{n} \left( \hat{\theta}_{ML} - \theta_0 \right) \to_{\L} \N\left(0,\Sigma_{ML}^{-1}\right).
\]

\end{prop}

\begin{prop} \label{prop: info>0_ML}
Assume that condition \ref{cond: Ktheta} is satisfied.

For $\epsilon = 0$, if there does not exist $v_{\lambda} = \left(\lambda_1,...,\lambda_p\right)^t \in \RR^p$,
$v_{\lambda}$ different from zero, so that
$\sum_{k=1}^p \lambda_k \frac{\partial}{\partial \theta_k} K_{\theta_0}\left(v\right) = 0$
for all $v \in \ZZ^d$, then $\Sigma_{ML}$ is positive.

For $\epsilon \neq 0$, we denote $D_{\epsilon} = \cup_{v \in \ZZ^d \backslash 0} \left(v + \epsilon C_{S_X}\right) $,
with $C_{S_X} = \left\{ t_1-t_2,t_1 \in S_X, t_2 \in S_X \right\}$.
If there does not exist $v_{\lambda} = \left(\lambda_1,...,\lambda_p\right)^t \in \RR^p$,
$v_{\lambda}$ different from zero, so that
$\sum_{k=1}^p \lambda_k \frac{\partial}{\partial \theta_k} K_{\theta_0}\left(t\right):
\RR^d \to \RR$
is almost surely zero on $D_{\epsilon}$, with respect to the Lebesgue measure on $D_{\epsilon}$,
and that $\sum_{k=1}^p \lambda_k \frac{\partial}{\partial \theta_k} K_{\theta_0}\left(0\right)$
is null, then $\Sigma_{ML}$ is positive.
\end{prop}

\begin{rem}  \label{rem: Lahiri}
The asymptotic distribution of $M$-estimators of the regression parameters in a linear regression framework
with dependent errors, with a stochastic sampling, has been addressed in \cite{Lahiri2004asymptotic}.
In this reference, the observation errors stem from a general random field, and increasing-domain asymptotics is addressed, as well
as mixed increasing-domain asymptotics, in which both the observation domain and the density of the observation points go to infinity.
In \cite{Lahiri2004asymptotic}, the rate of convergence of the estimation of the regression parameters is not the square root of the
number of observation points anymore, but the square root of the volume of the observation domain.
This is the rate we also find in proposition \ref{prop: normaliteML} for the covariance parameters, in the particular case of increasing-domain asymptotics, where
the number of observation points grows like the volume of the observation domain.
It is nevertheless possible that this analogy would not happen in the mixed domain-asymptotic framework, since it is known that
the estimation of microergodic covariance parameters can have a $\sqrt{n}$ rate of convergence even in the fixed-domain asymptotic framework \citep{APMLEDGP}.

Notice also that, similarly to us, it is explicit in \cite{Lahiri2004asymptotic} that the asymptotic covariance matrix of the estimators
depends on the distribution of the random observation points.
\end{rem}

Proposition \ref{prop: consistance_CV} addresses the consistency of the CV estimator.
The identifiability assumption is required, like in the ML case. Since the CV estimator is designed for estimating correlation parameters, we assume
that the parametric model $\{K_{\theta} , \theta \in \Theta\}$ contains only correlation functions. This assumption holds in most classical cases, and
yields results that are easy to express and interpret. The case of hybrid covariance parameters, specifying both a variance and a correlation structure should
be consistently estimated by the CV estimator. Nevertheless, since such covariance parameters
do not exist in the classical families of covariance functions, we do not address this case.

\begin{prop} \label{prop: consistance_CV}

Assume that condition \ref{cond: Ktheta} is satisfied and that for all $\theta \in \Theta$, $K_{\theta}\left(0\right) = 1$.

For $\epsilon = 0$, if there does not exist $\theta \neq \theta_0$ so that
$K_{\theta}\left(v\right) = K_{\theta_0}\left(v\right)$ for all $v \in \ZZ^d$, then the CV estimator is consistent.

For $\epsilon \neq 0$, we denote $D_{\epsilon} = \cup_{v \in \ZZ^d \backslash 0} \left(v + \epsilon C_{S_X}\right)$,
with $C_{S_X} = \left\{ t_1-t_2,t_1 \in S_X, t_2 \in S_X \right\}$.
Then, if their does not exist $\theta \neq \theta_0$ so that
$K_{\theta} = K_{\theta_0}$ a.s. on $D_{\epsilon}$, with respect to the Lebesgue measure on $D_{\epsilon}$, the CV estimator is consistent.
\end{prop}

Proposition \ref{prop: gradients_CV} gives the expression of the covariance matrix of the gradient of the CV criterion $CV_{\theta}$
and of the mean matrix of its Hessian.
These moments are classically used in statistics to prove asymptotic distributions of consistent estimators. We also prove the convergence of these moments,
the limit matrices being functions of the $p \times p$ matrices $\Sigma_{CV,1}$ and $\Sigma_{CV,2}$, for which we prove the existence. These matrices are deterministic and
depend only on the regularity parameter $\epsilon$.

\begin{prop} \label{prop: gradients_CV}

Assume that condition \ref{cond: Ktheta} is satisfied.

With, for $1 \leq i \leq p$,
\[
 M^i_{\theta}= R_{\theta}^{-1} \diag\left(R_{\theta}^{-1}\right)^{-2} \left\{ \diag\left( R_{\theta}^{-1} \frac{\partial R_{\theta}}{\partial \theta_i} R_{\theta}^{-1} \right)
\diag\left( R_{\theta}^{-1} \right)^{-1} - R_{\theta}^{-1} \frac{\partial R_{\theta}}{\partial \theta_i} \right\} R_{\theta}^{-1},
\]
we have, for all $1 \leq i,j \leq p$,
 \[
  \frac{\partial}{ \partial \theta_i } CV_{\theta} =
\frac{1}{n} 2 y^t M^i_{\theta} y.
 \]
We define the sequence of $n \times n$
matrices, indexed by $n \in \NN^*$,
$M_{CV,1}^{i,j}$ by
\[
M_{CV,1}^{i,j} = 2 \left\{M^i_{\theta_0}+\left(M^i_{\theta_0}\right)^t\right\}R_{\theta_0} \left\{M^j_{\theta_0}+\left(M^j_{\theta_0}\right)^t\right\}R_{\theta_0}.
\]
Then,
\begin{eqnarray} \label{eq: cov_d1_CV}
 \cov \left( \sqrt{n} \frac{\partial}{ \partial \theta_i } CV_{\theta_0} , \sqrt{n} \frac{\partial}{ \partial \theta_j } CV_{\theta_0} | X \right) = \frac{1}{n}
{\rm{Tr}}\left( M_{CV,1}^{i,j} \right).
\end{eqnarray}
Furthermore, the random trace $\frac{1}{n}
{\rm{Tr}}\left( M_{CV,1}^{i,j} \right)$ converges a.s. to the element $\left(\Sigma_{CV,1}\right)_{i,j}$ of a $p \times p$
deterministic matrix $\Sigma_{CV,1}$ as $n \to +\infty$.

Defining
\begin{eqnarray*}  
 M_{CV,2}^{i,j} 
& = & -8  \diag\left(R_{\theta_0}^{-1}\right)^{-3} \diag\left( R_{\theta_0}^{-1} \frac{\partial R_{\theta_0}}{\partial \theta_i} R_{\theta_0}^{-1}\right) 
R_{\theta_0}^{-1} \frac{\partial R_{\theta_0}}{\partial \theta_j} R_{\theta_0}^{-1}   \\
& & +2  \diag\left(R_{\theta_0}^{-1}\right)^{-2}  R_{\theta_0}^{-1} \frac{\partial R_{\theta_0}}{\partial \theta_i} R_{\theta_0}^{-1}
\frac{\partial R_{\theta_0}}{\partial \theta_j} R_{\theta_0}^{-1} \\ 
& & +6  \diag\left(R_{\theta_0}^{-1}\right)^{-4}  \diag\left( R_{\theta_0}^{-1} \frac{\partial R_{\theta_0}}{\partial \theta_i} R_{\theta_0}^{-1}\right)
\diag\left( R_{\theta_0}^{-1} \frac{\partial R_{\theta_0}}{\partial \theta_j} R_{\theta_0}^{-1}\right) R_{\theta_0}^{-1},  
\end{eqnarray*}
we also have
\begin{equation}  \label{eq: esp_d2_CV}
 \EE\left( \frac{\partial^2}{ \partial \theta_i \partial \theta_j } CV_{\theta_0} | X \right) 
 = \frac{1}{n}
{\rm{Tr}}\left( M_{CV,2}^{i,j} \right).
\end{equation}

Furthermore, the random trace $\frac{1}{n}
{\rm{Tr}}\left( M_{CV,2}^{i,j} \right)$ converges a.s. to the element $\left(\Sigma_{CV,2}\right)_{i,j}$ of a $p \times p$
deterministic matrix $\Sigma_{CV,2}$ as $n \to +\infty$.

\end{prop}

In proposition \ref{prop: normaliteCV}, we address the asymptotic normality of CV. The conditions are, as for the consistency, identifiability and
that the set of covariance functions contains only correlation functions.
The convergence rate is also $\sqrt{n}$, and we have the expression of
the deterministic asymptotic covariance matrix of $\sqrt{n} \hat{\theta}_{CV}$, depending only of the matrices $\Sigma_{CV,1}$ and $\Sigma_{CV,2}$
of proposition \ref{prop: gradients_CV}.
In proposition \ref{prop: info>0_CV}, we prove that the asymptotic matrix $\Sigma_{CV,2}$ is positive. The minimal assumption is, as for the ML case, that
the different derivative functions
with respect to $\theta$ at $\theta_0$ of the covariance function are non redundant on the set of inter-point distances covered by the random spatial sampling.

\begin{prop} \label{prop: normaliteCV}

Assume that condition \ref{cond: Ktheta} is satisfied.

 If $\hat{\theta}_{CV}$ is consistent and if $\Sigma_{CV,2}$ is positive, then
\[
 \sqrt{n} \left( \hat{\theta}_{CV} - \theta_0 \right) \to_{\L} \N\left(0 ,  \Sigma_{CV,2}^{-1} \Sigma_{CV,1} \Sigma_{CV,2}^{-1} \right) ~ ~ \mbox{as $n \to +\infty$}.
\]

\end{prop}

\begin{prop} \label{prop: info>0_CV}
Assume that condition \ref{cond: Ktheta} is satisfied and that for all $\theta \in \Theta$, $K_{\theta}\left(0\right) = 1$.

For $\epsilon = 0$, if there does not exist $v_{\lambda} = \left(\lambda_1,...,\lambda_p\right)^t \in \RR^p$,
$v_{\lambda}$ different from zero, so that
$\sum_{k=1}^p \lambda_k \frac{\partial}{\partial \theta_k} K_{\theta_0}\left(v\right) = 0$
for all $v \in \ZZ^d$, then $\Sigma_{CV,2}$ is positive.

For $\epsilon \neq 0$, we denote $D_{\epsilon} = \cup_{v \in \ZZ^d \backslash 0} \left(v + \epsilon C_{S_X}\right)$,
with $C_{S_X} = \left\{ t_1-t_2,t_1 \in S_X, t_2 \in S_X \right\}$.
If there does not exist $v_{\lambda} = \left(\lambda_1,...,\lambda_p\right)^t \in \RR^p$,
$v_{\lambda}$ different from zero, so that
$\sum_{k=1}^p \lambda_k \frac{\partial}{\partial \theta_k} K_{\theta_0}\left(t\right): \RR^d \to \RR$
is almost surely zero on $D_{\epsilon}$, with respect to the Lebesgue measure on $D_{\epsilon}$, then $\Sigma_{CV,2}$ is positive.
\end{prop}

The conclusion for ML and CV is that, for all the most classical parametric families of covariance functions,
consistency and asymptotic normality hold, with deterministic positive
asymptotic covariance matrices depending only on the regularity parameter $\epsilon$. Therefore, these covariance matrices are analyzed in
section \ref{section: numerical_study}, to address the influence of the irregularity of the spatial sampling on the ML and CV estimation.

\FloatBarrier

\section{Analysis of the asymptotic covariance matrices}   \label{section: numerical_study}

The limit distributions of the ML and CV estimators
only depend on the regularity parameter $\epsilon$ through the asymptotic covariance matrices in propositions \ref{prop: normaliteML} and
\ref{prop: normaliteCV}. The aim of this section is to numerically study the influence of $\epsilon$ on these asymptotic covariance matrices.
We address the case $d=1$, with $p=1$ in subsections \ref{subsection: numeric_local} and \ref{subsection: numeric_global}
and $p=2$ in subsection \ref{subsection: joint}.

The asymptotic covariance matrices of propositions \ref{prop: normaliteML} and \ref{prop: normaliteCV}
are expressed as functions of a.s. limits of traces of sums, products and inverses of random matrices. In the case $\epsilon=0$, for $d=1$,
these matrices are deterministic Toeplitz matrices, so that the limits can be expressed using Fourier transform techniques (see \cite{TCMR}). In section
\ref{section: Toeplitz}, we give the closed form  expression of $\Sigma_{ML}$, $\Sigma_{CV,1}$ and $\Sigma_{CV,2}$ for $\epsilon = 0$ and $d=1$.
In the case $\epsilon \neq 0$, there does not exist, to the best of our knowledge, any random matrix technique that would give a closed form expression of
$\Sigma_{ML}$, $\Sigma_{CV,1}$ and $\Sigma_{CV,2}$. Therefore, for the numerical study with $\epsilon \neq 0$, these matrices will be approximated by the random traces
for large $n$.

\subsection{The derivatives of $\Sigma_{ML}$, $\Sigma_{CV,1}$ and $\Sigma_{CV,2}$}  \label{subsection: echange_limite_derivee}

In proposition \ref{prop: deriveesSigma} we show that, under the mild condition \ref{cond: ddt_Ktheta}, the asymptotic covariance matrices obtained from
$\Sigma_{ML}$, $\Sigma_{CV,1}$ and $\Sigma_{CV,2}$,
of propositions \ref{prop: normaliteML} and \ref{prop: normaliteCV},
are twice differentiable with respect to $\epsilon$. This result is useful for the numerical study of subsection \ref{subsection: numeric_local}.

\begin{cond}  \label{cond: ddt_Ktheta}
\begin{itemize}
\item Condition \ref{cond: Ktheta} is satisfied. 
\item $K_{\theta}(t)$ and $\frac{\partial}{ \partial \theta_i } K_{\theta}\left(t\right)$, for $1 \leq i\leq p$,
are three times differentiable in $t$ for $t \neq 0$.
\item For all $T >0$, $\theta \in \Theta$, $1 \leq i \leq p$, $k \in \left\{ 1,2,3 \right\}$, $i_1,...,i_k \in \{1,...,d \}^k$ ,
there exists $C_{T} < + \infty$ so that for $|t| \geq T$,
\begin{eqnarray} \label{eq: controleDeriveesK}
\left| \frac{\partial}{ \partial t_{i_1} },...,\frac{\partial}{ \partial t_{i_k} } K_{\theta}\left(t\right) \right| & \leq & \frac{C_{T}}{1 + |t|^{d+1}},  \\ \nonumber
 \left| \frac{\partial}{ \partial t_{i_1} },...,\frac{\partial}{ \partial t_{i_k} } \frac{\partial}{ \partial \theta_i } K_{\theta}\left(t\right) \right|
& \leq & \frac{C_{T}}{1 + |t|^{d+1}}.  \\ \nonumber
\end{eqnarray}
\end{itemize}
\end{cond}

\begin{prop}  \label{prop: deriveesSigma}

Assume that condition \ref{cond: ddt_Ktheta} is satisfied.

Let us fix $1 \leq i,j \leq p$. The elements $\left(\Sigma_{ML}\right)_{i,j}$, $\left(\Sigma_{CV,1}\right)_{i,j}$ and $\left(\Sigma_{CV,2}\right)_{i,j}$
(as defined in propositions \ref{prop: normaliteML} and \ref{prop: gradients_CV}) are
$C^2$ in $\epsilon$ on $[ 0, \frac{1}{2} )$. Furthermore, with $\frac{1}{n}  {\rm{Tr}}\left( M_{ML}^{i,j} \right) \to_{a.s.} \left(\Sigma_{ML}\right)_{i,j}$,
$\frac{1}{n}  {\rm{Tr}}\left( M_{CV,1}^{i,j} \right) \to_{a.s.} \left(\Sigma_{CV,1}\right)_{i,j}$ and
$\frac{1}{n}  {\rm{Tr}}\left( M_{CV,2}^{i,j} \right) \to_{a.s.} \left(\Sigma_{CV,2}\right)_{i,j}$
(propositions \ref{prop: normaliteML} and \ref{prop: gradients_CV}), we have,
for $\left(\Sigma\right)_{i,j}$ being $\left(\Sigma_{ML}\right)_{i,j}$, $\left(\Sigma_{CV,1}\right)_{i,j}$ or $\left(\Sigma_{CV,2}\right)_{i,j}$ and $M^{i,j}$ being $M_{ML}^{i,j}$, $M_{CV,1}^{i,j}$ or $M_{CV,2}^{i,j}$,
\[
 \frac{\partial^2}{\partial \epsilon^2} \left(\Sigma\right)_{i,j} = \underset{n \to + \infty}{\lim} \frac{1}{n} \EE\left\{ \frac{\partial^2}{\partial \epsilon^2} {\rm{Tr}}\left(M^{i,j}\right) \right\}.
\]
\end{prop}

Proposition \ref{prop: deriveesSigma} shows that we can compute numerically the derivatives of $\left(\Sigma_{ML}\right)_{i,j}$, $\left(\Sigma_{CV,k}\right)_{i,j}$, $k=1,2$,
with respect to $\epsilon$ by computing the derivatives of $M_{ML}^{i,j}$, $M_{CV,k}^{i,j}$, $k=1,2$, for $n$ large. The fact that it is possible to
exchange the limit in $n$ and the derivative in $\epsilon$ was not {\it a priori} obvious.

In the rest of the section, we address specifically the case where $d=1$, and the distribution of the $X_i$, $1 \leq i \leq n$, is uniform on $[-1,1]$.
We focus on
the case of the Mat\'ern covariance function. In dimension one, this covariance model is parameterized by the
correlation length
$\ell$ and the smoothness parameter $\nu$. The covariance function $K_{\ell,\nu}$ is Mat\'ern $\left(\ell,\nu\right)$ where
\begin{equation} \label{eq: Rmat}
K_{\ell,\nu}\left(t\right) = \frac{1}{\Gamma\left(\nu\right)2^{\nu-1}} \left(  2 \sqrt{\nu} \frac{|t|}{\ell} \right)^{\nu} K_{\nu} \left( 2 \sqrt{\nu} \frac{|t|}{\ell} \right),
\end{equation}
with $\Gamma$ the Gamma function and $K_{\nu}$ the modified Bessel function of second order.
See e.g \cite[p.31]{ISDSTK} for a presentation of the Mat\'ern correlation function.

In subsections \ref{subsection: numeric_local} and \ref{subsection: numeric_global}, we estimate $\nu$
when $\ell$ is known and conversely. This case $p=1$ enables us to deal with a scalar asymptotic variance, which
is an unambiguous criterion for addressing the impact of the irregularity of the spatial sampling.
In subsection \ref{subsection: joint}, we address the case $p=2$, where $\ell$ and $\nu$ are jointly estimated.

\subsection{Small random perturbations of the regular grid} \label{subsection: numeric_local}

In our study, the two true covariance parameters $\ell_0,\nu_0$ vary over $0.3 \leq \ell_0 \leq 3$ and $0.5 \leq \nu_0 \leq 5$.
We will successively address the two cases where $\ell$ is estimated and $\nu$ is known, and where $\nu$ is estimated and $\ell$ is known.
It is shown in section \ref{subsection: echange_limite_derivee} that for both ML and CV, the asymptotic variances are smooth functions of $\epsilon$.
The quantity of interest is the ratio of the second derivative with respect to $\epsilon$ at $\epsilon=0$ of the
asymptotic variance over its value at $\epsilon=0$.
When this quantity is negative, this means that the asymptotic variance of the covariance parameter estimator decreases with $\epsilon$, and therefore
that a small perturbation of the regular grid improves the estimation.
The second derivative is calculated exactly for ML, using the results of section \ref{section: Toeplitz},
and is approximated by finite differences for $n$ large for CV. Proposition \ref{prop: deriveesSigma} ensures that this approximation is numerically
consistent (because the limits in $n$ and the derivatives in $\epsilon$ are exchangeable).

On figure \ref{fig: d2surVal0_hatLc}, we show the numerical results for the estimation of $\ell$. First we see that the relative improvement of the estimation
due to the perturbation
is maximum when the true correlation length $\ell_0$ is small. Indeed, the inter-observation distance being $1$, a correlation length of approximatively $0.3$ means that
the observations are almost independent, making the estimation of the covariance very hard. Hence, the perturbations of the grid create pairs of observations
that are less independent and subsequently facilitate the estimation. For large $\ell_0$, this phenomenon does not take place anymore, and thus the relative effect of the
perturbations is smaller. Second, we observe that for ML the small perturbations are always an advantage for the estimation of $\ell$.
This is not the case for CV, where the asymptotic variance can increase with $\epsilon$. Finally, we can see that the two particular points
$\left(\ell_0=0.5,\nu_0=5\right)$ and $\left(\ell_0=2.7,\nu_0=1\right)$ are particularly interesting and representative. Indeed, $\left(\ell_0=0.5,\nu_0=5\right)$ corresponds to
covariance parameters for which the small perturbations of the regular grid have a strong and favorable impact on the estimation for ML and CV,
while $\left(\ell_0=2.7,\nu_0=1\right)$ corresponds to
covariance parameters for which they have an unfavorable impact on the estimation for CV.
We retain these two points for a further global investigation for $0 \leq \epsilon \leq 0.45$
in subsection \ref{subsection: numeric_global}.

\begin{figure}[]
\centering
 \hspace*{-2cm}

\begin{tabular}{c c}
\includegraphics[width=8cm,angle=0]{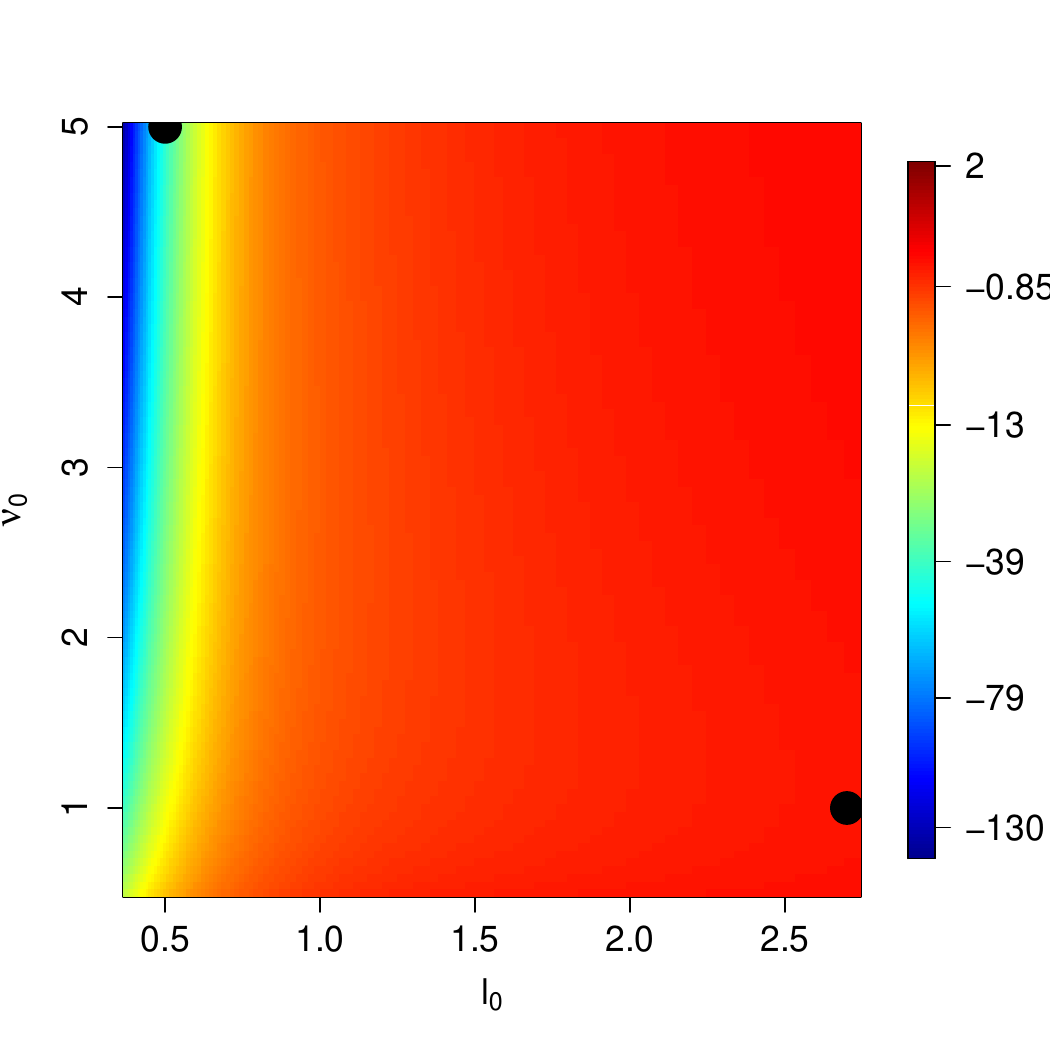} &
\includegraphics[width=8cm,angle=0]{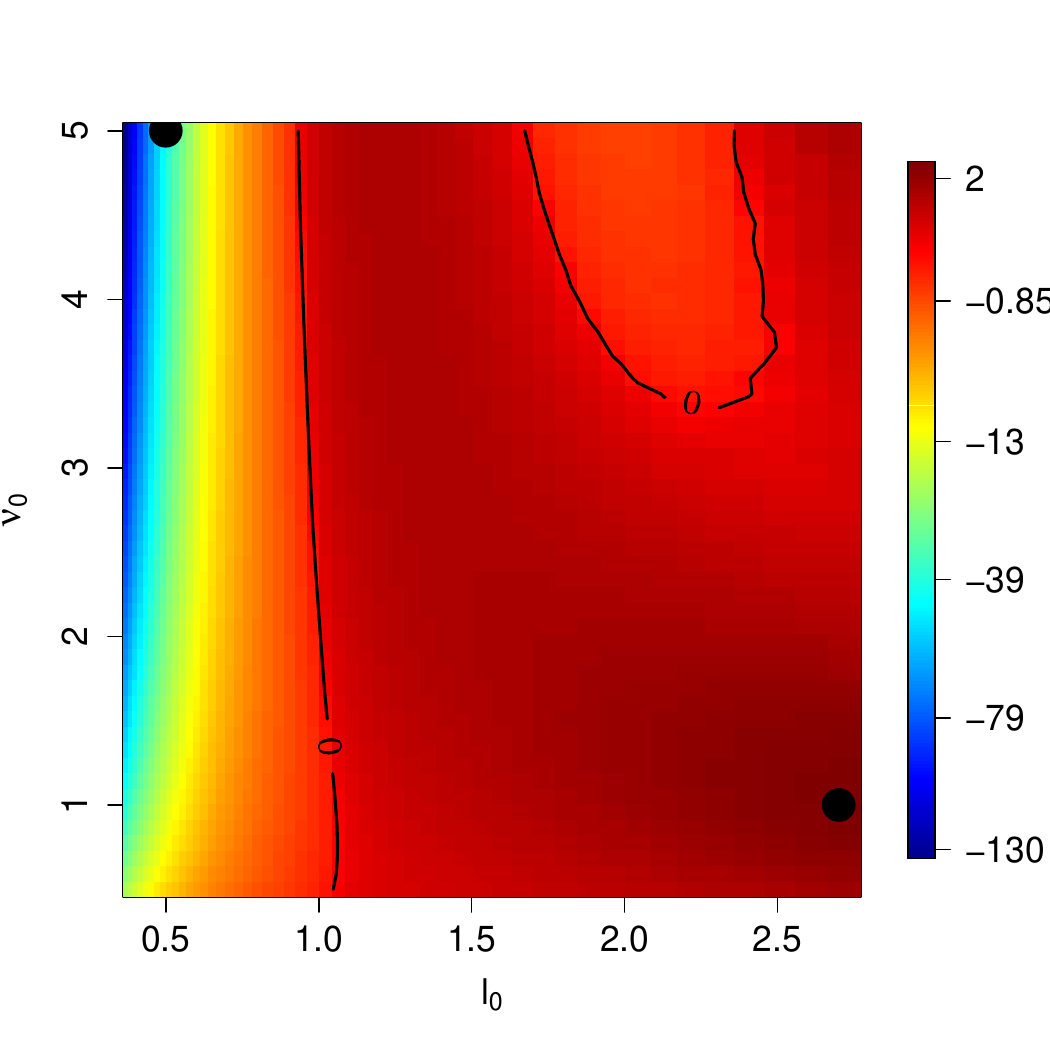}
\end{tabular}
\caption{Estimation of $\ell$. Plot of the ratio of the second derivative of the
asymptotic variance over its value at $\epsilon=0$, for ML (left) and CV (right).
The true covariance function is Mat\'ern with varying $\ell_0$ and $\nu_0$. We retain the two particular points
$(\ell_0=0.5,\nu_0=5)$ and $(\ell_0=2.7,\nu_0=1)$ for further investigation in subsection \ref{subsection: numeric_global}
(these are the black dots).}
\label{fig: d2surVal0_hatLc}
\end{figure}

On figure \ref{fig: d2surVal0_hatNu}, we show the numerical results for the estimation of $\nu$. We observe
that for $\ell_0$ relatively small, the asymptotic variance is an increasing function of $\epsilon$ (for small $\epsilon$).
This happens
approximatively in the band $0.4 \leq \ell_0 \leq 0.6$, and for both ML and CV. There is a plausible explanation
from this fact, which is not easy to interpret at first sight. It can be seen that for $\ell \approx 0.73$,
the value of the one-dimensional Mat\'ern covariance function at $t=1$ is almost independent of $\nu$ for $\nu \in [1,5]$. As an illustration,
for $\nu=2.5$, the derivative of this value with respect to $\nu$ is $-3.7 \times 10^{-5}$ for a value of $0.15$.
When $0.4 \leq \ell_0 \leq 0.6$, $\ell_0$ is small so that most of the information for estimating $\nu$
is obtained from the pairs of successive observation points. Perturbing the regular grid creates pairs of successive observation
points $i + \epsilon x_i, i + 1 + \epsilon x_{i+1}$ verifying $\frac{|1 + \epsilon (x_{i+1} - x_i )  |}{\ell_0} \approx \frac{1}{0.73}$, so that the
correlation of the two observations becomes almost independent of $\nu$. Thus, due to a specificity of the Mat\'ern covariance function,
decreasing the distance between two successive observation points unintuitively removes information on $\nu$.

For $0.6 \leq \ell_0 \leq 0.8$ and 
$\nu_0 \geq 2$, the relative improvement is maximum. This is explained the same way as above, this time the case $\epsilon=0$
yields successive observation points for which the correlation is independent of $\nu$, and increasing $\epsilon$
changes the distance between two successive observation points, making the correlation of the observations dependent of $\nu$.

In the case $\ell_0 \geq 0.8$, there is no more impact of the specificity of the case $\ell_0 \approx 0.73$ and
the improvement of the estimation when $\epsilon$ increases remains significant, though smaller. 
Finally, we see the three particular points
$\left(\ell_0=0.5,\nu_0=2.5\right)$, $\left(\ell_0=0.7,\nu_0=2.5\right)$ and $\left(\ell_0=2.7,\nu_0=2.5\right)$ as representative
of the discussion above, and we retain them for further global investigation
for $0 \leq \epsilon \leq 0.45$ in subsection \ref{subsection: numeric_global}.

\begin{figure}[]
\centering
 \hspace*{-2cm}

\begin{tabular}{c c}
\includegraphics[width=8cm,angle=0]{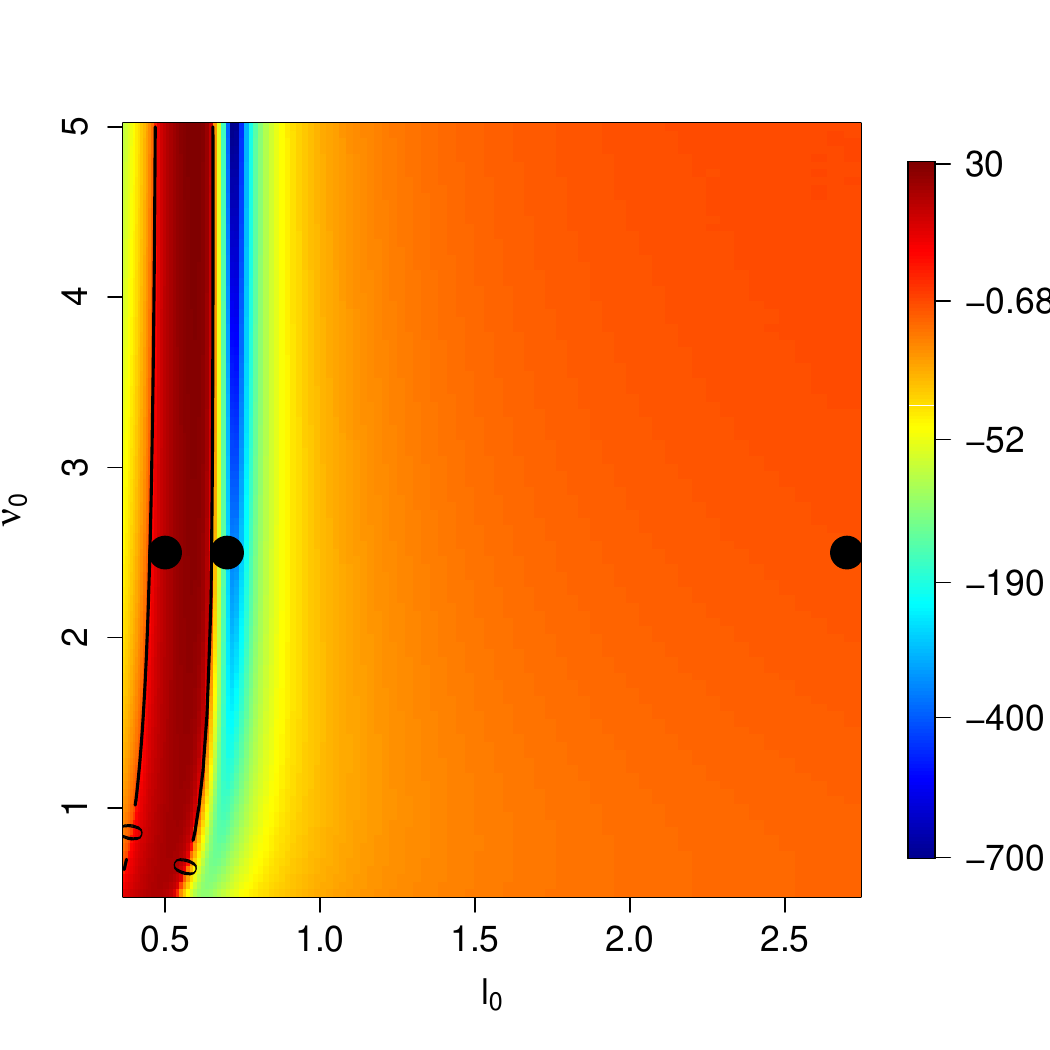} &
\includegraphics[width=8cm,angle=0]{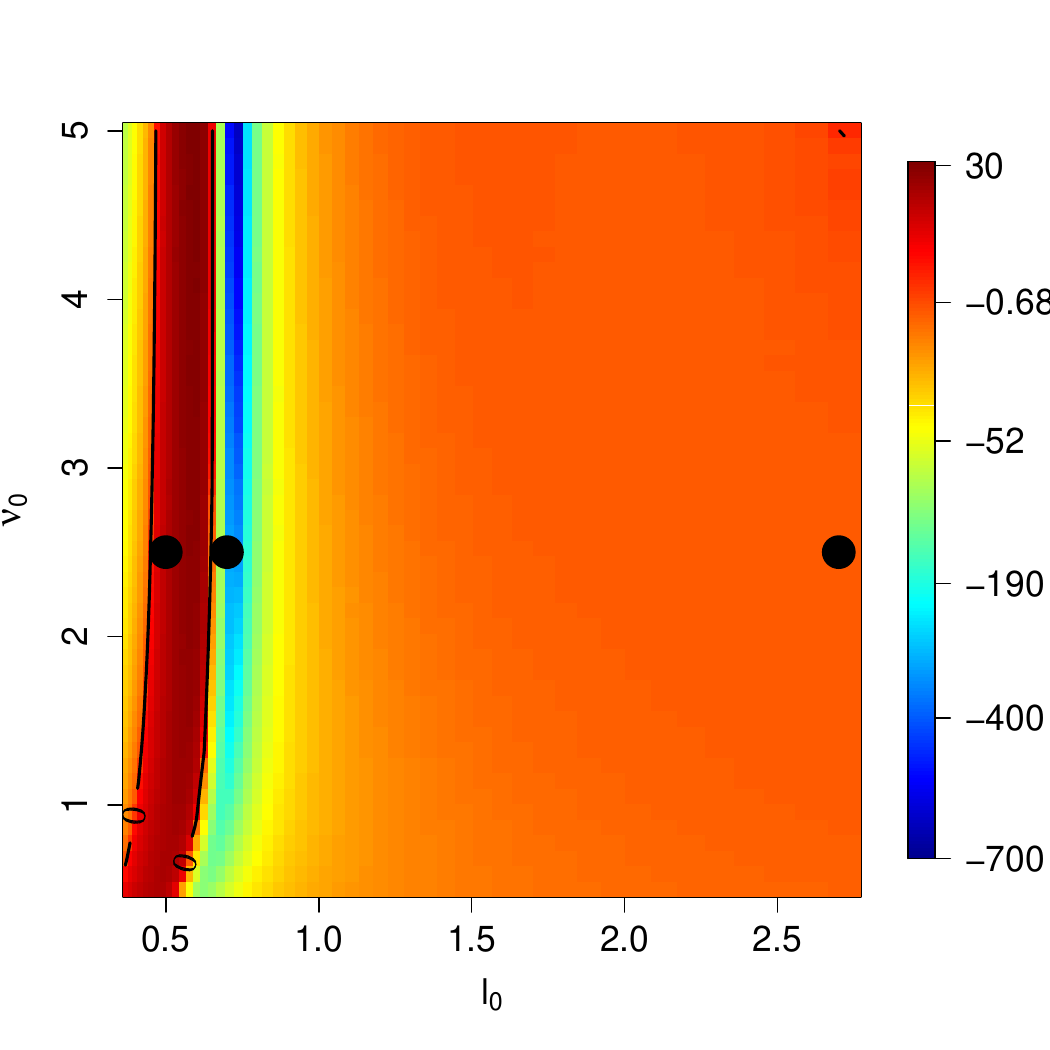}
\end{tabular}
\caption{Same setting as figure \ref{fig: d2surVal0_hatLc}, but for the estimation of $\nu$. We retain the three particular points
$\left(\ell_0=0.5,\nu_0=2.5\right)$, $\left(\ell_0=0.7,\nu_0=2.5\right)$ and $\left(\ell_0=2.7,\nu_0=2.5\right)$ for further
investigation in subsection \ref{subsection: numeric_global}.}
\label{fig: d2surVal0_hatNu}
\end{figure}

\subsection{Large random perturbations of the regular grid} \label{subsection: numeric_global}

On figures \ref{fig: Var0surVar45_hatLc} and \ref{fig: Var0surVar45_hatNu},
we plot the ratio of the
asymptotic variance for $\epsilon=0$ over the asymptotic variance for $\epsilon=0.45$,
with varying $\ell_0$ and $\nu_0$, for ML and CV
and in the two cases where $\ell$ is estimated and $\nu$ known and conversely.
We observe that this ratio is always larger than one for ML, that is strong perturbations of the regular
grid are always beneficial to ML estimation. This is the most important numerical conclusion of this section \ref{section: numerical_study}.
As ML is the preferable method to use in the well-specified case addressed here, we reformulate this conclusion by saying that,
in our experiments, using pairs of closely spaced observation points is always beneficial for covariance parameter estimation compared to evenly spaced
observation points. This is an important practical conclusion, that is in agreement with the references \cite{ISDSTK}
and \cite{SSDUIAF} discussed in section \ref{section: introduction}.

For CV, on the contrary, we exhibit cases for which strong perturbations of the regular grid decrease the accuracy of the estimation,
particularly for the estimation of $\ell$.
This can be due to the fact
that the Leave-One-Out errors in the CV functional \eqref{eq: thetaCV} are unnormalized. Hence, when the regular grid is perturbed, roughly speaking, error terms concerning
observation points with close neighbors are small, while error terms concerning observation points without close neighbors are large.
Hence, the CV functional mainly depends on the large error terms and hence has a larger variance.
This increases the variance of the CV estimator minimizing it.

\begin{figure}[]
\centering
 \hspace*{-2cm}

\begin{tabular}{c c}
\includegraphics[width=8cm,angle=0]{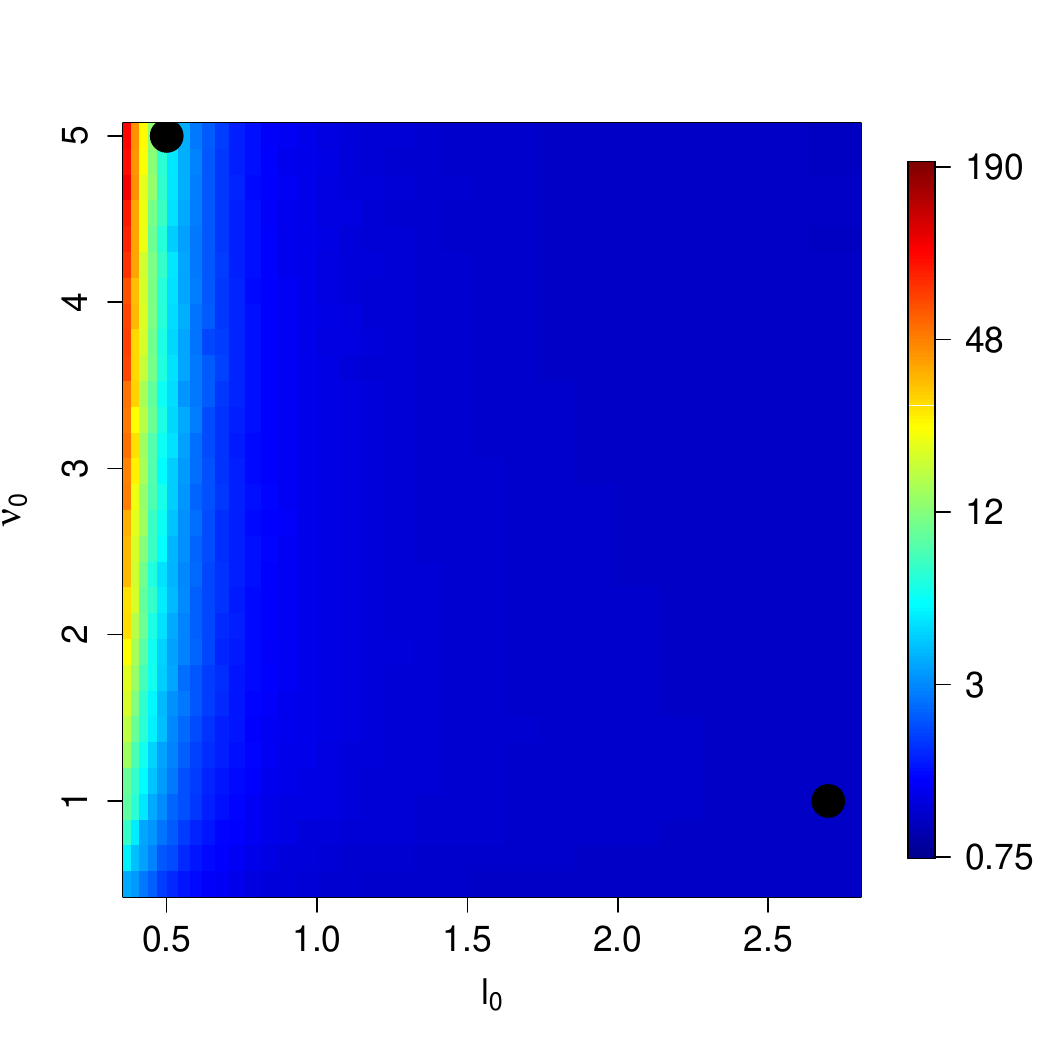} &
\includegraphics[width=8cm,angle=0]{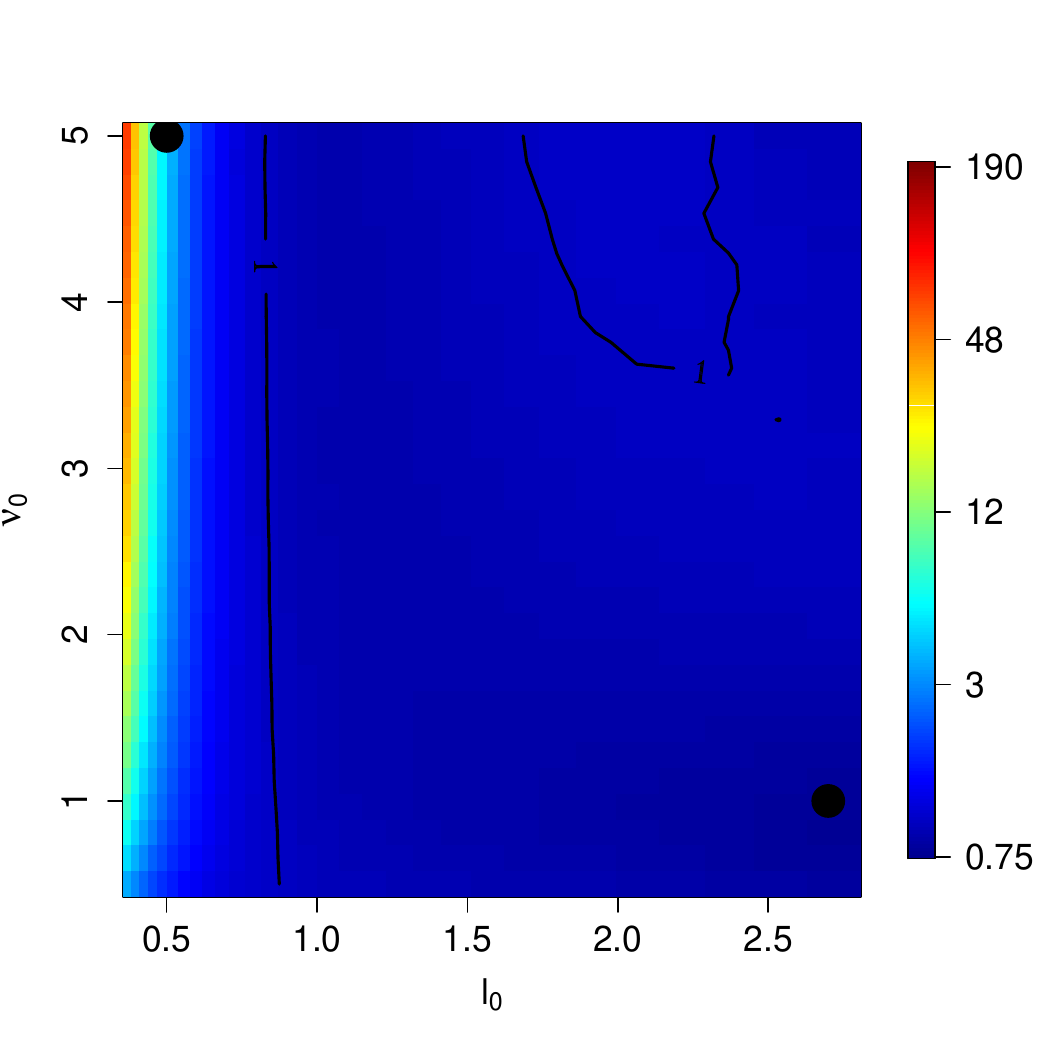}
\end{tabular}
\caption{Estimation of $\ell$. Plot of the ratio of the
asymptotic variance for $\epsilon=0$ over the asymptotic variance for $\epsilon=0.45$ for ML (left) and CV (right).
The true covariance function is Mat\'ern with varying $\ell_0$ and $\nu_0$. We retain the two particular points
$(\ell_0=0.5,\nu_0=5)$ and $(\ell_0=2.7,\nu_0=1)$ for further investigation below in this subsection \ref{subsection: numeric_global}
(these are the black dots).}
\label{fig: Var0surVar45_hatLc}
\end{figure}

\begin{figure}[]
\centering
 \hspace*{-2cm}

\begin{tabular}{c c}
\includegraphics[width=8cm,angle=0]{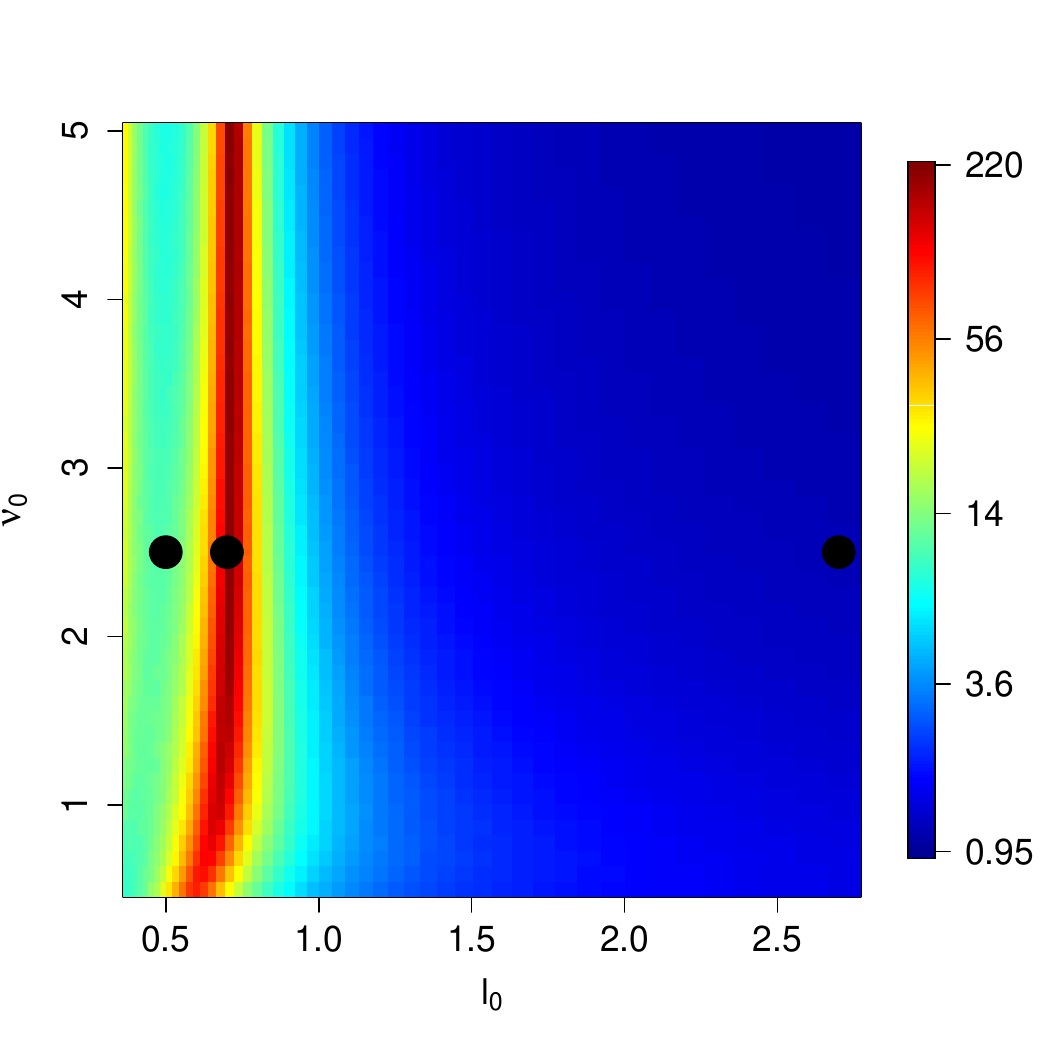} &
\includegraphics[width=8cm,angle=0]{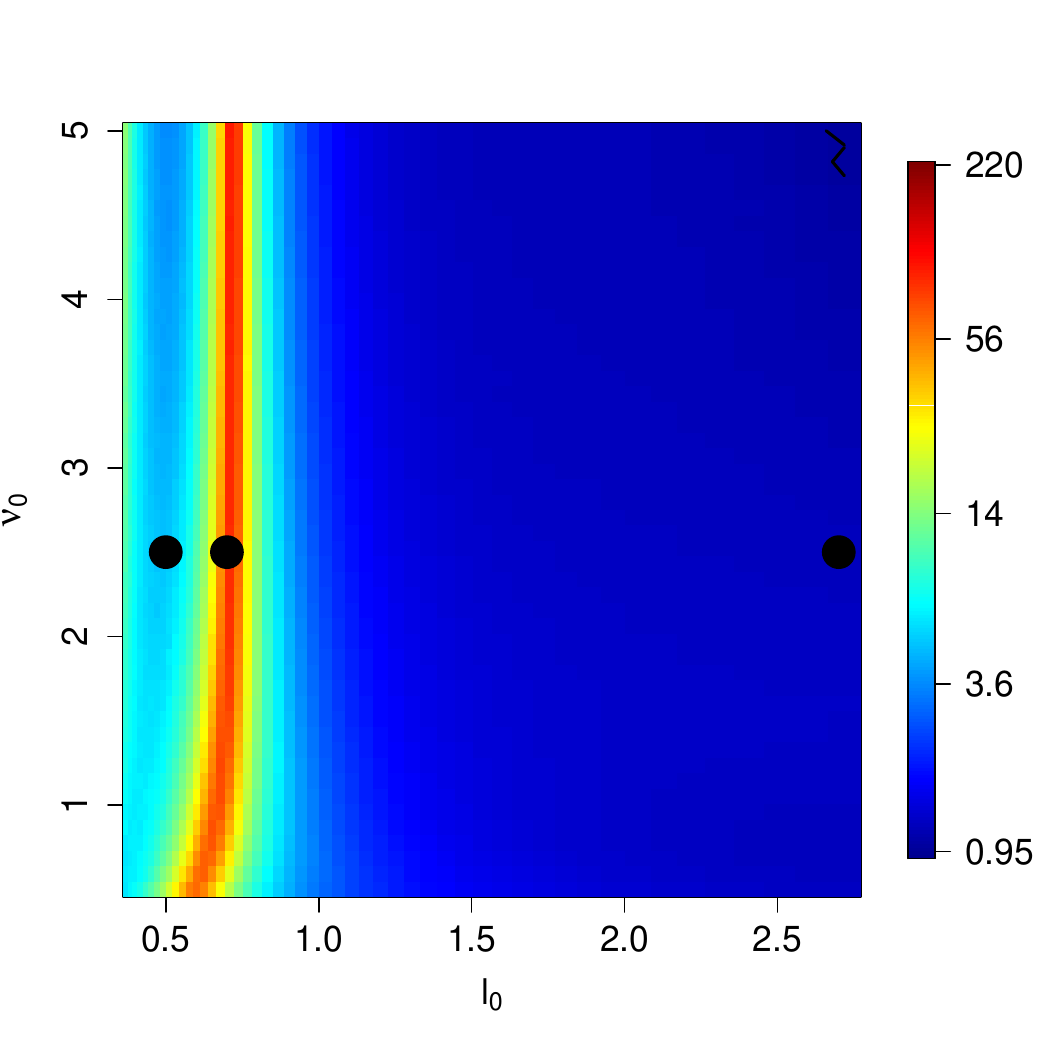}
\end{tabular}
\caption{Same setting as in figure \ref{fig: Var0surVar45_hatLc}, but for the estimation of $\nu$. We retain the three particular points
$\left(\ell_0=0.5,\nu_0=2.5\right)$, $\left(\ell_0=0.7,\nu_0=2.5\right)$ and $\left(\ell_0=2.7,\nu_0=2.5\right)$ for further
investigation below in this subsection \ref{subsection: numeric_global}.}
\label{fig: Var0surVar45_hatNu}
\end{figure}

We now consider the five particular points that we have discussed in subsection \ref{subsection: numeric_local}:
$(\ell_0=0.5,\nu_0=5)$ and $(\ell_0=2.7,\nu_0=1)$ for the estimation of $\ell$
and $\left(\ell_0=0.5,\nu_0=2.5\right)$, $\left(\ell_0=0.7,\nu_0=2.5\right)$ and $\left(\ell_0=2.7,\nu_0=2.5\right)$ for the estimation
of $\nu$. For these particular points, we plot the asymptotic variances of propositions \ref{prop: normaliteML} and \ref{prop: normaliteCV} as functions of $\epsilon$
for $-0.45 \leq \epsilon \leq 0.45$.
The asymptotic variances are even functions of $\epsilon$ since $(\epsilon X_i)_{1 \leq i \leq n}$ has the same distribution as $(- \epsilon X_i)_{1 \leq i \leq n}$. Nevertheless, they are approximated by empirical means of $iid$
realizations of the random traces in propositions \ref{prop: normaliteML} and \ref{prop: gradients_CV}, for $n$ large enough. Hence, the functions we plot
are not exactly even. The fact that they are almost even is a graphical verification that the random fluctuations of the
results of the calculations, for finite (but large) $n$, are very small.
We also plot the second-order Taylor-series expansion given by the value at $\epsilon=0$
and the second derivative at $\epsilon=0$.

On figure \ref{fig: varAss_hatLc_lc=0.5_nu=5}, we show the numerical results for the estimation of $\ell$ with $\left(\ell_0=0.5,\nu_0=5\right)$.
The first observation is that the asymptotic variance is slightly larger for CV than for ML. This is expected: indeed we address a well-specified case,
so that the asymptotic variance of ML is the almost sure limit of the Cramer-Rao bound.
Therefore, this observation turns out to be true
in all the subsection, and we will not comment on it anymore.  
We see that, for both ML and CV, the improvement of the estimation given by the irregularity of the spatial sampling is true for all values of $\epsilon$.
One can indeed gain up to a factor six for the
asymptotic variances. This is explained by the reason mentioned in subsection \ref{subsection: numeric_local}, for $\ell_0$ small, increasing $\epsilon$
yields pairs of observations that become dependent, and hence give information on the covariance structure.

\begin{figure}[]
\centering
 \hspace*{-2cm}

\begin{tabular}{c c}
\includegraphics[width=8cm,angle=0]{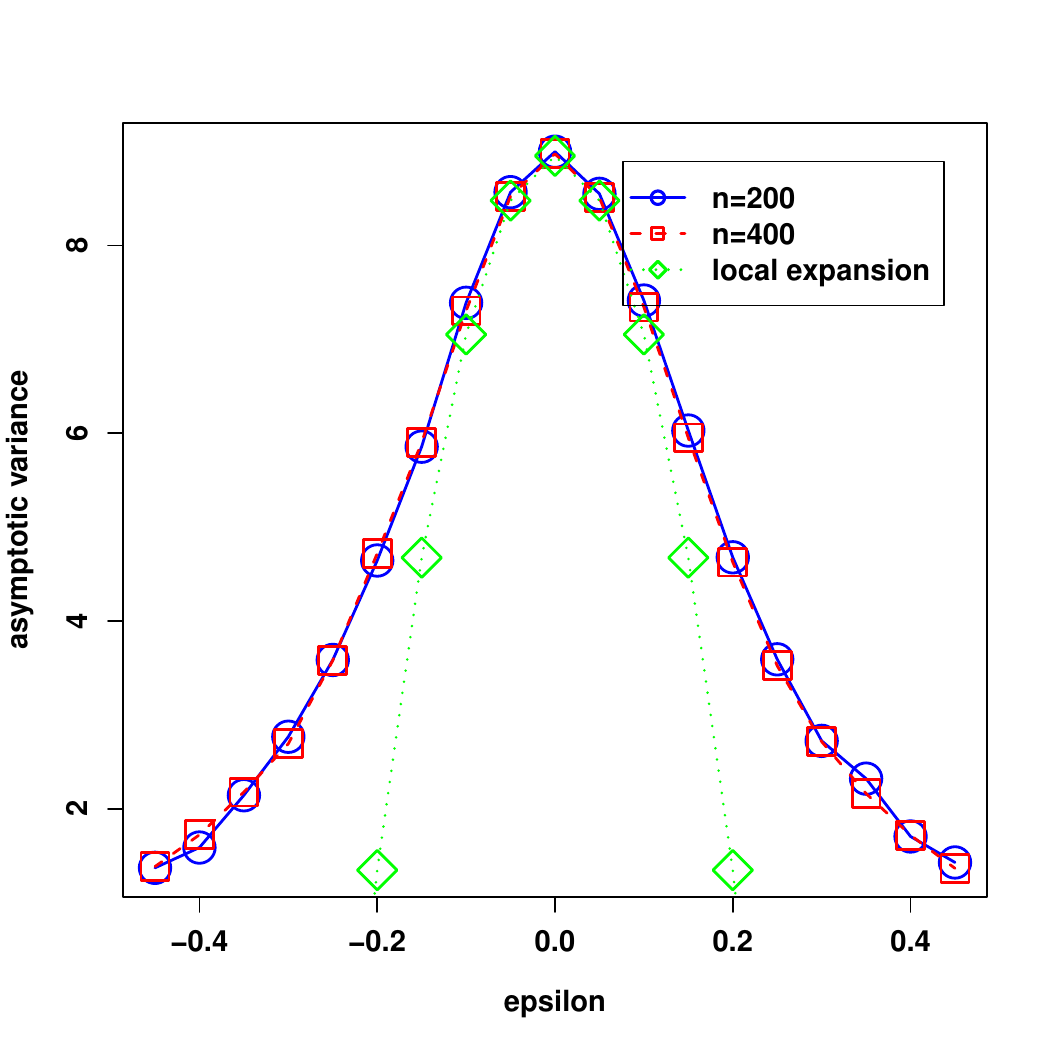} &
\includegraphics[width=8cm,angle=0]{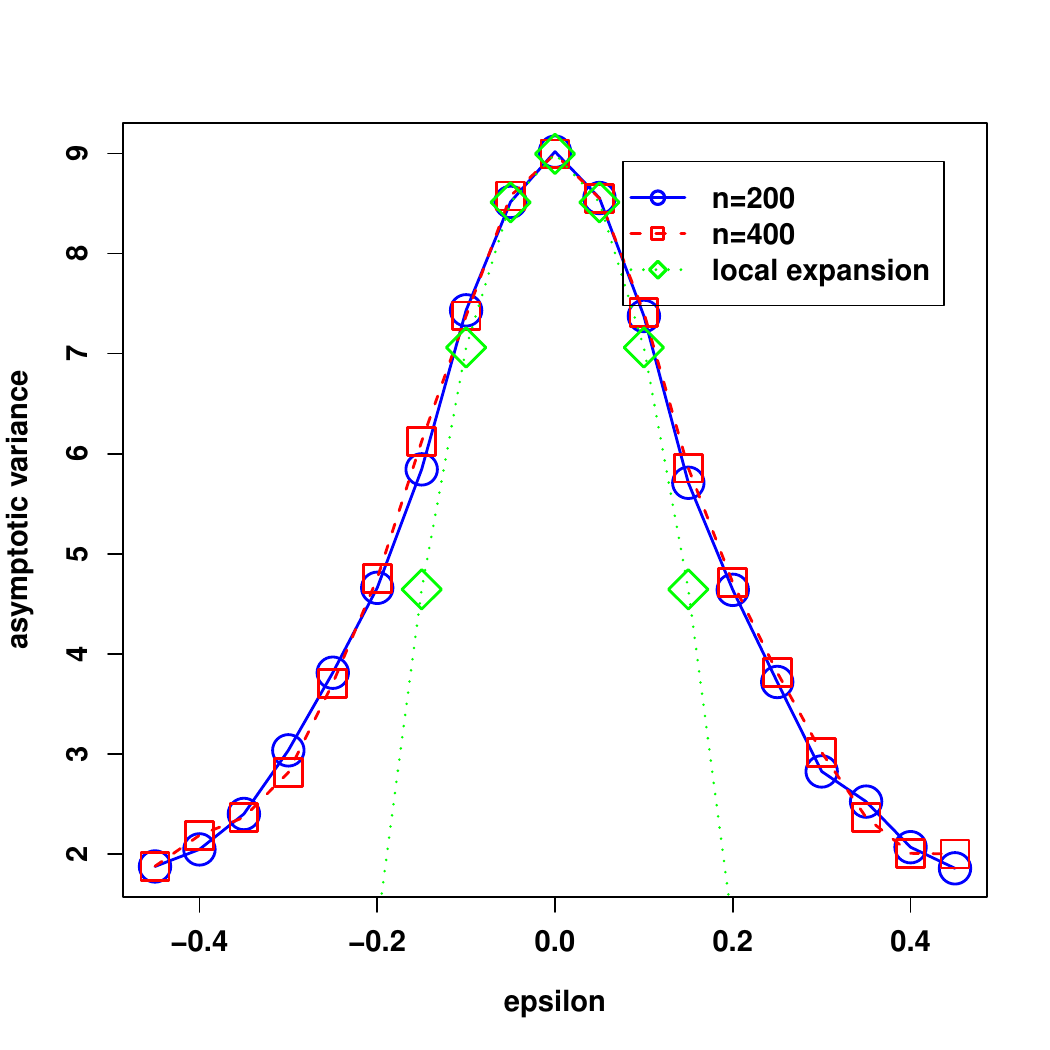}
\end{tabular}
\caption{Global influence of $\epsilon$ for the estimation of the correlation length $\ell$.
Plot of the
asymptotic variance for ML (left) and CV (right), calculated with varying $n$,
and of the second order Taylor series expansion given by the value at $\epsilon=0$ and the second derivative at $\epsilon=0$.
The true covariance function is Mat\'ern with $\ell_0=0.5$ and $\nu_0=5$.
}
\label{fig: varAss_hatLc_lc=0.5_nu=5}
\end{figure}

On figure \ref{fig: varAss_hatLc_lc=2.7_nu=1}, we show the numerical results for the estimation of $\ell$ with $\left(\ell_0=2.7,\nu_0=1\right)$.
For ML, there is a slight global improvement of the estimation with the irregularity of the spatial sampling. However, for CV, there is a
significant degradation of the estimation. Hence the irregularity of the spatial sampling has more relative influence on CV than on ML.
Finally, the advantage of ML over CV for the estimation is by a factor seven, contrary to the case $\ell_0 = 0.5$, where this factor was close to one.

\begin{figure}[]
\centering
 \hspace*{-2cm}

\begin{tabular}{c c}
\includegraphics[width=8cm,angle=0]{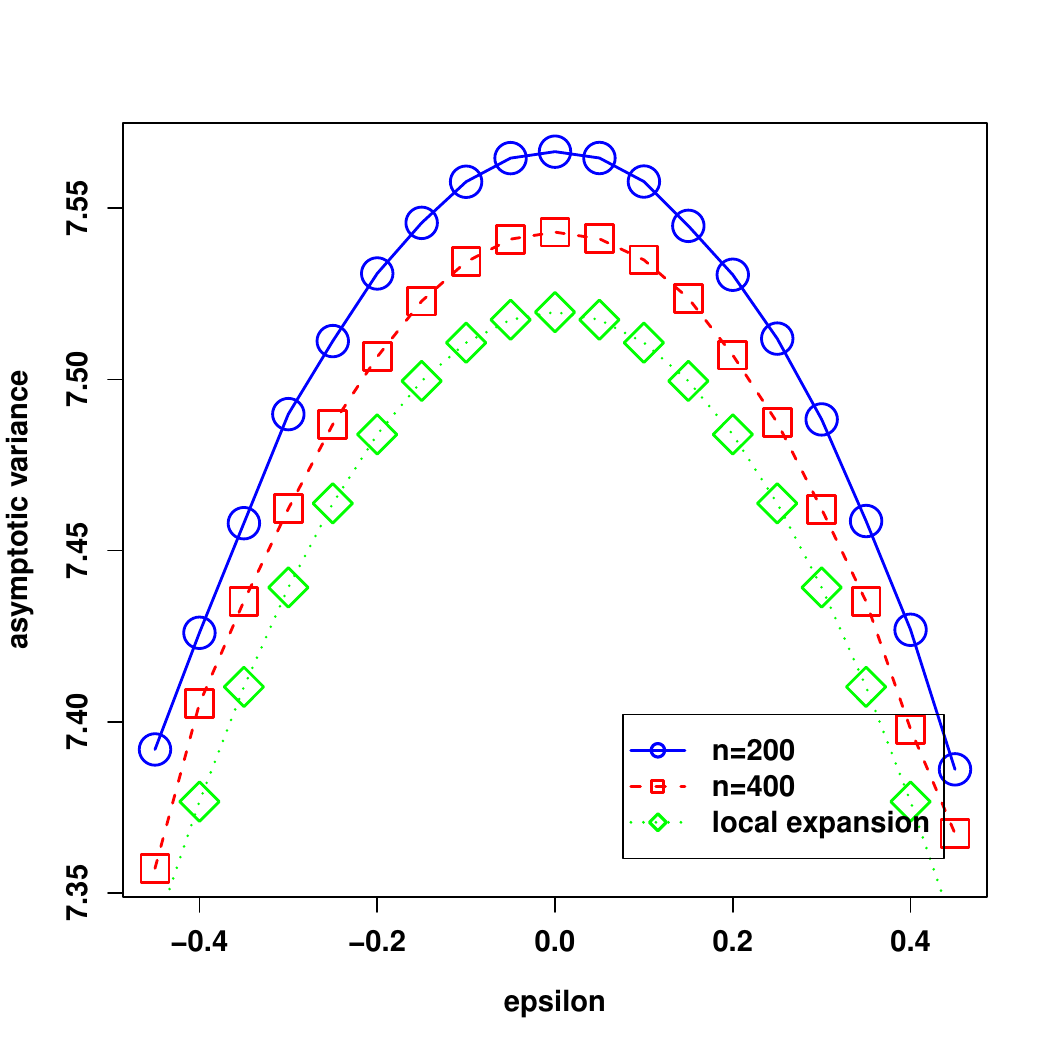} &
\includegraphics[width=8cm,angle=0]{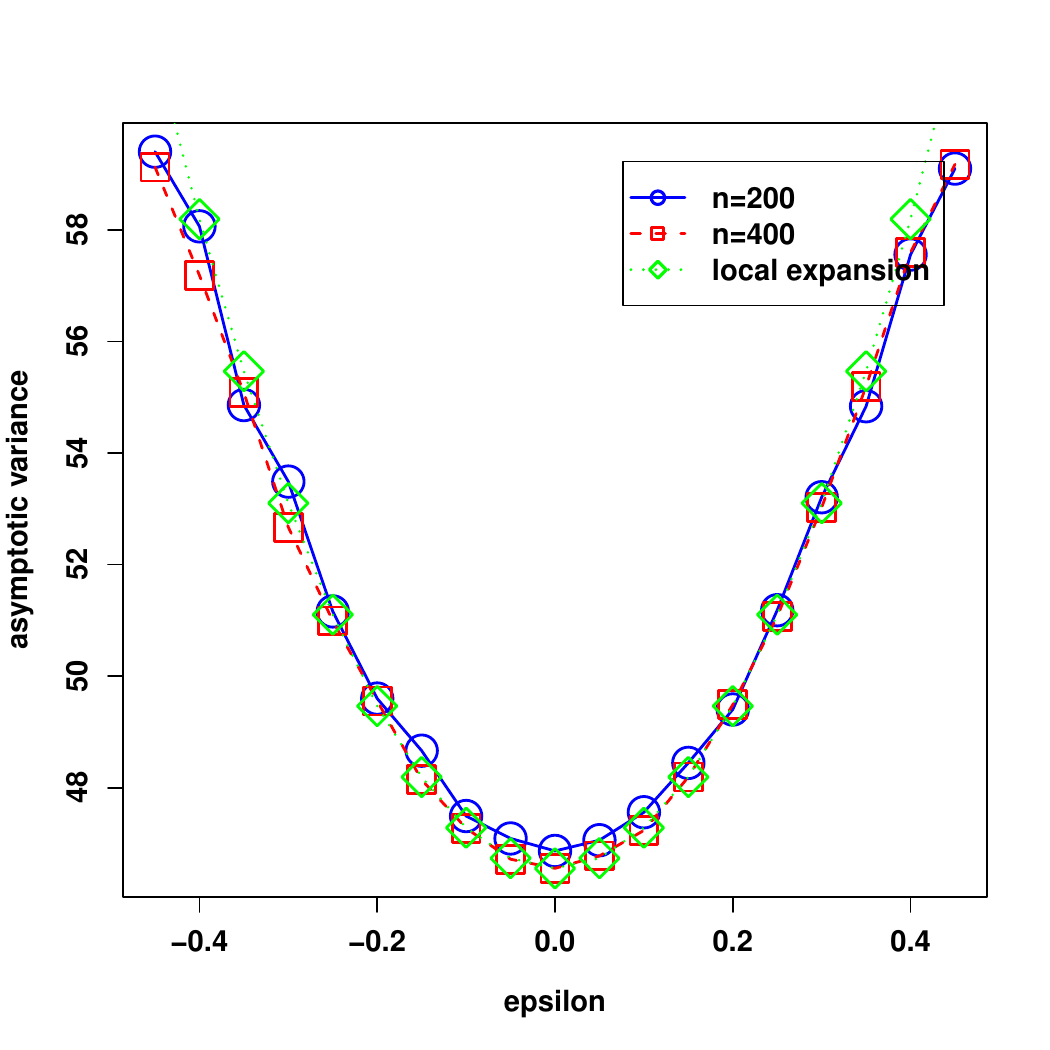}
\end{tabular}
\caption{Same setting as in figure \ref{fig: varAss_hatLc_lc=0.5_nu=5} but with $\ell_0=2.7$ and $\nu_0=1$.
}
\label{fig: varAss_hatLc_lc=2.7_nu=1}
\end{figure}

On figure \ref{fig: varAss_hatNu_lc=0.5_nu=2.5}, we show the numerical results for the estimation of $\nu$ with $\left(\ell_0=0.5,\nu_0=2.5\right)$.
The numerical results are similar for ML and CV. For $\epsilon$ small, the asymptotic variance is very large, because, $\ell_0$ being small, the observations
are almost independent, as the observation points are further apart than the correlation length, making inference on the dependence structure very difficult.
We see that, for $\epsilon=0$, the asymptotic variance is several orders of magnitude larger than for the estimation of $\ell$ in figure \ref{fig: varAss_hatLc_lc=0.5_nu=5},
where $\ell_0$ has the same value. Indeed, in the Mat\'ern model, $\nu$ is a smoothness parameter, and its estimation is very sensitive to the absence of observation points
with small spacing.
We observe, as
discussed in figure \ref{fig: d2surVal0_hatNu}, that for $\epsilon \in [0,0.2]$, the asymptotic variance increases with $\epsilon$
because pairs of observation points can reach the state where the covariance of the two observations is almost independent of $\nu$.
For $\epsilon \in [0.2,0.5)$,
a threshold is reached where pairs of subsequently dependent observations start to appear, greatly reducing the asymptotic variance for the estimation
of $\nu$.

\begin{figure}[]
\centering
 \hspace*{-2cm}

\begin{tabular}{c c}
\includegraphics[width=8cm,angle=0]{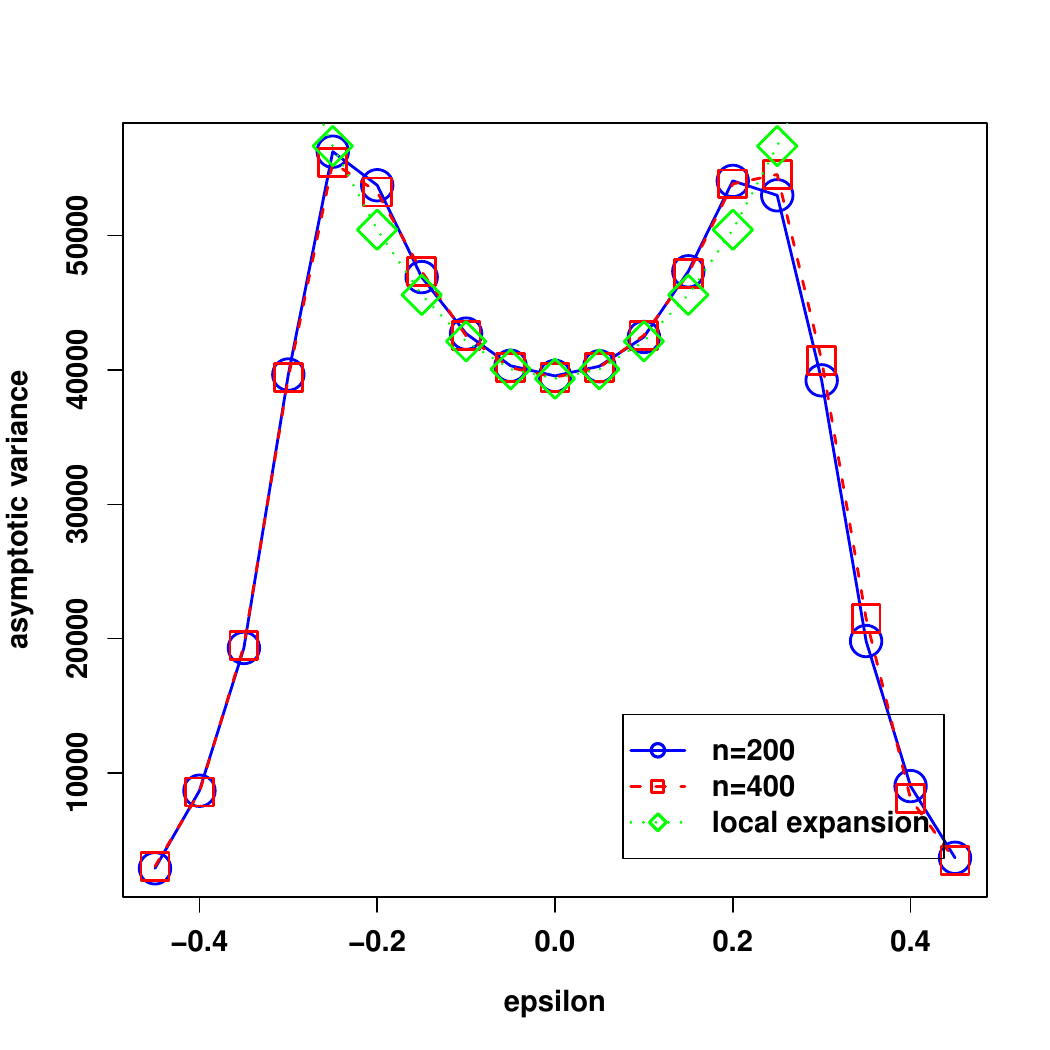} &
\includegraphics[width=8cm,angle=0]{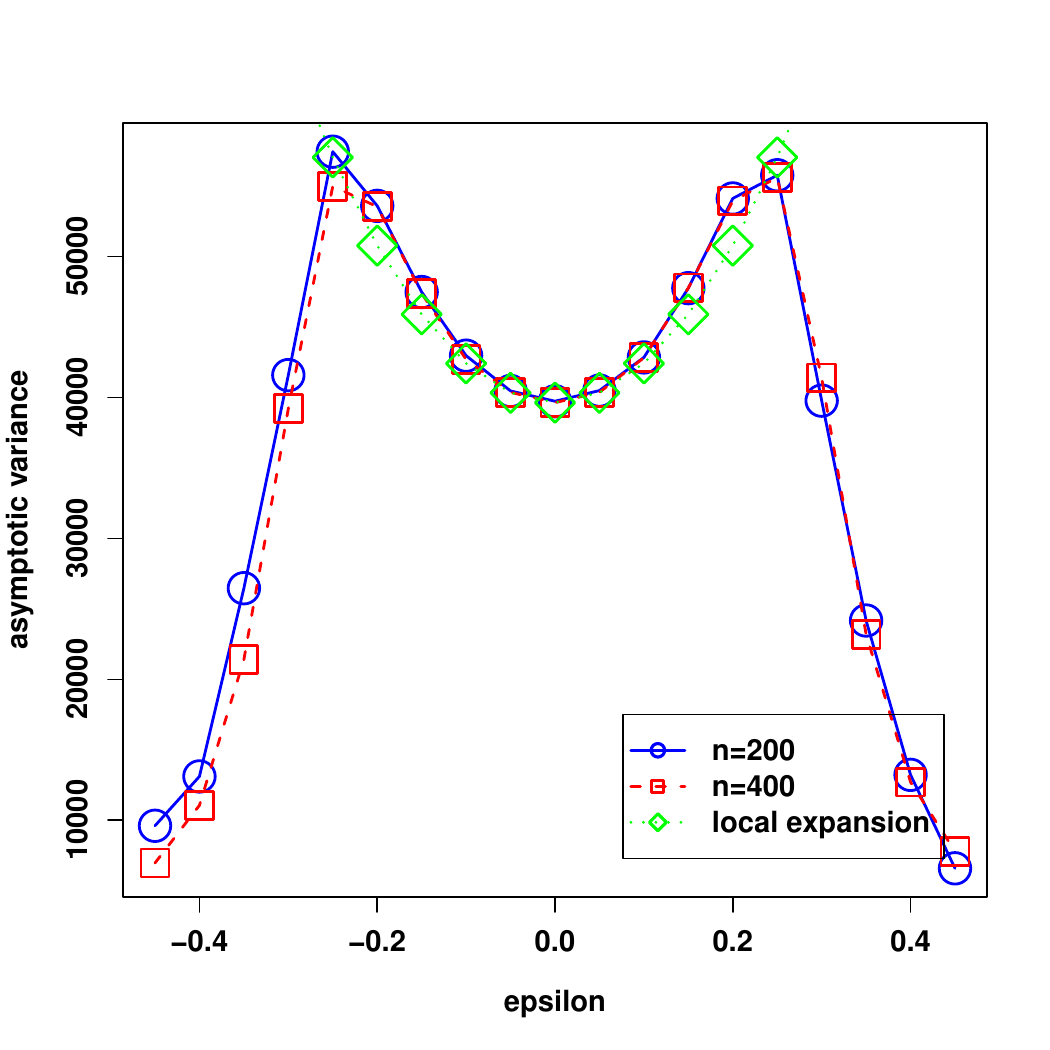}
\end{tabular}
\caption{Same setting as in figure \ref{fig: varAss_hatLc_lc=0.5_nu=5} but for the estimation of $\nu$ and with $\ell_0=0.5$ and $\nu_0=2.5$.}
\label{fig: varAss_hatNu_lc=0.5_nu=2.5}
\end{figure}

On figure \ref{fig: varAss_hatNu_lc=0.7_nu=2.5}, we show the numerical results for the estimation of $\nu$ with $\left(\ell_0=0.7,\nu_0=2.5\right)$.
The numerical results are similar for ML and CV. Similarly to figure \ref{fig: varAss_hatNu_lc=0.5_nu=2.5}, the asymptotic variance is very large, because the
observations are almost independent. For $\epsilon=0$, it is even larger than in figure \ref{fig: varAss_hatLc_lc=0.5_nu=5}
because we are in the state where the covariance between two successive observations is almost independent of $\nu$.
As an illustration,
for $\ell=0.7$ and $\nu=2.5$, the derivative of this covariance with respect to $\nu$ is $-1.3 \times 10^{-3}$ for a value of $0.13$ ($1\%$ relative variation),
while for $\ell=0.5$ and $\nu=2.5$, this derivative is $-5 \times 10^{-3}$ for a value of $0.037$ ($13\%$ relative variation).
Hence, the asymptotic variance is globally decreasing with $\epsilon$ and the decrease is very strong for small $\epsilon$. The variance is several orders of magnitude
smaller for large $\epsilon$, where pairs of dependent observations start to appear.

\begin{figure}[]
\centering
 \hspace*{-2cm}

\begin{tabular}{c c}
\includegraphics[width=8cm,angle=0]{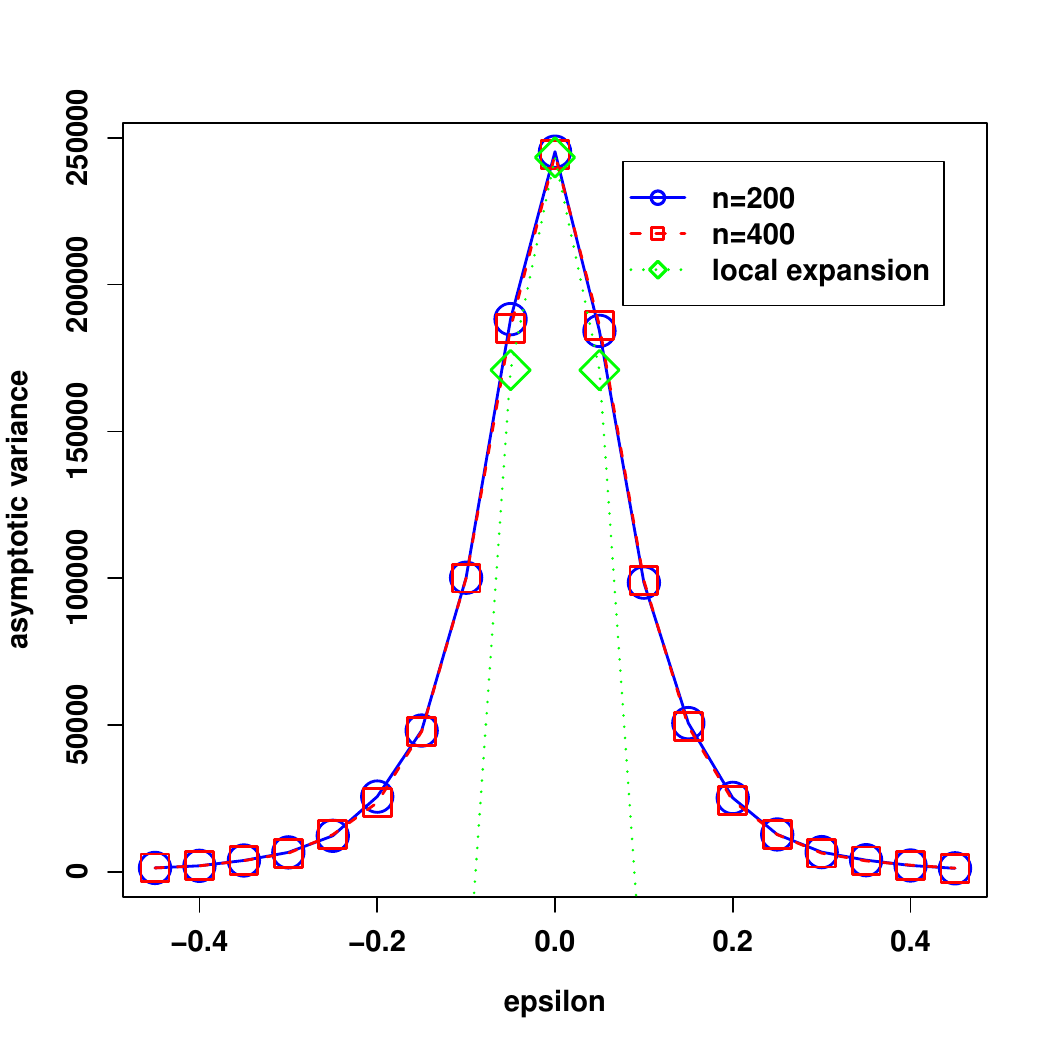} &
\includegraphics[width=8cm,angle=0]{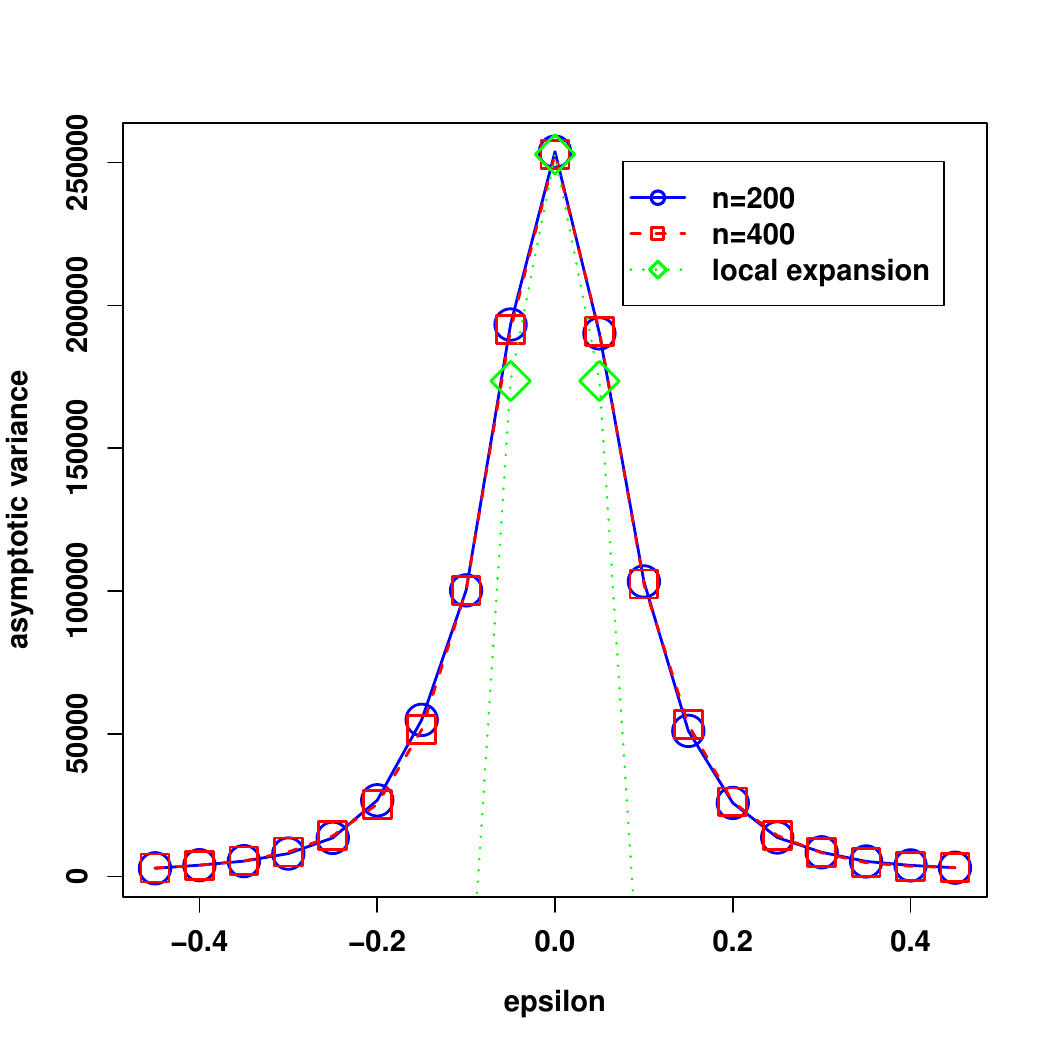}
\end{tabular}
\caption{Same setting as in figure \ref{fig: varAss_hatLc_lc=0.5_nu=5} but for the estimation of $\nu$ and with $\ell_0=0.7$ and $\nu_0=2.5$.}
\label{fig: varAss_hatNu_lc=0.7_nu=2.5}
\end{figure}

On figure \ref{fig: varAss_hatNu_lc=2.7_nu=2.5}, we show the numerical results for the estimation of $\nu$ with $\left(\ell_0=2.7,\nu_0=2.5\right)$.
For both ML and CV, there is a global improvement of the estimation with the irregularity of the spatial sampling.
Moreover,
the advantage of ML over CV for the estimation, is by a factor seven,
contrary to figures \ref{fig: varAss_hatNu_lc=0.5_nu=2.5} and \ref{fig: varAss_hatNu_lc=0.7_nu=2.5}, where this factor was close to one.

\begin{figure}[]
\centering
 \hspace*{-2cm}

\begin{tabular}{c c}
\includegraphics[width=8cm,angle=0]{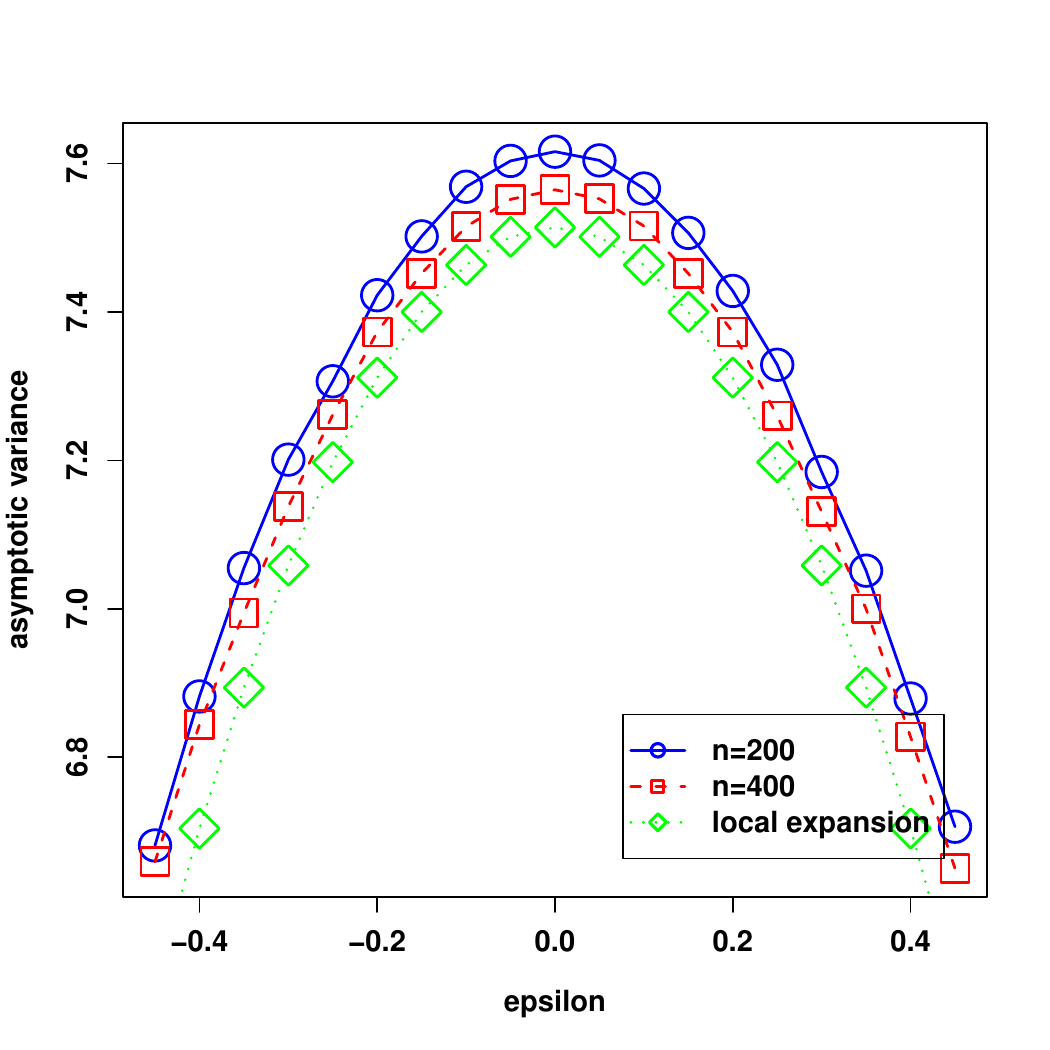} &
\includegraphics[width=8cm,angle=0]{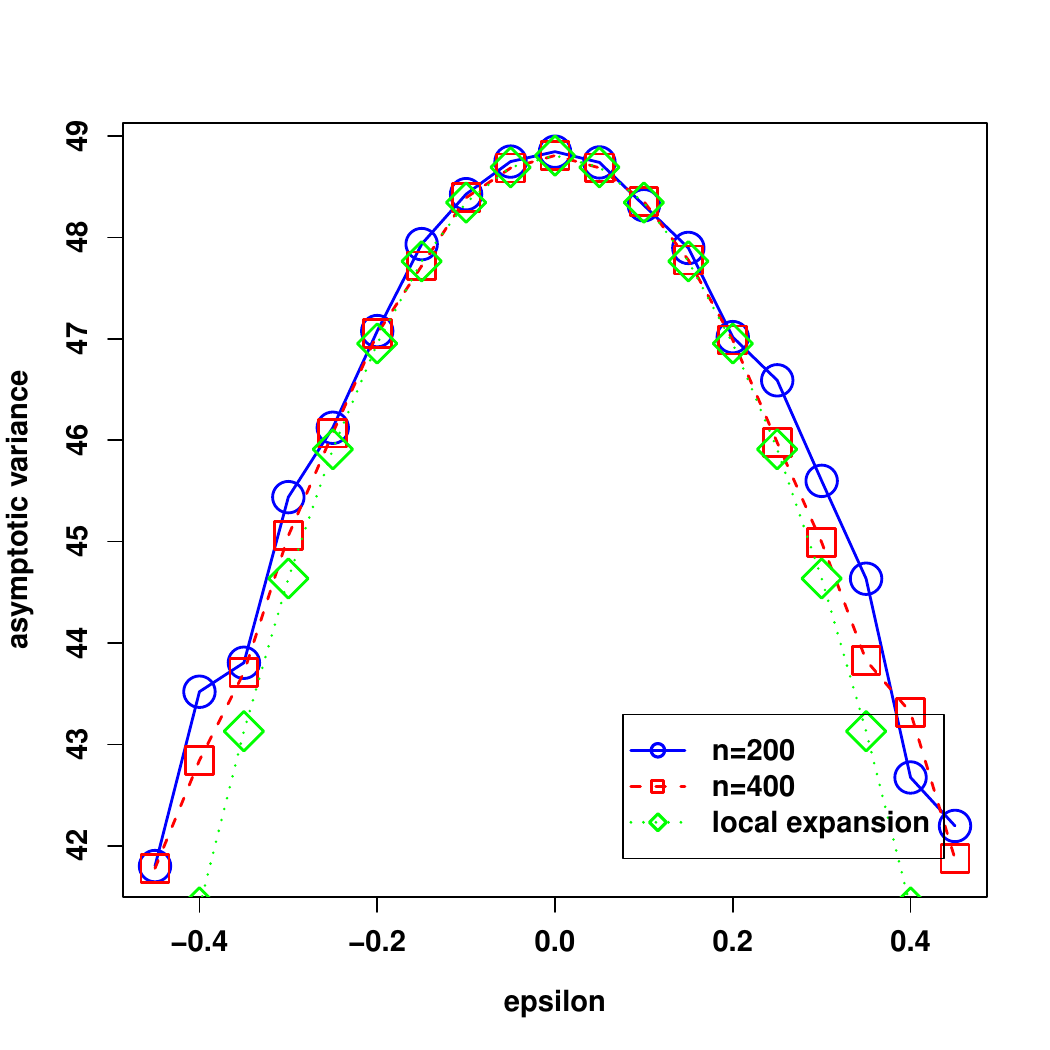}
\end{tabular}
\caption{Same setting as in figure \ref{fig: varAss_hatLc_lc=0.5_nu=5} but for the estimation of $\nu$ and with $\ell_0=2.7$ and $\nu_0=2.5$.}
\label{fig: varAss_hatNu_lc=2.7_nu=2.5}
\end{figure}

\subsection{Estimating both the correlation length and the smoothness parameter}  \label{subsection: joint}

In this subsection \ref{subsection: joint}, the case of the joint estimation of $\ell$ and $\nu$ is addressed.
We denote, for ML and CV, $V_{\ell}$, $V_{\nu}$ and $C_{\ell,\nu}$, the asymptotic variances of $\sqrt{n} \hat{\ell}$
and $\sqrt{n}\hat{\nu}$ and the asymptotic covariance of $\sqrt{n}\hat{\ell}$ and $\sqrt{n}\hat{\nu}$ (propositions \ref{prop: normaliteML}
and \ref{prop: normaliteCV}).

Since we here address $2 \times 2$ covariance matrices, the impact of the irregularity parameter $\epsilon$
on the estimation is now more complex to assess. For instance, increasing $\epsilon$ could increase $V_{\ell}$ and at the same time decrease $V_{\nu}$.
Thus, it is desirable to build scalar criteria, defined in terms of $V_{\ell}$, $V_{\nu}$ and $C_{\ell,\nu}$, measuring the quality of the estimation.
In \cite{SSDUIAF}, the criterion used is the average, over a prior distribution on $(\ell_0,\nu_0)$, of
$\log{(V_{\ell} V_{\nu} - C_{\ell,\nu}^2 )}$, that is the averaged logarithm of the determinant of the covariance matrix. This criterion corresponds
to $D$-optimality in standard linear regression with uncorrelated errors, as noted in \cite{SSDUIAF}. In our case, we know
the true $(\ell_0,\nu_0)$, so that the Bayesian average is not needed. The first scalar criterion we study is thus
$D_{\ell,\nu} := V_{\ell} V_{\nu} - C_{\ell,\nu}^2$. This criterion is interpreted as a general objective-free estimation criterion, in the sense that
the impact of the estimation on Kriging predictions that would be made afterward, on new input points, is not directly addressed in $D_{\ell,\nu}$.

One could build other scalar criteria, explicitly addressing the impact of the covariance function estimation error, on the quality of the Kriging predictions
that are made afterward. In \cite{SSDUIAF}, the criterion studied for the prediction error is the integral over the prediction domain of
$\EE \left[ \left( \hat{Y}_{\theta_0}(t) - \hat{Y}_{\hat{\theta}}(t)  \right)^2 \right]$, where $\hat{Y}_{\theta}(t)$ is the prediction
of $Y(t)$, from the observation vector, and under covariance function $K_{\theta}$. This criterion is the difference of integrated
prediction mean square error, between the true and estimated covariance functions. In \cite{abt99estimating}
and in \cite{SSDUIAF}, two different asymptotic approximations of this criterion are studied. In \cite{SSDUIAF},
another criterion, focusing on the accuracy of the Kriging predictive variances built from $\hat{\theta}$, is also treated, together with a
corresponding asymptotic approximation.
In \cite{Bachoc2013cross}, a criterion for the accuracy of the Kriging predictive variances, obtained from an estimator of the variance parameter, is studied, when the correlation function is fixed and misspecified.
Since we specifically address the case of Kriging prediction in the asymptotic
framework addressed here in section \ref{subsection: inf_hyp_est}, we refer to \cite{abt99estimating,SSDUIAF,Bachoc2013cross}
for details on the aforementioned criteria. In this subsection \ref{subsection: joint},
we study the estimation criteria $V_{\ell}$, $V_{\nu}$, $C_{\ell,\nu}$ and $D_{\ell,\nu}$.

In figure \ref{fig: Var0surVar45_joint_ML}, we consider the ML estimation, with varying $(\ell_0,\nu_0)$.
We study the ratio of $V_{\ell}$, $V_{\nu}$ and $D_{\ell,\nu}$, between $\epsilon=0$ and $\epsilon=0.45$.
We first observe that $V_{\nu}$ is always smaller for $\epsilon=0.45$ than for $\epsilon=0$, that is to say
there is an improvement of the estimation of $\nu$ when using a strongly irregular sampling. For $V_{\ell}$,
this is the same, except in a thin band around $\ell_0 \approx 0.73$. Our explanation for this fact is the same as
for a similar singularity in figure \ref{fig: d2surVal0_hatNu}.
For $\ell_0 = 0.73$ and $\epsilon=0$, the correlation between two successive points is approximatively only a function of $\ell$.
For instance, the derivative of this correlation with respect to $\nu$ at $\ell= 0.73,\nu=2.5$ is $-3.7 \times 10^{-5}$ for a correlation of $0.15$.
Thus, the very large uncertainty on $\nu$ has no negative impact on the information brought by the pairs of successive observation points on $\ell$
for $\epsilon=0$. These pairs of successive points bring most of the information on the covariance function, since $\ell_0$ is small.
When $\epsilon = 0.45$, this favorable case is broken by the random perturbations, and the large uncertainty on $\nu$ has a negative impact
on the estimation of $\ell$, even when considering the pairs of successive observation points.

Nevertheless, in the band around $\ell_0 \approx 0.73$, when going from $\epsilon=0$ to $\epsilon=0.45$, the improvement of the estimation of $\nu$
is much stronger than the degradation of the estimation of $\ell$.
This is confirmed by the plot of $D_{\ell,\nu}$, which always decreases
when going from $\epsilon=0$ to $\epsilon=0.45$. Thus, we confirm our global conclusion of subsection \ref{subsection: numeric_global}:
strong perturbations of the regular grid create pairs of observation points with small spacing, which is always beneficial for ML in the cases we address.

Finally, notice that we have discussed a case where the estimation of a covariance parameter is degraded, while the estimation of the other one is improved.
This justifies the use of scalar criteria of the estimation, such as $D_{\ell,\nu}$, or the ones related with prediction discussed above.

We retain the particular point
$\left(\ell_0=0.73,\nu_0=2.5\right)$, that corresponds to the case where
going from $\epsilon=0$ to $\epsilon=0.45$ decreases $V_{\nu}$ and increases $V_{\ell}$, for further global
investigation in figure \ref{fig: global_joint_ML}.

\begin{figure}[]
\centering
 \hspace*{-2cm}

\begin{tabular}{c c c}
\includegraphics[width=5cm,angle=0]{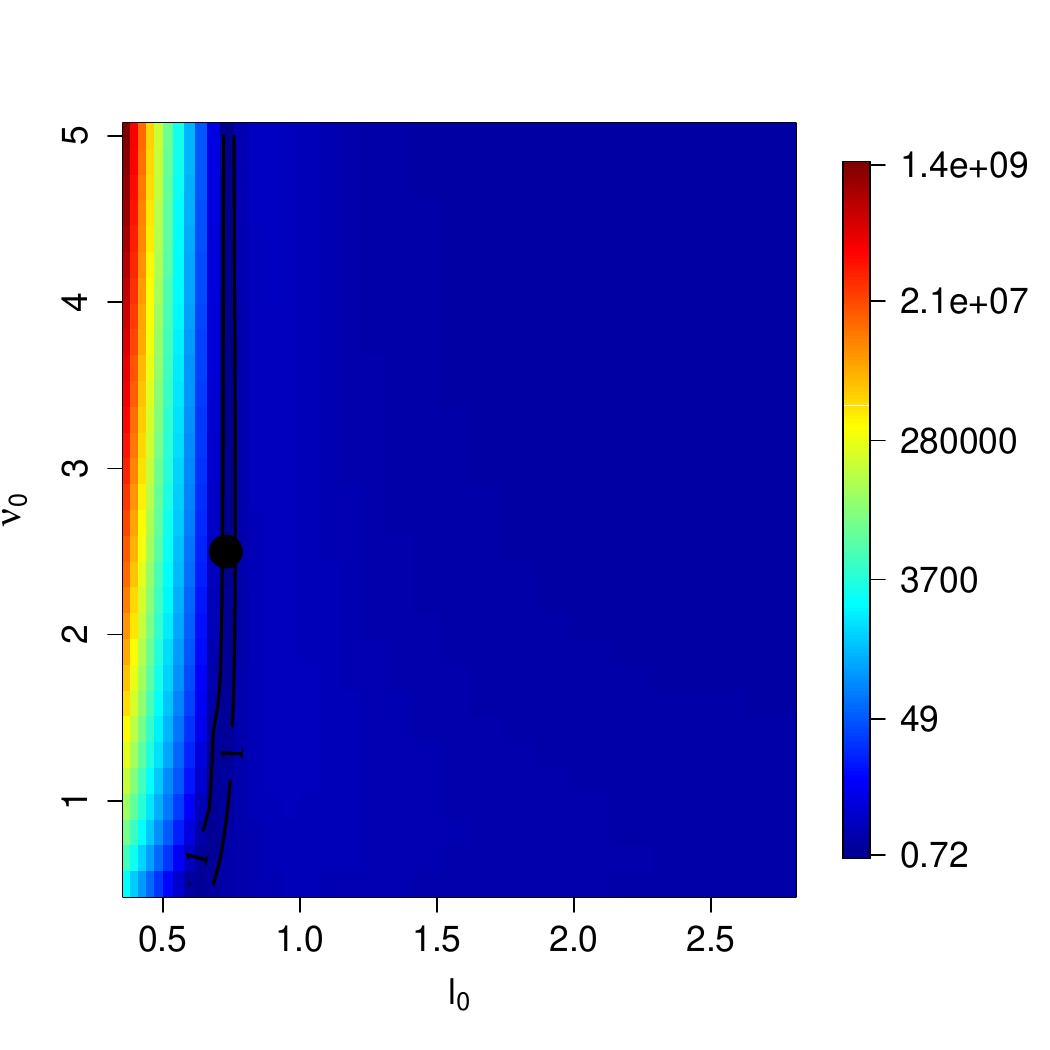} &
\includegraphics[width=5cm,angle=0]{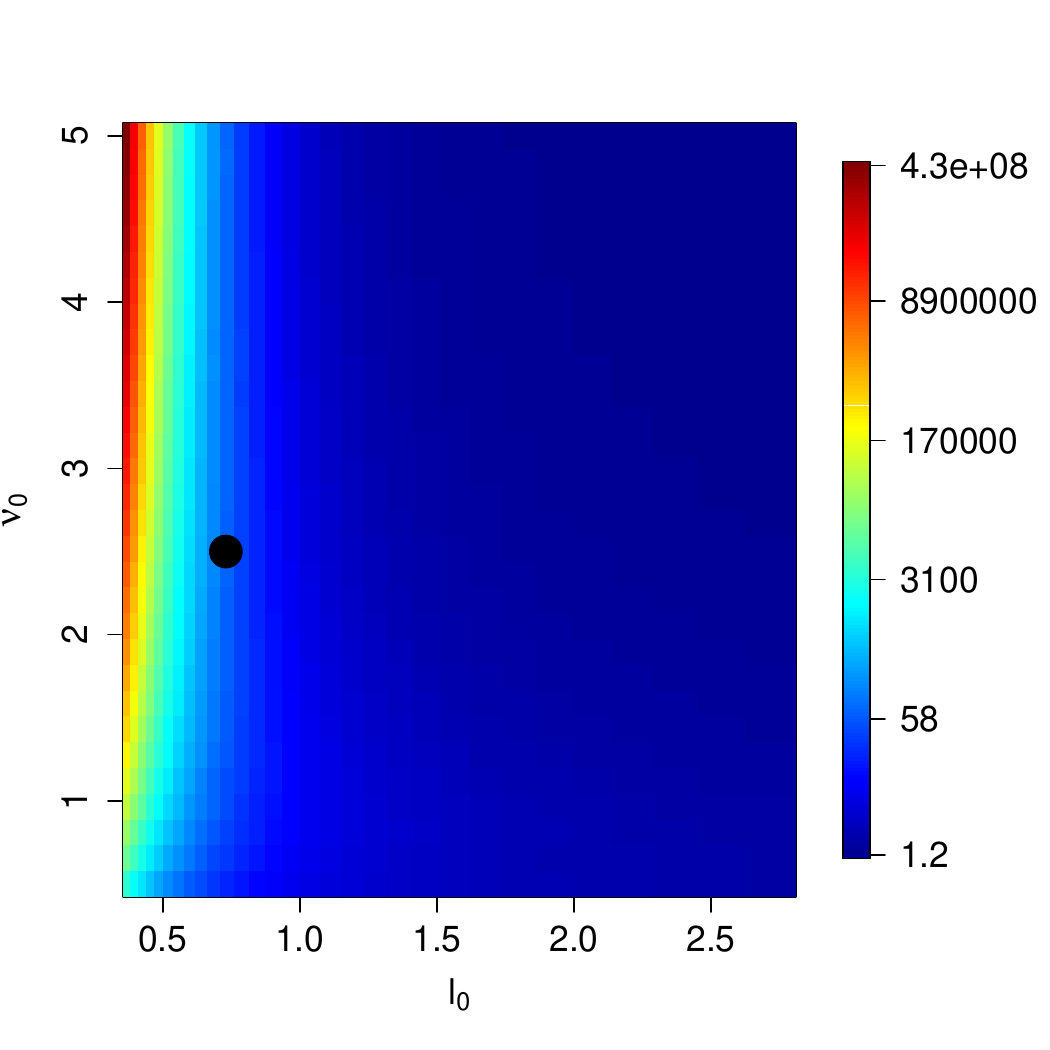} &
\includegraphics[width=5cm,angle=0]{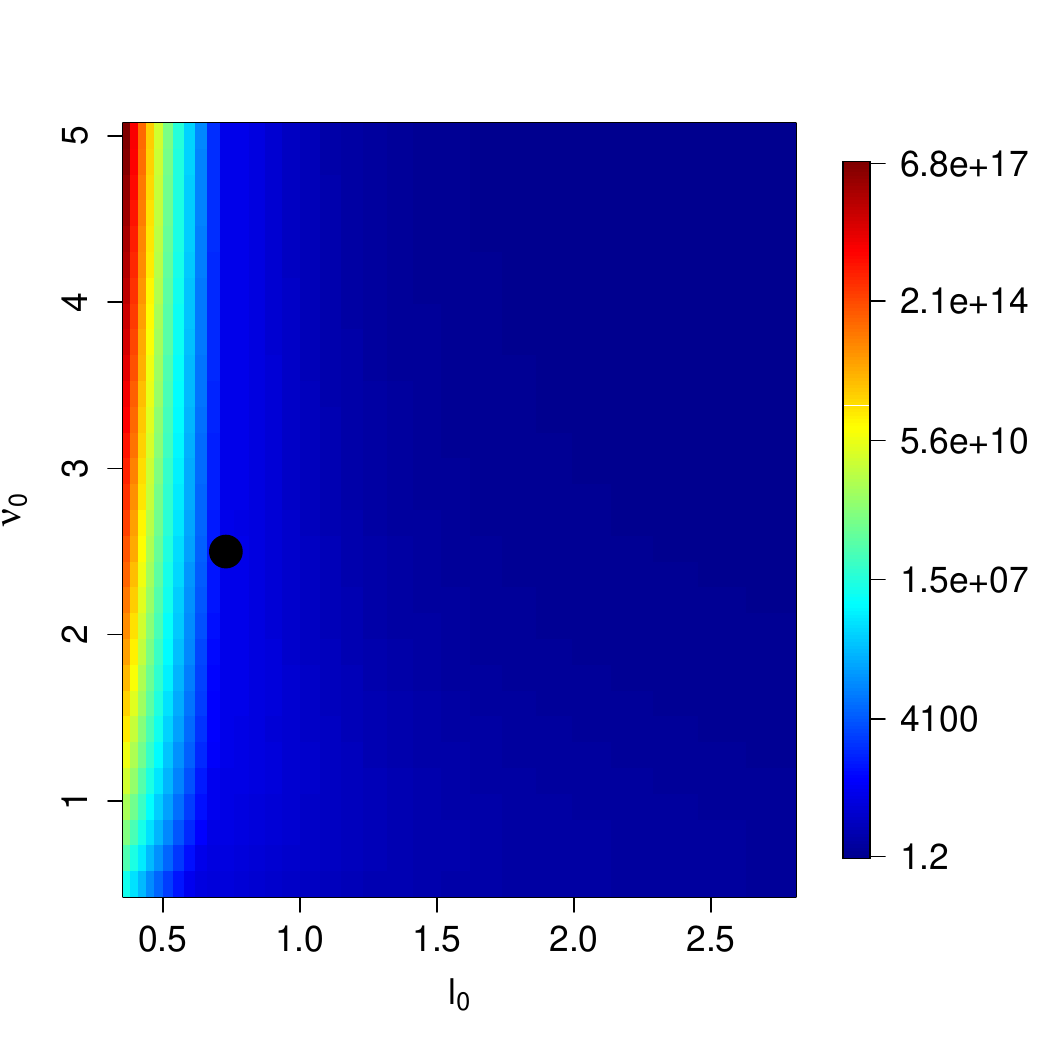}
\end{tabular}
\caption{For ML, plot of the ratio, between $\epsilon=0$ and $\epsilon=0.45$, of $V_{\ell}$ (left), $V_{\nu}$ (center) and $D_{\ell,\nu}$ (right).
The true covariance function is Mat\'ern with varying $\ell_0$ and $\nu_0$. We jointly estimate  $\ell$ and $\nu$.
We retain the particular point
$\left(\ell_0=0.73,\nu_0=2.5\right)$ for further
investigation below in this subsection \ref{subsection: joint}.}
\label{fig: Var0surVar45_joint_ML}
\end{figure}

In figure \ref{fig: Var0surVar45_joint_CV}, we address the same setting as in figure \ref{fig: Var0surVar45_joint_ML}, but for
the CV estimation. We observe that going from $\epsilon=0$ to $\epsilon=0.45$ can increase $D_{\ell,\nu}$.
This is a confirmation of what was observed in figure \ref{fig: Var0surVar45_hatLc}:
strong irregularities of the spatial sampling can globally damage the CV estimation. The justification is the same as
before: the LOO error variances become heterogeneous when the regular grid is perturbed.

We also observe an hybrid case, in which the estimation of $\ell$ and $\nu$ is improved by the irregularity, but the determinant of their asymptotic covariance
matrix increases, because the absolute value of their asymptotic covariance decreases. This case happens for instance
around the point $\left(\ell_0=1.7,\nu_0=5\right)$, that we retain for a further global investigation in figure \ref{fig: global_joint_CV}.

\begin{figure}[]
\centering
 \hspace*{-2cm}

\begin{tabular}{c c c}
\includegraphics[width=5cm,angle=0]{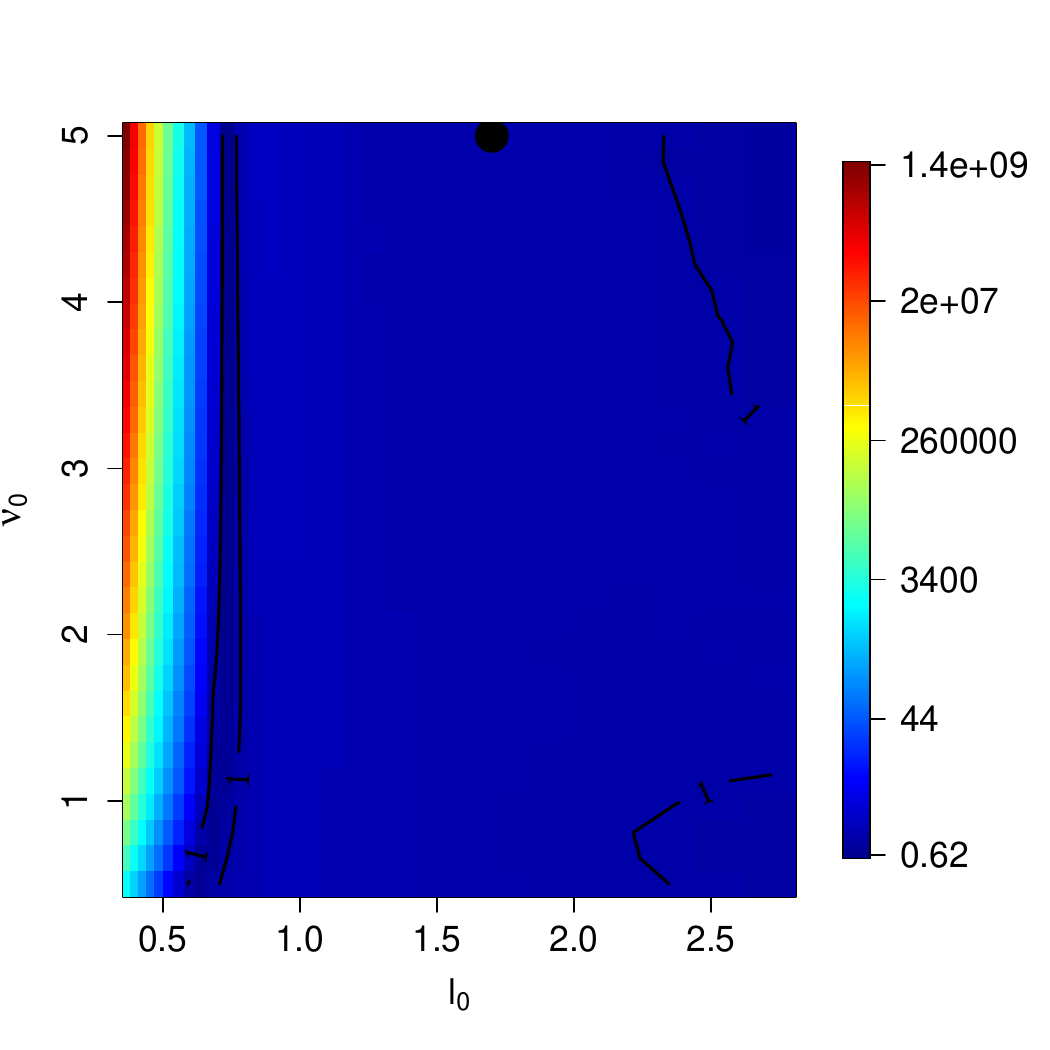} &
\includegraphics[width=5cm,angle=0]{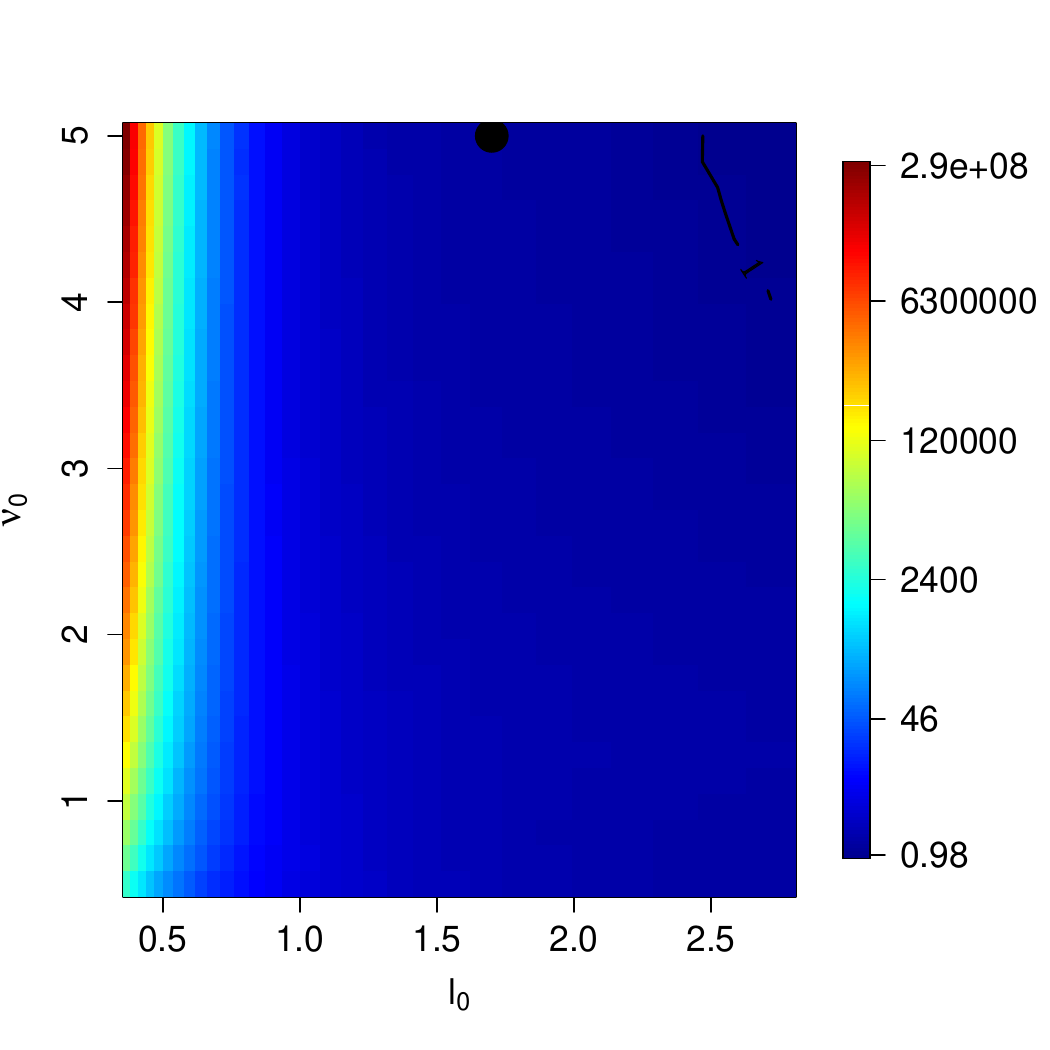} &
\includegraphics[width=5cm,angle=0]{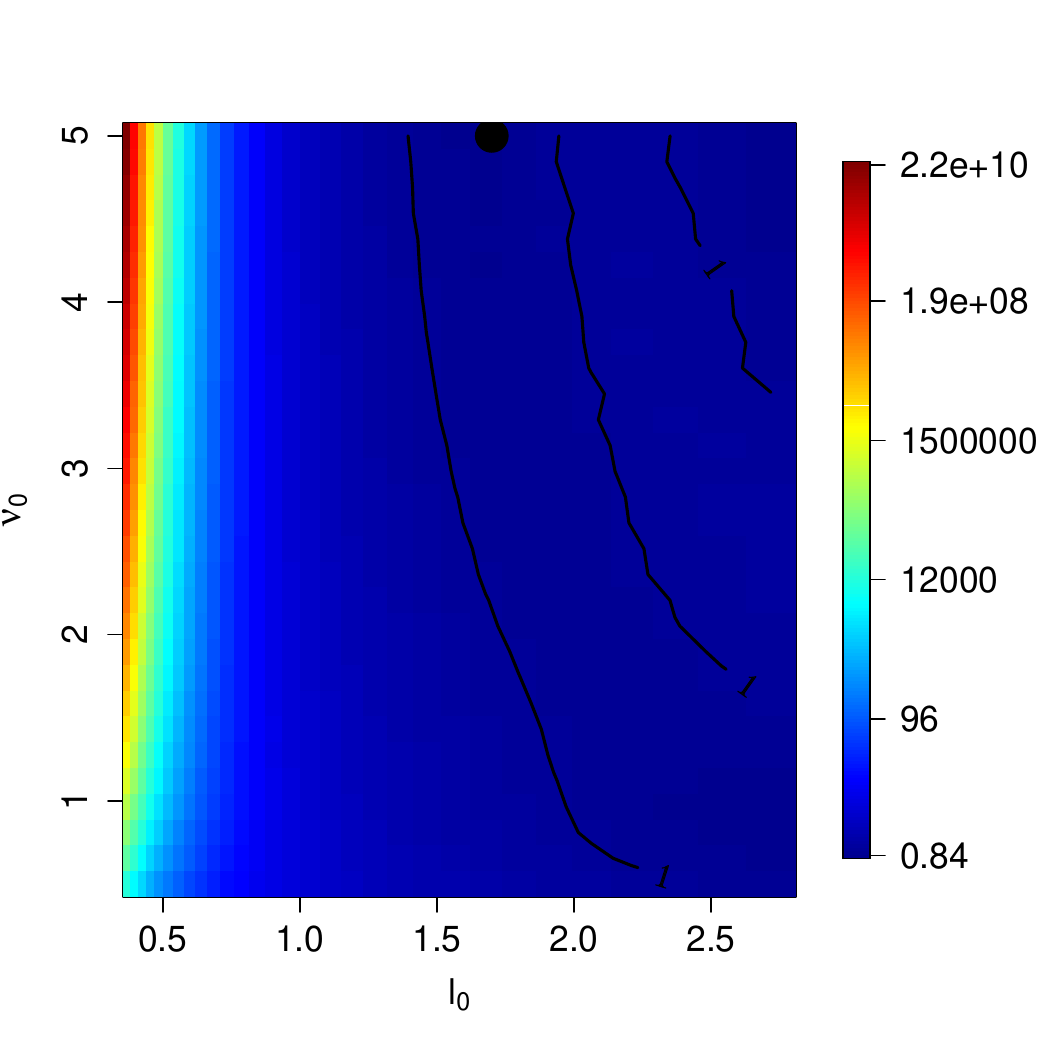}
\end{tabular}
\caption{Same setting as in figure \ref{fig: Var0surVar45_joint_ML} but for CV.
We retain the particular point
$\left(\ell_0=1.7,\nu_0=5\right)$ for further
investigation below in this subsection \ref{subsection: joint}.}
\label{fig: Var0surVar45_joint_CV}
\end{figure}

In figure \ref{fig: global_joint_ML}, for $\ell_0 = 0.73,\nu_0 = 2.5$ and for ML, we plot $V_{\ell}$, $V_{\nu}$
and $D_{\ell,\nu}$ with respect to $\epsilon$, for $\epsilon \in [0, 0.45 ]$. We confirm that
when $\epsilon$ increases, the decrease of $V_{\nu}$ is much stronger than the increase of $V_{\ell}$.
As a result, there is a strong decrease of $D_{\ell,\nu}$. This is a confirmation of our main conclusion
on the impact of the spatial sampling on the estimation: using pairs of closely spaced observation points improves the ML estimation.

\begin{figure}[]
\centering
 \hspace*{-2cm}

\begin{tabular}{c c c}
\includegraphics[width=5cm,angle=0]{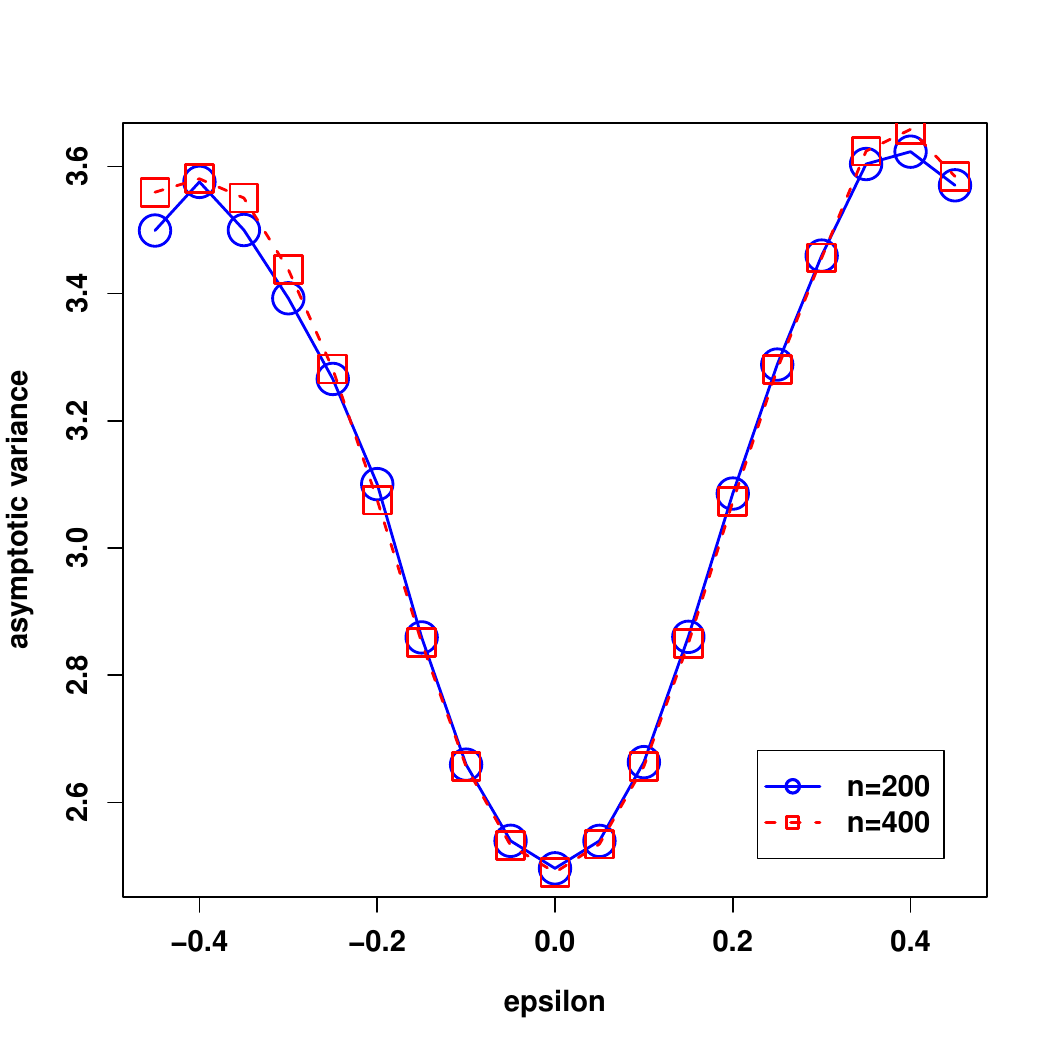} &
\includegraphics[width=5cm,angle=0]{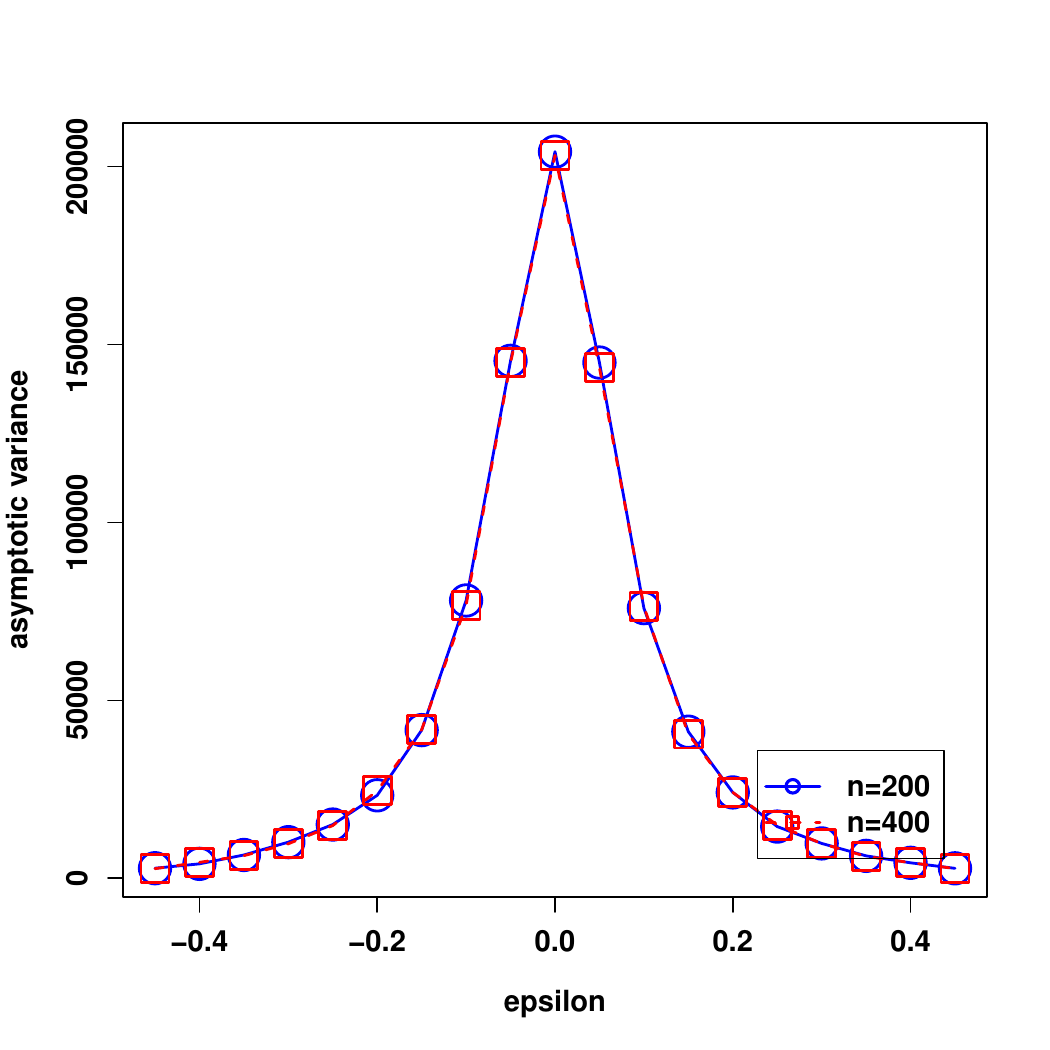} &
\includegraphics[width=5cm,angle=0]{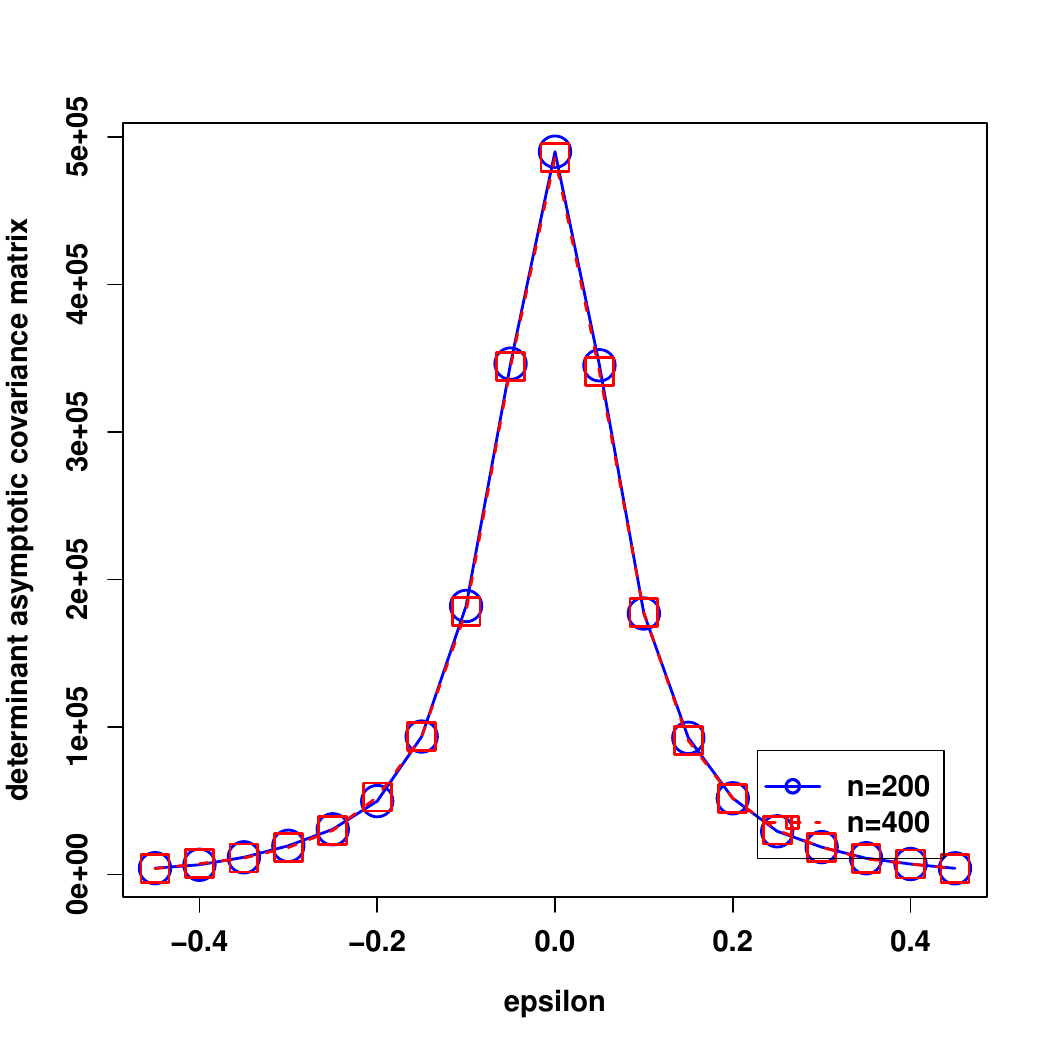}  
\end{tabular}
\caption{Joint estimation of $\ell$ and $\nu$ by ML. $\ell_0 = 0.73$ and $\nu_0 = 2.5$.
Plot of $V_{\ell}$ (left), $V_{\nu}$ (center) and $D_{\ell,\nu}$ (right) with respect to $\epsilon$.}
\label{fig: global_joint_ML}
\end{figure}

In figure \ref{fig: global_joint_CV}, for $\ell_0 = 1.7,\nu_0 = 5$ and for CV, we plot $V_{\ell}$, $V_{\nu}$, $C_{\ell, \nu}$
and $D_{\ell,\nu}$ with respect to $\epsilon$, for $\epsilon \in [0, 0.45 ]$.
We observe the particular case mentioned in figure \ref{fig: Var0surVar45_joint_CV},
in which the estimation of $\ell$ and $\nu$ is improved by the irregularity, but the determinant of their asymptotic covariance
matrix increases, because the absolute value of their asymptotic covariance decreases. This particular case is again a confirmation
that the criteria $V_{\ell}$ and $V_{\nu}$ can be insufficient for evaluating the impact of the irregularity on the estimation,
in a case of joint estimation.

\begin{figure}[]
\centering
 \hspace*{-2cm}

\begin{tabular}{c c}
\includegraphics[width=6cm,angle=0]{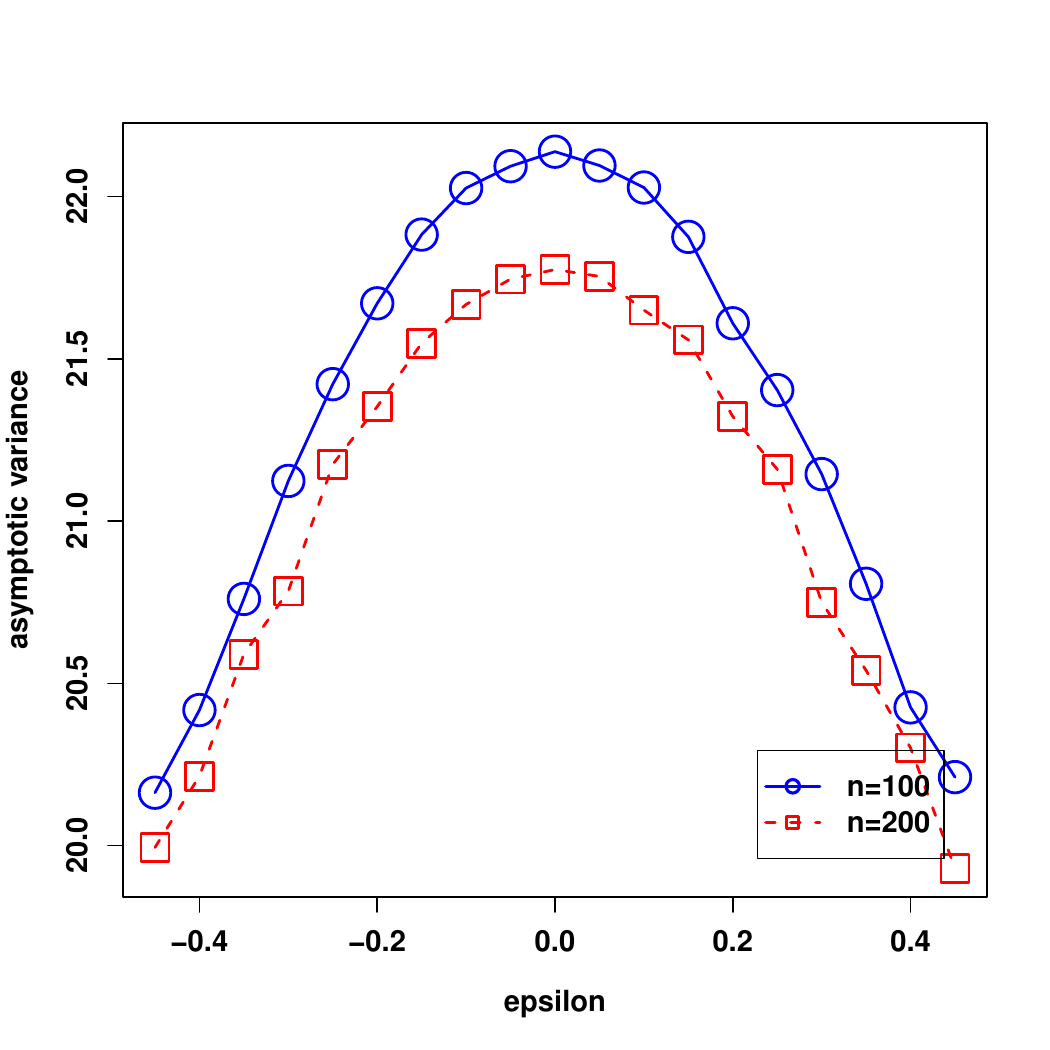} &
\includegraphics[width=6cm,angle=0]{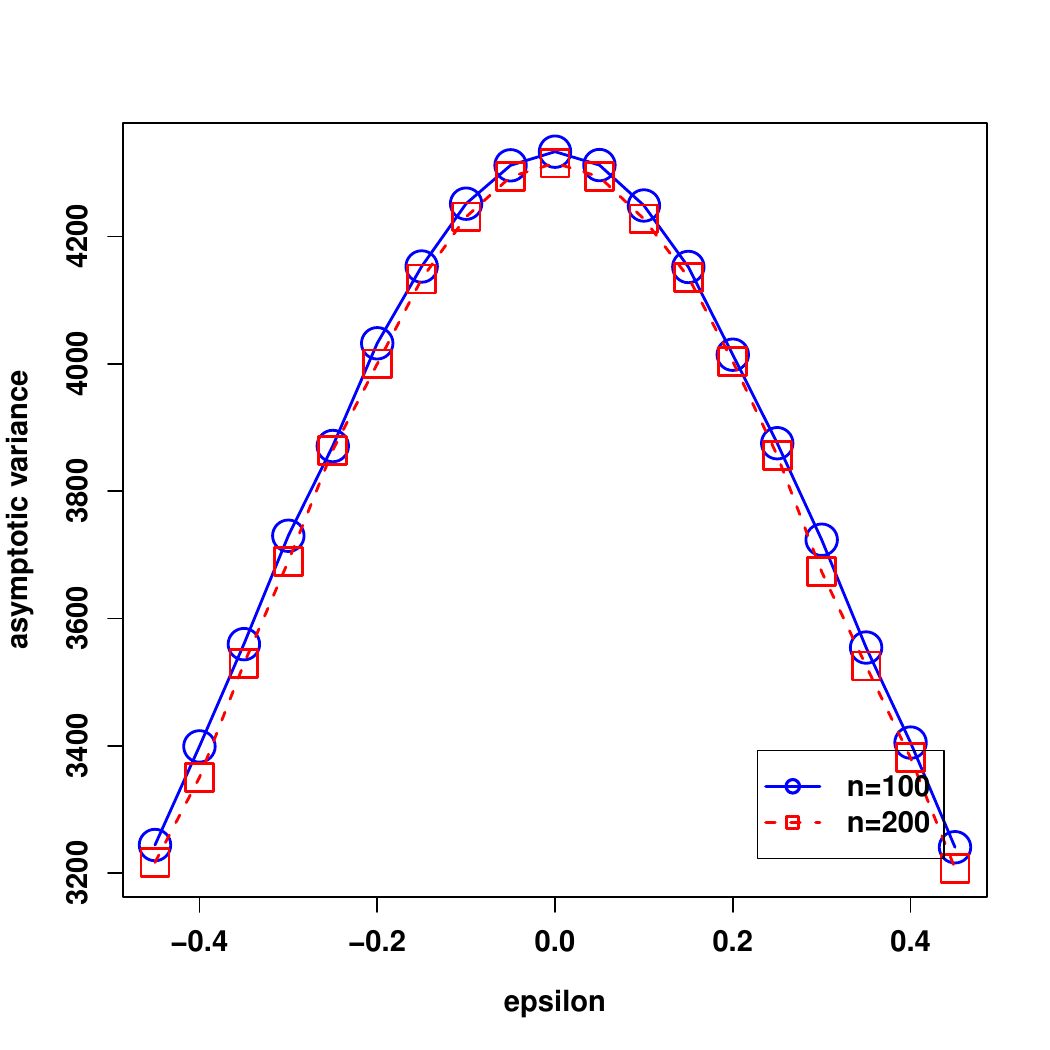} \\
\includegraphics[width=6cm,angle=0]{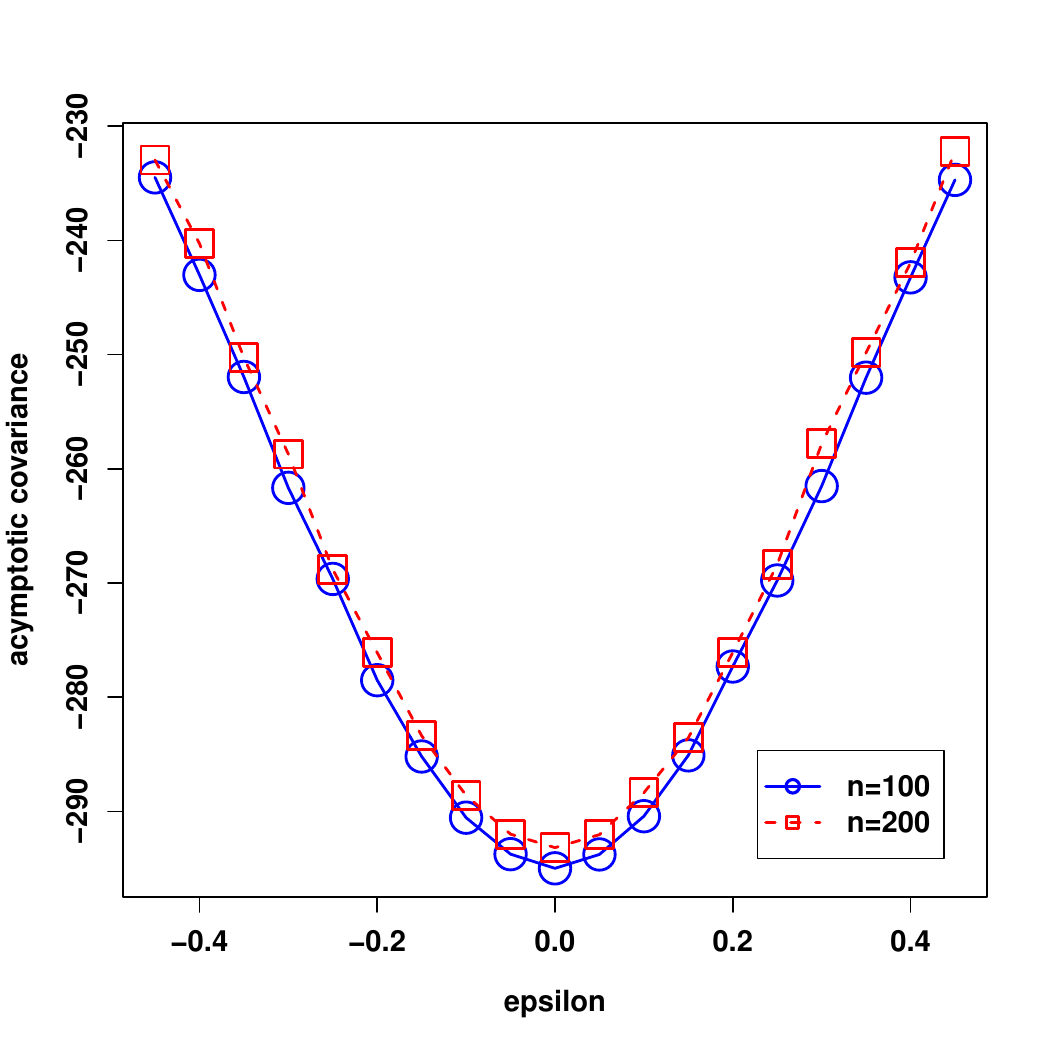} &
\includegraphics[width=6cm,angle=0]{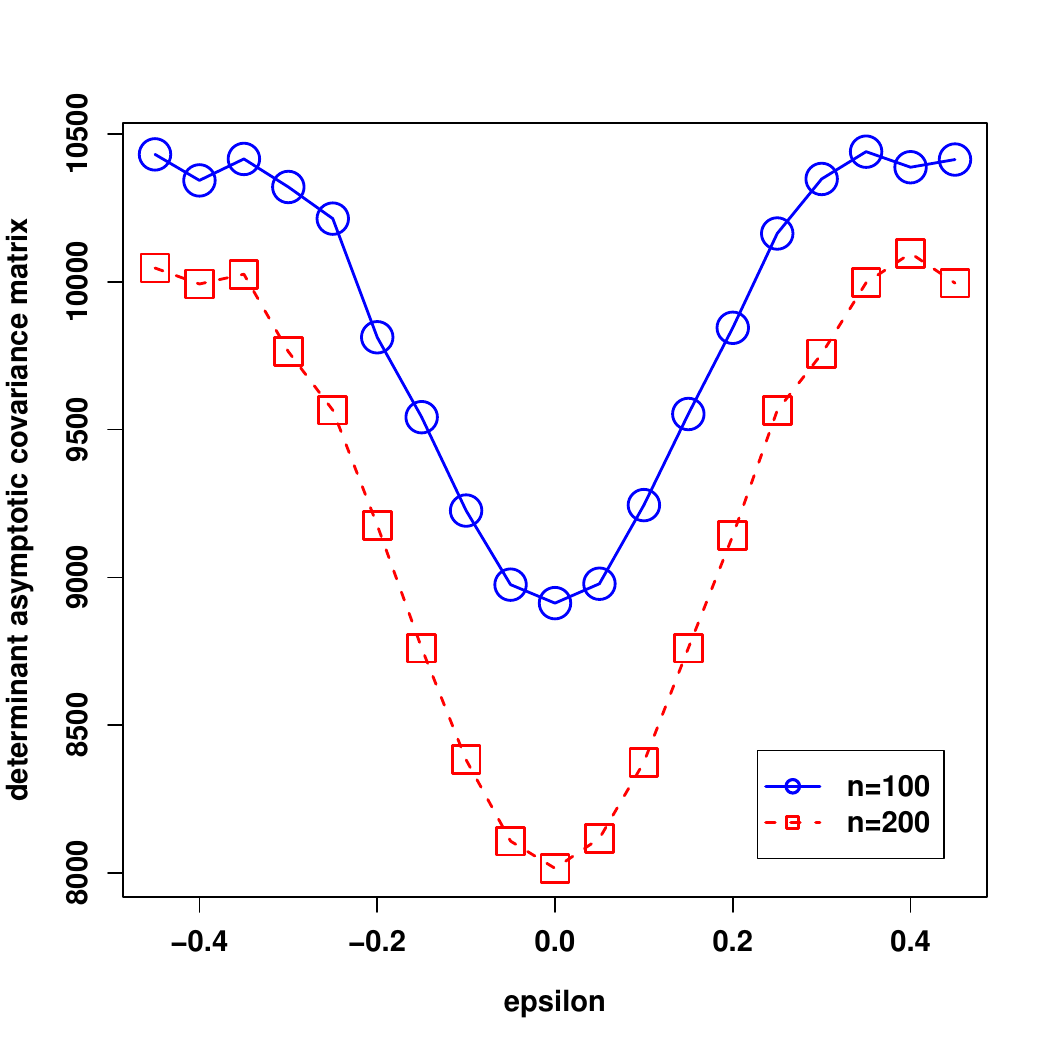} 
\end{tabular}
\caption{Joint estimation of $\ell$ and $\nu$ by CV. $\ell_0 = 1.7$ and $\nu_0 = 5$.
Plot of $V_{\ell}$ (top-left), $V_{\nu}$ (top-right), $C_{\ell,\nu}$ (bottom-left) and $D_{\ell,\nu}$ (bottom-right) with respect to $\epsilon$.}
\label{fig: global_joint_CV}
\end{figure}

\subsection{Discussion}

We have seen that local perturbations of the regular grid can damage both the ML and the CV estimation
(figure \ref{fig: d2surVal0_hatNu}). The CV estimation can even be damaged for strong perturbations
of the regular grid (figure \ref{fig: Var0surVar45_hatLc}). As we have discussed, 
our interpretation is that perturbing the regular grid creates LOO
errors with heterogeneous variances in \eqref{eq: thetaCV}, increasing the variance of the CV estimator minimizing their sum of squares.

Our main conclusion is that strong perturbations of the regular grid ($\epsilon = 0.45$) are beneficial to the ML
estimation in all the cases we have addressed (figures \ref{fig: Var0surVar45_hatLc}, \ref{fig: Var0surVar45_hatNu}, \ref{fig: Var0surVar45_joint_ML}).
Furthermore, ML is shown to be the preferable estimator in the well-specified case addressed here.
This main result is in agreement with the references \cite{ISDSTK,SSDUIAF,Jeannee2008geostatistical} discussed in section
\ref{section: introduction}. The global conclusion is that using groups of observation points with small spacing, compared to
the observation point density in the prediction domain, is beneficial for estimation.

Notice also that, in the reference \cite{Bachoc2013cross} addressing the misspecified case, it is shown that
using a sparse regular grid of observation points, compared to a sampling
with $iid$ uniform observation points, strongly damages the CV estimation, and has considerably less influence on the ML estimation.
This result is also an argument against using only evenly spaced observation points, from an estimation point of view.

\section{Analysis of the Kriging prediction}   \label{section: analysis_prediction}

The asymptotic analysis of the influence of the spatial sampling on the covariance parameter estimation being complete, we now address the
case of the Kriging prediction error, and its interaction with the covariance function estimation. In short words, we study
Kriging prediction with estimated covariance parameters \cite{ISDSTK}.

In subsection \ref{subsection: inf_hyp_miss}, we show that any fixed, constant, covariance function error
has a non-zero asymptotic impact on the prediction error. This fact is interesting in that the conclusion is different
in a fixed-domain asymptotic context, for which we have discussed in section \ref{section: introduction} that there exist
non-microergodic covariance parameters that have no asymptotic influence on prediction.
In subsection \ref{subsection: inf_hyp_est}, we show that, in the expansion-domain asymptotic context we address, the
covariance function estimation procedure has, however, no impact on the prediction error, as long as it is consistent.
Thus, the prediction error is a new criterion for the spatial sampling, that is independent of the estimation criteria
we address in section \ref{section: numerical_study}. In subsection \ref{subsection: impact_sampling_prediction}, we study numerically,
in the case of the Mat\'ern covariance function, the impact of the regularity parameter $\epsilon$ on the mean square prediction error
on the prediction domain.

\subsection{Influence of covariance parameter misspecification on prediction}  \label{subsection: inf_hyp_miss}

In proposition \ref{prop: inf_hyp_pred}, we show that the misspecification of correlation parameters has an asymptotic influence on the prediction errors.
Indeed, the difference of the asymptotic Leave-One-Out mean square errors,
between incorrect and correct covariance parameters, is lower and upper bounded by
finite positive constants times the integrated square difference between the two correlation functions.

\begin{prop} \label{prop: inf_hyp_pred}
 Assume that condition \ref{cond: Ktheta} is satisfied and that for all $\theta \in \Theta$, $K_{\theta}(0)=1$. 

Let, for $1 \leq i \leq n$, $\hat{y}_{i,\theta} \left(y_{-i}\right)  := \EE_{\theta|X} \left( y_i | y_1,...,y_{i-1},y_{i+1},...,y_n \right)$ be the
Kriging Leave-One-Out prediction of $y_i$ with covariance-parameter $\theta$. We then denote
\[
 D_{p} (\theta,\theta_0) := \EE \left[ \frac{1}{n} \sum_{i=1}^n \left\{ y_i -  \hat{y}_{i,\theta} (y_{-i}) \right\}^2 \right]
- \EE \left[ \frac{1}{n} \sum_{i=1}^n \left\{ y_i -  \hat{y}_{i,\theta_0} (y_{-i}) \right\}^2 \right].
\]
Then there exist constants $0 < A < B < +\infty$ so that, for $\epsilon=0$,
\[
A \sum_{v \in \ZZ^d}  \left\{ K_{\theta}(v) - K_{\theta_0}(v) \right\}^2
\leq \underset{ n \to + \infty }{\underline{\lim}} D_p(\theta,\theta_0)
\]
and
\[
\underset{ n \to + \infty }{\overline{\lim}} D_p(\theta,\theta_0)
\leq B \sum_{v \in \ZZ^d}  \left\{ K_{\theta}(v) - K_{\theta_0}(v) \right\}^2.
\]
For $\epsilon \neq 0$, we denote $D_{\epsilon} = \cup_{v \in \ZZ^d \backslash 0} \left(v + \epsilon C_{S_X}\right)$,
with $C_{S_X} = \left\{ t_1-t_2,t_1 \in S_X, t_2 \in S_X \right\}$. Then
\[
A \int_{D_{\epsilon}}  \left\{ K_{\theta}(t) - K_{\theta_0}(t) \right\}^2 dt
\leq \underset{ n \to + \infty }{\underline{\lim}} D_p(\theta,\theta_0)
\]
and
\[
\underset{ n \to + \infty }{\overline{\lim}} D_p(\theta,\theta_0)
\leq B \int_{ D_{\epsilon}}  \left\{ K_{\theta}(t) - K_{\theta_0}(t) \right\}^2 dt.
\]

\end{prop}
\begin{proof}
The minoration is proved in the proof of proposition \ref{prop: consistance_CV}. The majoration is obtained with similar techniques.
\end{proof}

\subsection{Influence of covariance parameter estimation on prediction} \label{subsection: inf_hyp_est}

In proposition \ref{prop: infl_est_pred}, proved in section \ref{section: proof_prediction_independente_estimation},
we address the influence of covariance parameter estimation on prediction.

\begin{prop} \label{prop: infl_est_pred}
Assume that condition \ref{cond: Ktheta} is satisfied and that the Gaussian process $Y$, with covariance function $K_{\theta_0}(t)$, yields almost surely continuous trajectories.
Assume also that for every $\theta \in \Theta$, $1 \leq i \leq p$,
$\frac{\partial}{\partial \theta_i} K_{\theta}(t)$
is continuous with respect to $t$.
Let $\hat{Y}_{\theta}(t)$ be the Kriging prediction of the
Gaussian process $Y$ at $t$, under correlation function $K_{\theta}$ and given the observations $y_1,...,y_n$.
For any $n$, let $N_{1,n}$ so that $ N_{1,n}^d \leq n < (N_{1,n}+1)^d$. Define
\begin{equation} \label{eq: def_E_eps_theta}
E_{\epsilon,\theta} := \frac{1}{N_{1,n}^d} \int_{[0,N_{1,n}]^d} \left( \hat{Y}_{\theta}(t) - Y(t) \right)^2 dt.
\end{equation}
Consider a consistent estimator $\hat{\theta}$ of $\theta_0$. Then
\begin{equation} \label{eq: no_inf_est_on _pred}
| E_{\epsilon,\theta_0} - E_{\epsilon,\hat{\theta}} | = o_p(1). 
\end{equation}
Furthermore, there exists a constant $A >0$ so that for all
$n$,
\begin{equation} \label{eq: no_vanishing_pred_error}
\EE \left( E_{\epsilon,\theta_0} \right) \geq A.
\end{equation}
\end{prop}

In proposition \ref{prop: infl_est_pred}, the condition that the Gaussian process $Y$ yields continuous trajectories is not restrictive in practice, and can be checked using e.g. \cite{Adler81}.
In proposition \ref{prop: infl_est_pred}, we show that the mean square prediction error,
over the observation domain, with a consistently estimated covariance parameter, is asymptotically equivalent to the corresponding error
when the true covariance parameter is known.
Furthermore, the mean value of this prediction error with the true covariance parameter does not vanish when
$n \to + \infty$. This is intuitive because the density of observation points in the domain is constant.

Hence, expansion-domain asymptotics yields a situation in which the estimation error goes to zero, but the prediction error
does not, because the prediction domain increases with the number of observation points. Thus, increasing-domain asymptotic context
enables us to address the prediction and estimation problems separately, and the conclusions on the estimation problem are fruitful, as we have seen in sections
\ref{section: consistency_asymptotic_normality} and \ref{section: numerical_study}. However, this context does not enable us to study theoretically all
the practical aspects of the joint problem of prediction with estimated covariance parameters. For instance, the impact of the estimation method
on the prediction error is asymptotically zero under this theoretical framework, and using a constant proportion of the observation points
for estimation rather than prediction cannot decrease the asymptotic prediction error with estimated covariance parameters.

The two aforementioned practical problems would benefit from an asymptotic framework that would fully reproduce them, by giving a stronger
impact to the estimation on the prediction. Possible candidates for this framework are the mixed increasing-domain asymptotic framework,
addressed for instance in \cite{Lahiri2003central} and \cite{Lahiri2004asymptotic}, and discussed in remark \ref{rem: Lahiri},
and fixed-domain asymptotics.
In both frameworks, the estimation error, with respect to the number of observation points, is larger and the prediction error is smaller,
thus giving hope for more impact of the estimation on the prediction. Nevertheless, even in fixed-domain asymptotics, notice that in \cite{Putter2001}
and referring to \cite{VanDerVaart96maximum},
it is shown that, for the particular case of the tensor product exponential covariance function in two dimensions,
the prediction error, under covariance parameters estimated by ML, is asymptotically equal to the prediction error
under the true covariance parameters. This is a particular case in which estimation has no impact on prediction, even under
fixed-domain asymptotics.

\subsection{Analysis of the impact of the spatial sampling on the Kriging prediction}  \label{subsection: impact_sampling_prediction}

In this subsection \ref{subsection: impact_sampling_prediction}, we study the prediction
mean square error $\EE \left( E_{\epsilon,\ell_0,\nu_0} \right)$ of proposition \ref{prop: infl_est_pred},
as a function of $\epsilon$, $\ell_0$ and $\nu_0$, for the one-dimensional Mat\'ern model, and for large $n$. This function is independent of the estimation,
as we have seen, so there is now no point in distinguishing between ML and CV.
In the following figures, the function $\EE \left( E_{\epsilon,\ell_0,\nu_0} \right)$ is approximated by the average of $iid$ realizations of its conditional mean value given $X=x$,
\[
\frac{1}{n} \int_0^n \left( 1 - r_{\ell_0,\nu_0}^t(t) R^{-1}_{\ell_0,\nu_0} r_{\ell_0,\nu_0}(t) \right) dt,
\]
where $ \left( r_{\ell_0,\nu_0}(t) \right)_i = K_{\ell_0,\nu_0} ( i + \epsilon x_i -t )$
and $ \left( R_{\ell_0,\nu_0} \right)_{i,j} = K_{\ell_0,\nu_0} ( i - j + \epsilon [x_i -x_j] )$.

On figure \ref{fig: prediction_2d}, we plot the ratio of the mean square prediction error $\EE \left( E_{\epsilon,\ell_0,\nu_0} \right)$,
between $\epsilon=0$ and $\epsilon=0.45$, as a function of $\ell_0$ and $\nu_0$, for $n=100$ (we observed the same results for $n=50$). We see that this ratio is always smaller than one,
meaning that strongly perturbing the regular grid always increases the prediction error. This result is in agreement with the common
practices of using regular, also called space filling, samplings for optimizing the Kriging predictions with known covariance parameters, as illustrated
in figure $3$ of \cite{SSDUIAF}. Similarly, the widely used prediction-oriented maximin and minimax designs (see e.g chapter 5 in \cite{TDACE})
yield evenly spaced observation points.

\begin{figure}[]
\centering
 \hspace*{-2cm}

\begin{tabular}{c}
\includegraphics[width=8cm,angle=0]{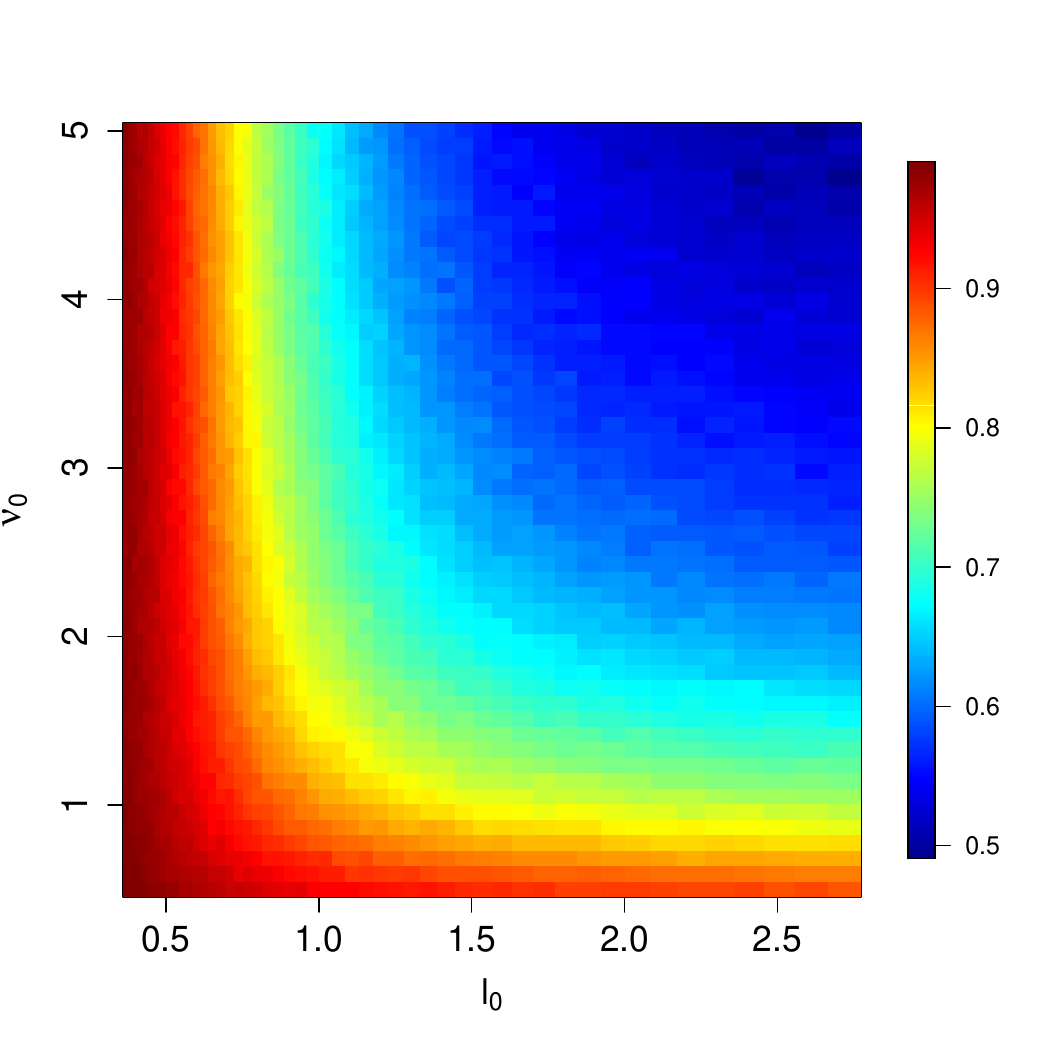} 
\end{tabular}
\caption{Ratio of the mean square prediction error $\EE \left( E_{\epsilon,\ell_0,\nu_0} \right)$ in proposition \ref{prop: infl_est_pred},
between $\epsilon=0$ and $\epsilon=0.45$, as a function of $\ell_0$ and $\nu_0$, for $n=100$.}
\label{fig: prediction_2d}
\end{figure}

In figure \ref{fig: prediction_1d}, we fix the true covariance parameters $\ell_0 = 0.5$, $\nu_0=2.5$, and we study the variations
with respect to $\epsilon$ of the asymptotic variance of the ML estimation of $\nu$, when $\ell_0$ is known (figure \ref{fig: varAss_hatNu_lc=0.5_nu=2.5}),
and of the prediction mean square error $\EE \left( E_{\epsilon,\ell_0,\nu_0} \right)$, for $n=50$
and $n=100$. The results are the same for $n=50$
and $n=100$.
We first observe that $\EE \left( E_{\epsilon,\ell_0,\nu_0} \right)$ is globally an increasing
function of $\epsilon$. In fact, we observe the same global increase of $\EE \left( E_{\epsilon,\ell_0,\nu_0} \right)$, for $n=50$ and $n=100$, with respect to $\epsilon$,
for all the values $(0.5,5)$, $(2.7,1)$, $(0.5,2.5)$, $(0.7,2.5)$, $(2.7,2.5)$, $(0.73,2.5)$ and $(1.7,5)$,
for $(\ell_0,\nu_0)$, that we have studied in section \ref{section: numerical_study}. This is again a confirmation that, in the increasing-domain asymptotic
framework treated here,
evenly spaced observations perform best for prediction. This conclusion corresponds to the common practice of using space filling samplings for Kriging prediction
with known covariance parameters.

The second conclusion than can be drawn for figure \ref{fig: prediction_1d} is that there is independence between estimation (ML in this case) and prediction. Indeed, the estimation error first increases and then decreases with respect to $\epsilon$, while the prediction error
globally decreases. Hence, in figure \ref{fig: prediction_1d}, the regular grid still gives better prediction, although it
leads to less asymptotic variance than mildly irregular samplings. 
Therefore, there is no simple antagonistic relationship between the
impact of the irregularity of the spatial sampling on estimation and on prediction.

\begin{figure}[]
\centering
 \hspace*{-2cm}

\begin{tabular}{c c}
\includegraphics[width=8cm,angle=0]{ML_varAss_hatNu_lc=0p5_nu=2p5.pdf}  &
\includegraphics[width=8cm,angle=0]{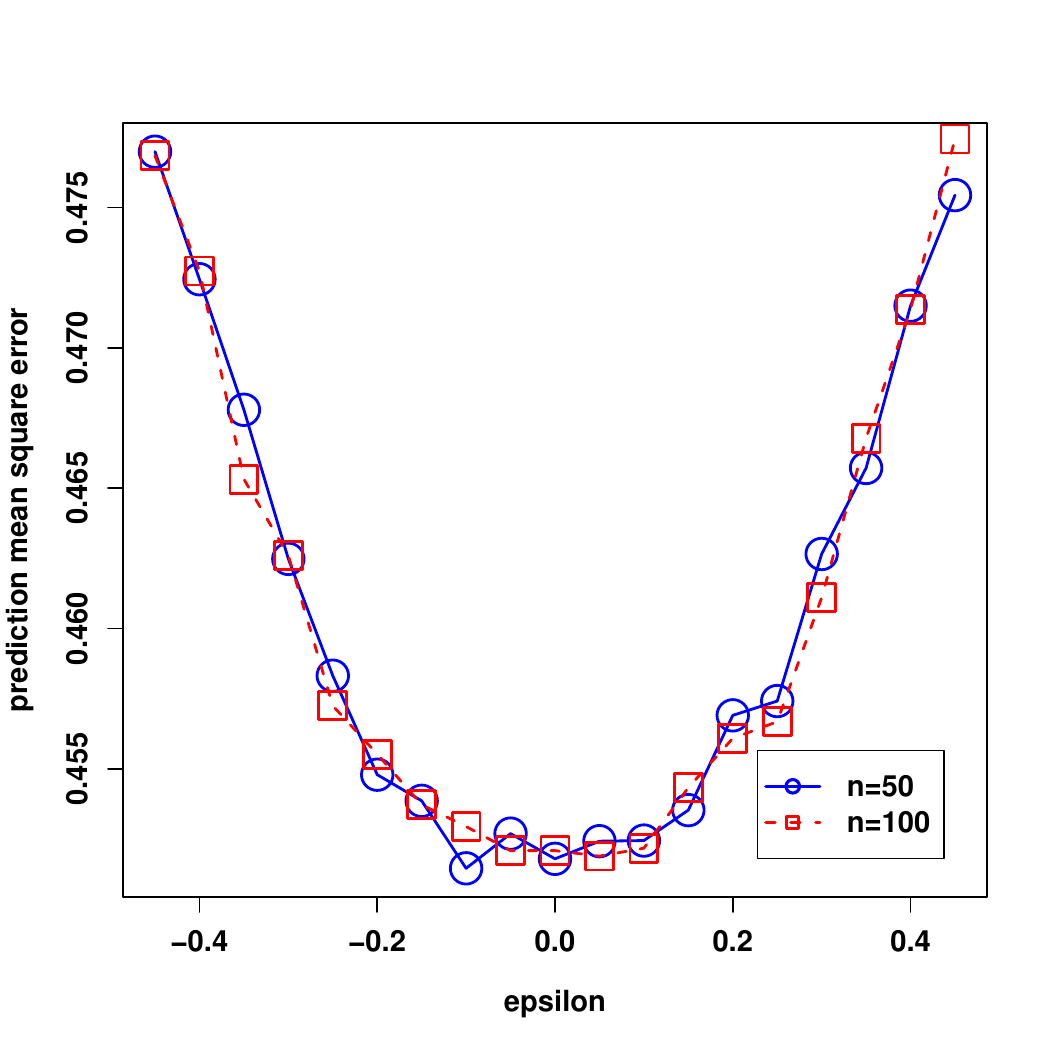}
\end{tabular}
\caption{$\ell_0 = 0.5$, $\nu_0=2.5$. Left: asymptotic variance for the ML estimation
of $\nu$, when $\ell$ is known, as a function of $\epsilon$ (same setting as in figure \ref{fig: varAss_hatNu_lc=0.5_nu=2.5}).
Right: prediction mean square error $\EE \left( E_{\epsilon,\ell_0,\nu_0} \right)$ in proposition \ref{prop: infl_est_pred} as a function of $\epsilon$.}
\label{fig: prediction_1d}
\end{figure}

\FloatBarrier

\section{Conclusion}

We have considered an increasing-domain asymptotic framework to address the influence of the irregularity of the spatial sampling on the
estimation of the covariance parameters. This framework is based on a random sequence of observation points, for which the
deviation from the regular grid is
controlled by a single scalar regularity parameter $\epsilon$.

We have proved consistency and asymptotic normality for the ML and CV estimators,
under rather minimal conditions. The asymptotic covariance matrices are deterministic functions of the regularity parameter only. Hence, they are the natural
tool to assess the influence of the irregularity of the spatial sampling on the ML and CV estimators.

This is carried out by means of an exhaustive study of the Mat\'ern model.
It is shown that mildly perturbing the regular grid can damage both ML and CV estimation, and that CV estimation can also be damaged
when strongly perturbing the regular grid.
However, we put into evidence that strongly perturbing the regular grid always
improves the ML estimation, which is a more efficient estimator than CV, in the well-specified case addressed here. Hence, we confirm the conclusion
of \cite{ISDSTK} and \cite{SSDUIAF} that using groups of observation points with small spacing, compared to the observation density
in the observation domain, improves the covariance function estimation. In geostatistics, such groups of points
are also added to regular samplings in practice \citep{Jeannee2008geostatistical}.

We have also studied the impact of the spatial sampling on the prediction error. Regular samplings were shown to be the most efficient
as regards to this criterion. This is in agreement with, e.g. \cite{SSDUIAF} and with \cite{pronzato12design}
where samplings for Kriging prediction with known covariance parameters are selected by optimizing a space filling criterion. 

The ultimate goal of a Kriging model is prediction with estimated covariance parameters. Hence, efficient samplings for
this criterion must address two criteria that have been shown to be antagonistic. In the literature, there seems to be
a commonly admitted practice for solving this issue \citep{SSDUIAF,pronzato12design}. Roughly speaking, for selecting
an efficient sampling for prediction with estimated covariance parameters, one may select a regular sampling for prediction
with known covariance parameters and augment it with a sampling for estimation (with closely spaced observation points).
The proportion of points for the two samplings is optimized in the two aforementioned references by optimizing a criterion
for prediction with estimated covariance parameters. This criterion is more expensive to compute, but is not optimized in a large
dimensional space. In \cite{SSDUIAF,pronzato12design}, the majority of the observation points belong to the regular sampling
for prediction with known covariance function. This is similar in the geostatistical community \citep{Jeannee2008geostatistical},
where regular samplings, augmented with few closely spaced observation points, making the inputs vary mildly, are used.
In view of our theoretical and practical results of sections \ref{section: numerical_study} and \ref{section: analysis_prediction},
we are in agreement with this method for building samplings for prediction with estimated covariance parameters.

An important limitation we see, though, in the expansion-domain asymptotic framework we address, is that prediction with estimated
covariance parameters corresponds asymptotically to prediction with known covariance function. Said differently, the proportion of observation points
addressing estimation, in the aforementioned trade-off, would go to zero. As we discuss after proposition \ref{prop: infl_est_pred},
mixed increasing-domain or fixed-domain asymptotics could give more importance to the estimation problem, compared to the problem of
predicting with known covariance function. 

Since our sampling is built by perturbing a regular grid, it may not enable us to address very high degree of irregularity,
such as the one obtained from $iid$ random observation points with non-uniform distribution \citep{Lahiri2003central,Lahiri2004asymptotic}.
Nevertheless, using directly $iid$ random observation points makes it difficult to parameterize the irregularity of the sampling,
which have enabled us to perform exhaustive numerical studies, in which comparing the irregularity of two samplings was
unambiguous. A parameterized sampling family we see as interesting would be samplings in which the observation points
are $iid$ realizations stemming from a given distribution, but conditionally to the fact that the minimum distance between two different
observation points is larger than a regularity parameter $\epsilon'$. In this framework, $\epsilon'=0$ would correspond to the most irregular sampling.
Increasing $\epsilon'$ would correspond to increase the regularity of the sampling. It seems to us that convergence in distribution
similar to propositions \ref{prop: normaliteML} and \ref{prop: normaliteCV} could be obtained in this setting, using similar methods, though technical
aspects may be addressed differently.

The CV criterion we have studied is a mean square error criterion.
This is a classical CV
criterion that is used, for instance, in \cite{TDACE} when the CV and ML estimations of the covariance parameters are compared.
We have shown that this CV estimator can have a considerably larger asymptotic variance than the ML estimator, on the one hand, and can be
sensitive to the irregularity of the spatial sampling, on the other hand. Although this estimator performs better than ML in cases of
model misspecification \citep{Bachoc2013cross},
further research may aim at studying alternative CV criteria that would have a better performance in the
well-specified case. 

Other CV criteria are proposed in the literature, for instance the LOO log-predictive probability in \cite{GPML} and the
Geisser's predictive mean square error in \cite{PACHGP}. It would be interesting to study, in the framework of this paper, the increasing-domain
asymptotics for these estimators and the influence of the irregularity of the spatial sampling. 

In section \ref{section: numerical_study}, we pointed out that, when the spatial sampling is irregular,
the mean square error CV criterion could be composed of LOO errors  with heterogeneous variances, which increases the CV estimation variance. Methods to
normalize the LOO errors would be an interesting research direction to explore.

\section*{Acknowledgments}
The author would like to thank his advisors Josselin Garnier (University Paris Diderot)
and Jean-Marc Martinez (French Alternative Energies and Atomic Energy
Commission - Nuclear Energy Division at CEA-Saclay, DEN, DM2S, STMF, LGLS) for their advices and suggestions.
The author acknowledges fruitful reports from the two anonymous reviewers, that helped to improve an earlier version
of the manuscript.
The author presented the contents of the manuscript at two workshops of the ReDICE consortium, where he
benefited from constructive comments and suggestions.

\appendix

\section{Proofs for section  \ref{section: consistency_asymptotic_normality}}  \label{section: appendix_proof_normality}

In the proofs, we distinguish three probability spaces.

$(\Omega_X,\mathcal{F}_X,P_X)$ is the probability space associated with the random perturbation of the regular grid.
$(X_i)_{i \in \NN^*}$ is a sequence of $iid$ $S_X$-valued
random variables defined on $(\Omega_X,\mathcal{F}_X,P_X)$, with distribution $\mathcal{L}_X$. We denote by $\omega_X$ an element of $\Omega_X$.

$(\Omega_Y,\mathcal{F}_Y,P_Y)$ is the probability space associated with the Gaussian process. $Y$ is a centered Gaussian process with
covariance function $K_{\theta_0}$ defined on $(\Omega_Y,\mathcal{F}_Y,P_Y)$. We denote by $\omega_Y$ an element of $\Omega_Y$.

$(\Omega, \mathcal{F},\mathbb{P})$ is the product space $(\Omega_X \times \Omega_Y , \mathcal{F}_X \otimes \mathcal{F}_Y ,P_X \times P_Y)$.
We denote by $\omega$ an element of $\Omega$.

All the random variables in the proofs can be defined relatively to the product space $(\Omega, \mathcal{F},\mathbb{P})$.
Hence, all the probabilistic statements in the proofs hold with respect to this product space, unless it is stated otherwise. 

In the proofs, when $(f_n)_{n \in \NN^*}$ is a sequence of real functions of $X=(X_i)_{i=1}^n$, $f_n$ is also a sequence of real random variables on
$(\Omega_X,\mathcal{F}_X,P_X)$. When we write that $f_n$ is bounded uniformly in $n$ and $x$, we mean that there exists a finite constant $K$ so that
 $\sup_{n} \sup_{x \in S_X^n} |f_n(x)| \leq K$.
We then have that $f_n$ is bounded $P_X$-a.s., i.e. $\sup_n f_n \leq K$ for a.e. $\omega_X \in \Omega_X$.
We may also write that $f_n$ is lower-bounded uniformly in $n$ and $x$ when there exists $a > 0$ so that
$\inf_{n} \inf_{x \in S_X^n} f_n(x) \geq a$.
When $f_n$ also depends on $\theta$, we say that
$f_n$ is bounded uniformly in $n$, $x$ and $\theta$ when $\sup_{\theta \in \Theta} f_n$ is bounded uniformly in $n$ and $x$.
We also say that $f_n$ is lower-bounded uniformly in $n$, $x$ and $\theta$ when $\inf_{\theta \in \Theta} f_n$ is
lower-bounded uniformly in $n$ and $x$.

When we write that $f_n$ converges to zero
uniformly in $x$, we mean that $\sup_{x \in S_X^n} |f_n(x)| \to_{n \to + \infty} 0$. One then have that $f_n$
converges to zero $P_X$-a.s. When $f_n$ also depends on $\theta$, we say that
$f_n$ converges to zero uniformly in $n$, $x$ and $\theta$ when $\sup_{\theta \in \Theta} f_n$ converges to zero uniformly in $n$ and $x$.

When $f_n$ is a sequence of real functions of $X$ and $Y$, $f_n$ is also a sequence of real random variables on
$(\Omega,\mathcal{F},\mathbb{P})$. When we say that $f_n$ is bounded in probability conditionally to $X=x$ and uniformly in $x$, we mean that, for every $\delta>0$,
there exist $M$, $N$ so that
$ \sup_{n \geq N} \sup_{x \in S_X^n} \mathbb{P}( |f_n| \geq  M | X = x) \leq \delta$.
One then have that $f_n$ is bounded in probability (defined on the product space).

\subsection{Proof of proposition \ref{prop: consistance_ML}}

\begin{proof} 
We show that there exist sequences of random variables, defined on $(\Omega_X,\mathcal{F}_X,P_X)$,
$D_{\theta,\theta_0}$ and $D_{2,\theta,\theta_0}$ (functions of $n$ and $X$), so that
$\sup_{\theta} \left| \left(L_{\theta} - L_{\theta_0}\right) - D_{\theta,\theta_0} \right| \to_p 0 $ (in probability of the product space)
and  $D_{\theta,\theta_0} \geq B D_{2,\theta,\theta_0}$ $P_X$-a.s. for
a constant $B > 0$.
We then show that there exists $D_{\infty,\theta,\theta_0}$, a deterministic non-negative function of $\theta,\theta_0$ only, so that 
$\sup_{\theta} \left| D_{2,\theta,\theta_0} - D_{\infty,\theta,\theta_0} \right| = o_p\left(1\right)$ and for any $\alpha>0$,
\begin{equation} \label{eq: condition_consistance}
\inf_{ |\theta - \theta_0| \geq \alpha } D_{\infty,\theta,\theta_0} > 0. 
\end{equation}
This implies consistency.

We have $L_{\theta} = \frac{1}{n}\log\left\{\det\left(R_{\theta}\right)\right\} + \frac{1}{n} y^t R_{\theta}^{-1}y$.
The eigenvalues of $R_{\theta}$ and $R_{\theta}^{-1}$ being bounded uniformly
in $n$ and $x$ (lemma \ref{lem: controle_valeurs_propres}),
$\var\left(L_{\theta} | X = x \right)$ converges to $0$ uniformly in $x$, and so $L_{\theta} - \EE\left(L_{\theta} | X\right)$
converges in probability $\mathbb{P}$ to zero.

Then, with $z = R_{\theta_0}^{-\frac{1}{2}} y$,
\begin{eqnarray*}
 \sup_{k \in \{1,...,p\},\theta \in \Theta} \left| \frac{\partial L_{\theta}}{\partial \theta_k} \right|
& = & \sup_{k \in \{1,...,p\},\theta \in \Theta} \frac{1}{n} \left\{ {\rm{Tr}}\left(R_{\theta}^{-1} \frac{\partial R_{\theta}}{\partial \theta_k} \right)
+ z^t R_{\theta_0}^{\frac{1}{2}} R_{\theta}^{-1} \frac{\partial R_{\theta}}{\partial \theta_k} R_{\theta}^{-1} R_{\theta_0}^{\frac{1}{2}} z \right\}  \\
& \leq & \sup_{k \in \{1,...,p\},\theta \in \Theta} \left\{ max\left( \left|\left|R_{\theta}^{-1}\right|\right| \left|\left| \frac{\partial R_{\theta}}{\partial \theta_k}\right|\right|
, ||R_{\theta_0}|| \left|\left|R_{\theta}^{-2}\right|\right| \left|\left| \frac{\partial R_{\theta}}{\partial \theta_k} \right|\right| \right) \right\}
\left( 1 + \frac{1}{n}  |z|^2 \right), \\
\end{eqnarray*}
and is hence bounded in probability conditionally to $X=x$, uniformly in $x$, because of lemma \ref{lem: controle_valeurs_propres}
and the fact that $z \sim \N\left(0,I_n\right)$ given $X=x$.

Because of the simple convergence and the boundedness of the derivatives, $ \sup_{\theta} | L_{\theta}  - \EE\left( L_{\theta} | X \right) | \to_p 0 $. 
We then denote $D_{\theta,\theta_0} := \EE\left(L_{\theta} | X\right) - \EE\left(L_{\theta_0} | X\right)$.
We then have $\sup_{\theta} | \left(L_{\theta} - L_{\theta_0}\right) - D_{\theta,\theta_0} | \to_p 0 $.

We have $\EE\left(L_{\theta} | X\right) = \frac{1}{n} \log\left\{\det\left(R_{\theta}\right)\right\} + \frac{1}{n} {\rm{Tr}}\left( R_{\theta}^{-1} R_{\theta_0} \right)$
and hence, $P_X$-a.s.

\begin{eqnarray*}
 D_{\theta,\theta_0} & = & \frac{1}{n} \log\left\{\det\left(R_{\theta}\right)\right\} + \frac{1}{n} {\rm{Tr}}\left( R_{\theta}^{-1} R_{\theta_0} \right)
- \frac{1}{n} \log\left\{\det\left(R_{\theta_0}\right)\right\} - 1  \\
& = & \frac{1}{n} \sum_{i=1}^n \left[ -\log\left\{ \phi_i\left( R_{\theta_0}^{\frac{1}{2}} R_{\theta}^{-1} R_{\theta_0}^{\frac{1}{2}} \right) \right\}
+ \phi_i\left(R_{\theta_0}^{\frac{1}{2}} R_{\theta}^{-1} R_{\theta_0}^{\frac{1}{2}} \right) -1 \right].
\end{eqnarray*}

Using proposition \ref{prop: minorationValeursPropres} and lemma \ref{lem: controle_valeurs_propres}, there exist $0 < a < b < +\infty$ so that for all $x$, $n$, $\theta$,
$ a < \phi_i\left( R_{\theta_0}^{\frac{1}{2}} R_{\theta}^{-1} R_{\theta_0}^{\frac{1}{2}} \right) < b $. We denote $f\left(t\right) = -\log\left(t\right) + t -1$.
As $f$ is minimal in $1$, $f'\left(1\right)= 0$ and $f''\left(1\right) = 1$, there exists $A>0$ so that, for  $t \in [a,b]$, $f\left(t\right)$ is larger than $A \left(t-1\right)^2$.
Then,
\begin{eqnarray*}
 D_{\theta,\theta_0} & \geq & A \frac{1}{n} \sum_{i=1}^n \left\{ 1 - \phi_i\left( R_{\theta_0}^{\frac{1}{2}} R_{\theta}^{-1} R_{\theta_0}^{\frac{1}{2}} \right) \right\}^2   \\
 & = & A \frac{1}{n} {\rm{Tr}}\left\{ \left( I - R_{\theta_0}^{\frac{1}{2}} R_{\theta}^{-1} R_{\theta_0}^{\frac{1}{2}} \right)^2 \right\} \\
& = & A  \left| R_{\theta}^{-\frac{1}{2}}  \left(R_{\theta} - R_{\theta_0}  \right) R_{\theta}^{-\frac{1}{2}} \right|^2.
\end{eqnarray*}
Then, as the eigenvalues of $R_{\theta}^{-\frac{1}{2}}$ are larger than $c>0$, uniformly in $n$, $x$ and $\theta$, and with
$|MN|^2 \geq \inf_{i} \phi_i^2\left(M\right) |N|^2$ for $M$ symmetric positive, we obtain, for some $B >0$, and uniformly in $n$, $x$ and $\theta$,
\[
 D_{\theta,\theta_0}  \geq B | R_{\theta} - R_{\theta_0} |^2 := B D_{2,\theta,\theta_0}. 
\]

For $\epsilon = 0$, $D_{2,\theta,\theta_0}$ is deterministic and we show that it converges uniformly in $\theta$ to
\begin{equation} \label{eq: proof_cons_ML_Dinfini}
 D_{\infty,\theta,\theta_0} :=  \sum_{v \in \ZZ^d} \left\{ K_{\theta}\left(v\right) - K_{\theta_0}\left(v\right) \right\}^2.
\end{equation}
Indeed,
\[
\left|  |R_{\theta} - R_{\theta_0}|^2 - D_{\infty,\theta,\theta_0} \right| \leq
\frac{1}{n} \sum_{i=1}^n
 \left|  \sum_{j=1}^n  \left\{ K_{\theta}(v_i - v_j) - K_{\theta_0} (v_i - v_j) \right\}^2 -
\sum_{v \in \ZZ^d} \left\{ K_{\theta}\left(v\right) - K_{\theta_0}\left(v\right) \right\}^2  \right|. 
\]
Now, using lemma \ref{lem: sommabilité} and \eqref{eq: controleCov},
$\sup_{\theta \in \Theta, 1 \leq i \leq n} \sum_{v \in (\NN^*)^d} \left\{ K_{\theta}\left(v_i - v\right) - K_{\theta_0}\left(v_i - v\right) \right\}^2$
is bounded by $C < + \infty$, independently of $n$. Let $\delta >0$. Using lemma \ref{lem: sommabilité3} and \eqref{eq: controleCov}, there exists $T < + \infty$ so that  
$\sup_{\theta \in \Theta,1 \leq i \leq n} \sum_{v \in (\NN^*)^d, |v - v_i|_{\infty} \geq T} \left\{ K_{\theta}\left(v_i - v\right) - K_{\theta_0}\left(v_i - v\right) \right\}^2$ is bounded by $\delta$ independently
of $n$. Finally, for any $n$, let $E_n$ be the set of points $v_i$, $1 \leq i \leq n$, so that
$\prod_{k=1}^d \left\{ (v_i)_k - T,...,(v_i)_k+T \right\} \not\subset \left\{ v_i, 1\leq i \leq n \right\}$. The cardinality of $E_n$ is denoted by $n_{1,n}$ and is so that $\frac{n_{1,n}}{n} \to_{n \to + \infty} 0$. 

then,
\begin{eqnarray*}
\left|  |R_{\theta} - R_{\theta_0}|^2 - D_{\infty,\theta,\theta_0} \right| &  \leq &  
\frac{1}{n}  \sum_{1 \leq i \leq n,  v_i \in E_n} 
 \left|  \sum_{j=1}^n  \left\{ K_{\theta}(v_i - v_j) - K_{\theta_0} (v_i - v_j) \right\}^2 -
\sum_{v \in \ZZ^d} \left\{ K_{\theta}\left(v\right) - K_{\theta_0}\left(v\right) \right\}^2  \right| \\
& & + \frac{1}{n}
 \sum_{1 \leq i \leq n, v_i \not\in E_n}
 \left|  \sum_{j=1}^n  \left\{ K_{\theta}(v_i - v_j) - K_{\theta_0} (v_i - v_j) \right\}^2 -
\sum_{v \in \ZZ^d} \left\{ K_{\theta}\left(v\right) - K_{\theta_0}\left(v\right) \right\}^2 
\right| \\
& \leq & \frac{n_{1,n}}{n} (2 C) 
+ \frac{1}{n} \sum_{1 \leq i \leq n, v_i \not\in E_n}
\sum_{v \in (\NN^*)^d, |v - v_i|_{\infty} \geq T} \left\{ K_{\theta}\left(v_i - v\right) - K_{\theta_0}\left(v_i - v\right) \right\}^2 \\
& \leq & \frac{n_1}{n} (2 C) + \delta.
\end{eqnarray*}
Thus, $\left|  |R_{\theta} - R_{\theta_0}|^2 - D_{\infty,\theta,\theta_0} \right|$ is smaller than $2 \delta$ for $n$ large enough.

For $\epsilon = 0$, $D_{\infty,\theta,\theta_0}$ is continuous in $\theta$ because the series of term $\sup_{\theta} | K_{\theta}\left(v\right) |^2$, $v \in \ZZ^d$
is summable using \eqref{eq: controleCov} and lemma \ref{lem: sommabilité}.
Hence, if there exists $\alpha>0$ so that
$ \inf_{ |\theta - \theta_0| \geq \alpha } D_{\infty,\theta,\theta_0} = 0$, we can, using a compacity and continuity argument, have $\theta_{\infty} \neq \theta_0$
so that \eqref{eq: proof_cons_ML_Dinfini} is null. Hence we showed \eqref{eq: condition_consistance}
by contradiction, which shows the proposition for $\epsilon = 0$.

For $\epsilon \neq 0$, $D_{2,\theta,\theta_0} = \frac{1}{n} {\rm{Tr}}\left\{ \left( R_{\theta} - R_{\theta_0} \right)^2 \right\}$.
With fixed $\theta$, using proposition \ref{prop: convergenceTrace}, $D_{2,\theta,\theta_0}$ converges in $P_X$-probability to
$D_{\infty,\theta,\theta_0} := \lim_{n \to \infty} E_X \left( D_{2,\theta,\theta_0} \right)$. The eigenvalues of the
$\frac{\partial R_{\theta}}{\partial \theta_i}$, $1 \leq i \leq n$,
being bounded uniformly in $n$, $\theta$, $x$, the partial derivatives with respect to $\theta$ of $D_{2,\theta,\theta_0}$ are uniformly bounded
in $n$, $\theta$ and $x$.
Hence $\sup_{\theta} | D_{2,\theta,\theta_0} - D_{\infty,\theta,\theta_0} | = o_p\left(1\right)$. Then
\[
 D_{\infty,\theta,\theta_0} = \underset{n \to + \infty}{\lim} \frac{1}{n} \sum_{1 \leq i,j \leq n , i \neq j}
\left[ \int_{ \epsilon C_{S_X} } \left\{K_{\theta}\left(v_i-v_j+t\right) - K_{\theta_0}\left(v_i-v_j+t\right)\right\}^2 f_T\left(t\right) dt \right] +  \left\{K_{\theta}\left(0\right) - K_{\theta_0}\left(0\right)\right\}^2,
\]
with $f_T\left(t\right)$ the probability density function of $\epsilon\left(X_i - X_j\right)$, $i \neq j$. We then show, 
\begin{eqnarray} \label{eq: proof_cons_ML_Dinfini_eps_nonzero}
  D_{\infty,\theta,\theta_0} & = & \sum_{v \in \ZZ^d \backslash 0}
\left[ \int_{ \epsilon C_{S_X} } \left\{K_{\theta}\left(v + t\right) - K_{\theta_0}\left( v+t \right)\right\}^2 f_T\left(t\right) dt \right]
+  \left\{K_{\theta}\left(0\right) - K_{\theta_0}\left(0\right)\right\}^2  \nonumber \\ 
& = & \int_{D_{\epsilon}} \left\{ K_{\theta}\left(t\right) - K_{\theta_0}\left(t\right) \right\}^2 \tilde{f}_T\left(t\right) dt + \left\{K_{\theta}\left(0\right) - K_{\theta_0}\left(0\right)\right\}^2,
\end{eqnarray}
where $\tilde{f}_T$ is a positive-valued function, almost surely with respect to the Lebesgue measure on $D_{\epsilon}$.
As $\sup_{\theta} | K_{\theta}\left(t\right) |^2$ is summable
on $D_{\epsilon}$, using \eqref{eq: controleCov}, $D_{\infty,\theta,\theta_0}$ is continuous.
Hence, if there exists $\alpha>0$ so that
$ \inf_{ |\theta - \theta_0| \geq \alpha } D_{\infty,\theta,\theta_0} = 0$, we can, using a compacity and continuity argument, show that there exists
$\theta_{\infty} \neq \theta_0$
so that \eqref{eq: proof_cons_ML_Dinfini_eps_nonzero} is null. Hence we proved \eqref{eq: condition_consistance}
by contradiction which proves the proposition for $\epsilon  \neq 0$.

\end{proof}

\subsection{proof of Proposition \ref{prop: normaliteML}}
\begin{proof}
For $1  \leq i , j \leq p$, we use proposition \ref{prop: convergenceTrace} to show that
$\frac{1}{n} {\rm{Tr}}\left( R^{-1} \frac{\partial R}{\partial \theta_i} R^{-1} \frac{\partial R}{\partial \theta_j} \right)$ has a $P_X$-almost sure limit as $n \to + \infty$.

We calculate $\frac{ \partial }{ \partial \theta_i } L_{\theta} = \frac{1}{n} \left\{ {\rm{Tr}}\left( R_{\theta}^{-1} \frac{ \partial R_{\theta} }{ \partial \theta_i } \right)
- y^t R_{\theta}^{-1} \frac{ \partial R_{\theta} }{ \partial \theta_i} R_{\theta}^{-1} y \right\} $. We use proposition \ref{prop: lindeberg_presque_sur}
with $M_i = R_{\theta}^{-1} \frac{ \partial R_{\theta} }{ \partial \theta_i }$ and
$N_i = - R_{\theta}^{-1} \frac{ \partial R_{\theta} }{ \partial \theta_i} R_{\theta}^{-1}$, together with proposition \ref{prop: convergenceTrace},
to show that
\[
 \sqrt{n} \frac{ \partial }{ \partial \theta } L_{\theta_0} \to_{\L} \N\left( 0 , 4 \Sigma_{ML} \right).
\]
We calculate
\begin{eqnarray*}
\frac{ \partial^2 }{ \partial \theta_i \partial \theta_j} L_{\theta_0} & = &
\frac{1}{n} {\rm{Tr}}\left( - R^{-1} \frac{\partial R}{\partial \theta_i} R^{-1} \frac{\partial R}{\partial \theta_j}
+ R^{-1} \frac{\partial^2 R}{\partial \theta_i \partial \theta_j} \right)  \\
& + &  \frac{1}{n} y^t \left( 2 R^{-1} \frac{\partial R}{\partial \theta_i} R^{-1} \frac{\partial R}{\partial \theta_j} R^{-1}
- R^{-1} \frac{\partial^2 R}{\partial \theta_i \partial \theta_j} R^{-1} \right) y.
\end{eqnarray*}

Hence, using proposition \ref{prop: convergence_forme_quadratique}, $\frac{ \partial^2 }{ \partial \theta^2 } L_{\theta_0} $
converges to $2 \Sigma_{ML}$ in the mean square sense (on the product space).

Finally, $\frac{ \partial^3 }{ \partial \theta_i \partial \theta_j \partial \theta_k } L_{\tilde{ \theta }} $
can be written as $\frac{1}{n} \left\{ {\rm{Tr}}\left( M_{\tilde{\theta}}\right) + z^t N_{\tilde{\theta}} z \right\}$, where
$M_{\tilde{\theta}}$ and $N_{\tilde{\theta}}$ are sums of matrices
of $\mathcal{M}_{\tilde{\theta}}$ (proposition \ref{prop: convergenceTrace}) and
where $z$ depends on $X$ and $Y$ and
$\L\left(z | X\right) = \N\left(0,I_n\right)$. Hence, the singular values of
$M_{\tilde{\theta}}$ and $N_{\tilde{\theta}}$ are bounded uniformly in
$\tilde{\theta}$, $n$ and $x$, and so
$ \sup_{i,j,k\tilde{\theta}} \frac{ \partial^3 }{ \partial \theta_i \partial \theta_j \partial \theta_k } L_{\tilde{ \theta }} $
is bounded by $a + b \frac{1}{n} |z|^2$, with constant $a,b < + \infty$ and
is hence bounded in probability. Hence we apply proposition \ref{prop: normaliteEstimateur} to conclude.

\end{proof}

\subsection{Proof of proposition \ref{prop: info>0_ML}}
\begin{proof}
We firstly prove the proposition in the case $p=1$, when $\Sigma_{ML}$ is a scalar. We then show how to generalize the proposition to the case $p >1$.

For $p=1$ we have seen that
$\frac{1}{n} {\rm{Tr}}\left( R_{\theta_0}^{-1} \frac{\partial R_{\theta_0}}{\partial \theta}  R_{\theta_0}^{-1}
\frac{\partial R_{\theta_0}}{\partial \theta}  \right) \to_{P_X} 2 \Sigma_{ML}$. Then
\begin{eqnarray*}
\frac{1}{n} {\rm{Tr}}\left( R_{\theta_0}^{-1} \frac{\partial R_{\theta_0}}{\partial \theta}  R_{\theta_0}^{-1} \frac{\partial R_{\theta_0}}{\partial \theta}  \right) = 
\frac{1}{n} {\rm{Tr}}\left(R_{\theta_0}^{-\frac{1}{2}} \frac{\partial R_{\theta_0}}{\partial \theta}  R_{\theta_0}^{-\frac{1}{2}} R_{\theta_0}^{-\frac{1}{2}} \frac{\partial R_{\theta_0}}{\partial \theta} R_{\theta_0}^{-\frac{1}{2}}  \right)
& = & \left| R_{\theta_0}^{-\frac{1}{2}} \frac{\partial R_{\theta_0}}{\partial \theta}  R_{\theta_0}^{-\frac{1}{2}} \right|^2  \\
& \geq &  \inf_{i,n,x} \phi_i\left(R_{\theta_0}^{-\frac{1}{2}}\right)^4 \left| \frac{\partial R_{\theta_0}}{\partial \theta} \right|^2.
\end{eqnarray*}
By lemma \ref{lem: controle_valeurs_propres},
there exists $a>0$ so that $\inf_{i,n,x} \phi_i\left(R_{\theta_0}^{-\frac{1}{2}}\right)^4 \geq a$. We then show, similarly to the proof of proposition
\ref{prop: consistance_ML}, that the limit of $\left| \frac{\partial R_{\theta_0}}{\partial \theta} \right|^2$ is positive.

We now address the case $p >1$. Let $v_{\lambda} = (\lambda_1,...,\lambda_p)^t \in \RR^p$, $v_{\lambda}$ different from zero. We define the model
$\left\{ K_{\delta} , \delta \in [\delta_{inf},\delta_{sup}] \right\}$, with $\delta_{inf} < 0 < \delta_{sup}$ by
$ K_{\delta} = K_{\left(\theta_0\right)_1 + \delta \lambda_1,...,\left(\theta_0\right)_p + \delta \lambda_p} $. Then $K_{\delta=0} = K_{\theta_0}$.
We have $ \frac{\partial}{\partial \delta} K_{\delta=0}\left(t\right) = \sum_{k=1}^p \lambda_k \frac{\partial}{ \partial \theta_k} K_{\theta_0}\left(t\right) $,
so the model $\left\{ K_{\delta}, \delta \in [\delta_{inf},\delta_{sup}] \right\}$ verifies the hypotheses of the proposition for $p=1$. Hence,
the $\mathbb{P}$-mean square limit of $ \frac{\partial^2}{\partial \delta^2} L_{\delta=0}$ is positive. 
We conclude with
$ \frac{\partial^2}{\partial \delta^2} L_{\delta=0} = v_{\lambda}^t \left( \frac{\partial^2}{\partial \theta^2} L_{\theta_0} \right) v_{\lambda}$.

\end{proof}

\subsection{Proof of proposition \ref{prop: consistance_CV}}
\begin{proof}

 We will show that there exists
a sequence of random variables defined on $(\Omega_X,\mathcal{F}_X,P_X)$ $D_{\theta,\theta_0}$ so that
$\sup_{\theta} | \left(CV_{\theta} - CV_{\theta_0}\right) - D_{\theta,\theta_0} | \to_p 0 $ and $C >0$ so that $P_X$-a.s.
\begin{equation} \label{eq: condition_consistance_CV}
 D_{\theta,\theta_0} \geq C |  R_{\theta} - R_{\theta_0} |^2.
\end{equation}
The proof of the proposition is then carried out similarly to the proof of proposition \ref{prop: consistance_ML}.

Similarly to the proof of proposition \ref{prop: consistance_ML},
we firstly show that
$\sup_{\theta} | CV_{\theta} - \EE\left(CV_{\theta}|X\right)| \to_p 0 $. We then denote  $D_{\theta,\theta_0} =  \EE\left(CV_{\theta}|X\right) -  \EE\left(CV_{\theta_0}|X\right)$.
We decompose, for all $i \in \{1,...,n\}$, with $P_i$ the matrix that exchanges lines $1$ and $i$ of a matrix,
\[
P_i R_{\theta} P_i^t = 
\begin{pmatrix}
 1 & r_{i,\theta}^t \\
r_{i,\theta} & R_{-i,\theta} \\
\end{pmatrix}.
\]
The conditional distributions being independent on the numbering of the observations, we have, using the Kriging equations
\begin{eqnarray*}
 D_{\theta,\theta_0} & = & \frac{1}{n} \sum_{i=1}^n \EE\left\{ \left(   r_{i,\theta}^t  R_{-i,\theta}^{-1} y_{-i} - r_{i,\theta_0}^t  R_{-i,\theta_0}^{-1} y_{-i}  \right)^2 | X \right\}  \\
& = & \frac{1}{n} \sum_{i=1}^n \left(   r_{i,\theta}^t  R_{-i,\theta}^{-1} - r_{i,\theta_0}^t  R_{-i,\theta_0}^{-1}   \right) R_{-i,\theta_0}
  \left( R_{-i,\theta}^{-1} r_{i,\theta} - R_{-i,\theta_0}^{-1} r_{i,\theta_0}   \right).   \\
\end{eqnarray*}
Similarly to lemma \ref{lem: controle_valeurs_propres}, it can be shown that the eigenvalues of $R_{-i,\theta_0}$
are larger than a constant $A>0$, uniformly in $n$ and $x$. Then
\begin{eqnarray*}
 D_{\theta,\theta_0} & \geq & A \frac{1}{n} \sum_{i=1}^n \left|\left| \left(   r_{i,\theta}^t  R_{-i,\theta}^{-1} - r_{i,\theta_0}^t  R_{-i,\theta_0}^{-1}   \right) \right|\right|^2.  \\
\end{eqnarray*}
Using the virtual Cross Validation equations \citep[ch.5.2]{SS}, the vector $R_{-i,\theta}^{-1} r_{i,\theta}$
is the vector of the $ \frac{\left(R_{\theta}^{-1}\right)_{i,j}}{\left(R_{\theta}^{-1}\right)_{i,i}} $ for $1 \leq j \leq n$, $j\neq i$. Hence $P_X$-a.s.
\begin{eqnarray*}
 D_{\theta,\theta_0}  & \geq & A \frac{1}{n} \sum_{i=1}^n \sum_{j \neq i}
\left\{ \frac{ \left(R_{\theta}^{-1}\right)_{i,j} }{ \left(R_{\theta}^{-1}\right)_{i,i} } - \frac{ \left(R_{\theta_0}^{-1}\right)_{i,j} }{ \left(R_{\theta_0}^{-1}\right)_{i,i} } \right\}^2  \\
& = & A \frac{1}{n}  \left| \diag\left( R_{\theta}^{-1} \right)^{-1}  R_{\theta}^{-1} - \diag\left( R_{\theta_0}^{-1} \right)^{-1}  R_{\theta_0}^{-1} \right|^2  \\
& \geq & A B  \frac{1}{n}  \left| \diag\left( R_{\theta_0}^{-1} \right) \diag\left( R_{\theta}^{-1} \right)^{-1}   R_{\theta}^{-1} -  R_{\theta_0}^{-1} \right|^2,
~  ~ \mbox{with $B = \inf_{i,n,x} \phi_i^2 \left\{ \diag\left(R_{\theta_0}^{-1}\right)^{-1} \right\}$, $B >0$. }
\end{eqnarray*}
The eigenvalues of $\diag\left( R_{\theta_0}^{-1} \right) \diag\left( R_{\theta}^{-1} \right)^{-1}$ are bounded between $a>0$ and $b < \infty$ uniformly in $n$ and $x$.
Hence we have, with $D_{\lambda}$ the diagonal matrix with values $\lambda_1,...,\lambda_n$,

\begin{eqnarray*}
D_{\theta,\theta_0}  & \geq & A B \inf_{a \leq \lambda_1,...,\lambda_n \leq b} \left| D_{\lambda} R_{\theta}^{-1} -  R_{\theta_0}^{-1} \right|^2 \\
& \geq & A B C  \inf_{a \leq \lambda_1,...,\lambda_n \leq b} | D_{\frac{1}{\lambda}} R_{\theta} -  R_{\theta_0} |^2,  ~ ~ ~ \mbox{using  \cite{TCMR} theorem 1,}  \\
 & \geq & A B C  \inf_{\lambda_1,...,\lambda_n} | D_{\lambda} R_{\theta} -  R_{\theta_0} |^2,
\end{eqnarray*}
with $C=  \frac{1}{b} \inf_{n,x,\theta} \frac{1}{||R_{\theta}||^2} \frac{1}{||R_{\theta_0}||^2} $, $C >0$.  Then
\begin{eqnarray*}
D_{\theta,\theta_0}  & \geq & ABC \frac{1}{n} \inf_{\lambda_1,...,\lambda_n}  \sum_{i,j=1}^n \left( \lambda_i R_{\theta,i,j} - R_{\theta_0,i,j} \right)^2  \\
& = & ABC \frac{1}{n} \sum_{i=1}^n \inf_{\lambda} \sum_{j=1}^n \left( \lambda R_{\theta,i,j} - R_{\theta_0,i,j} \right)^2  \\
& = & ABC \frac{1}{n} \sum_{i=1}^n \inf_{\lambda} \left\{ \left(\lambda-1\right)^2 +  \sum_{j \neq i} \left( \lambda R_{\theta,i,j} - R_{\theta_0,i,j} \right)^2 \right\}.  \\
\end{eqnarray*}

\begin{lem} \label{lem: lambda_compensateur}
 For any $a_1,...,a_n$ and $b_1,...,b_n \in \RR$,
\[
\inf_{\lambda} \left\{ \left(\lambda-1\right)^2  + \sum_{i=1}^n \left(a_i - \lambda b_i\right)^2 \right\} \geq \frac{ \sum_{i=1}^n \left(a_i - b_i\right)^2 }{1 + \sum_{i=1}^n b_i^2}.  
\]
\end{lem}
\begin{proof}
\[
\left(\lambda-1\right)^2  + \sum_{i=1}^n \left(a_i - \lambda b_i\right)^2 = \lambda^2 \left( 1 + \sum_{i=1}^n b_i^2 \right) - 2 \lambda \left(1 + \sum_{i=1}^n a_i b_i\right) + \left( 1 + \sum_{i=1}^n a_i^2 \right).
\]
The minimum in $x$ of $ax^2 -2bx +c$, is $- \frac{b^2}{a} +c$, hence
\begin{eqnarray*}
 \left(\lambda-1\right)^2  + \sum_{i=1}^n \left(a_i - \lambda b_i\right)^2 & \geq & \left( 1 + \sum_{i=1}^n a_i^2 \right) - \frac{ \left(1 + \sum_{i=1}^n a_i b_i\right)^2 }{ \left( 1 + \sum_{i=1}^n b_i^2 \right) }  \\
& = & \frac{ \sum_{i=1}^n \left(a_i - b_i \right)^2  - \left(\sum_{i=1}^n a_i b_i\right)^2 + \left(\sum_{i=1}^n a_i^2\right) \left(\sum_{i=1}^n b_i^2\right) }{ 1 + \sum_{i=1}^n b_i^2 }  \\
& \geq & \frac{ \sum_{i=1}^n \left(a_i - b_i \right)^2  }{ 1 + \sum_{i=1}^n b_i^2 }, ~ ~ \mbox{using Cauchy-Schwartz inequality.}
\end{eqnarray*}

\end{proof}
Using lemma \ref{lem: lambda_compensateur}, together with \eqref{eq: controleCov} and lemma \ref{lem: sommabilité} which ensures that
$\sum_{j \neq i} \left( R_{\theta,i,j} \right)^2 \leq c < + \infty$ uniformly in $i$, $\theta$ and $x$, we obtain
\begin{eqnarray*}
 D_{\theta,\theta_0}  & \geq  &  ABC \frac{1}{1+c} \frac{1}{n} \sum_{i=1}^n   \sum_{j \neq i} \left( R_{\theta,i,j} - R_{\theta_0,i,j} \right)^2   \\
& = & ABC \frac{1}{1+c} |  R_{\theta} - R_{\theta_0} |^2, ~ ~ \mbox{because $R_{\theta,i,i} = 1 = R_{\theta_0,i,i}$,}
\end{eqnarray*}
which proves  \eqref{eq: condition_consistance_CV} and ends the proof.

\end{proof}

\subsection{Proof of proposition \ref{prop: gradients_CV}}
\begin{proof}
 
It is shown in \cite{Bachoc2013cross} that
$\frac{\partial}{\partial \theta_i} CV_{\theta} = \frac{2}{n} y^t M^i_{\theta} y = \frac{1}{n} y^t\left\{M^i_{\theta} + \left(M^i_{\theta}\right)^t\right\}y$.
We then show that $\cov\left( y^t A y , y^t B y |X \right) = 2 {\rm{Tr}}\left( A R_{\theta_0} B R_{\theta_0} \right)$ for symmetric matrices $A$ and $B$, which shows
\eqref{eq: cov_d1_CV}.

A straightforward but relatively long calculation then shows
\begin{eqnarray*} 
 \frac{\partial^2}{\partial \theta_i \partial \theta_j} CV_{\theta} & = &
- 4 \frac{1}{n} y^t R_{\theta}^{-1} \frac{\partial R_{\theta}}{\partial \theta_j} R_{\theta}^{-1}
\diag\left( R_{\theta}^{-1} \right)^{-3} \diag\left( R_{\theta}^{-1} \frac{\partial R_{\theta}}{\partial \theta_i} R_{\theta}^{-1} \right) R_{\theta}^{-1} y \\
& & - 4 \frac{1}{n} y^t R_{\theta}^{-1} \frac{\partial R_{\theta}}{\partial \theta_i} R_{\theta}^{-1}
\diag\left( R_{\theta}^{-1} \right)^{-3} \diag\left( R_{\theta}^{-1} \frac{\partial R_{\theta}}{\partial \theta_j} R_{\theta}^{-1} \right) R_{\theta}^{-1} y \\
&  & + 2 \frac{1}{n} y^t  R_{\theta}^{-1} \frac{\partial R_{\theta}}{\partial \theta_j} R_{\theta}^{-1} \diag\left( R_{\theta}^{-1} \right)^{-2} \nonumber
R_{\theta}^{-1} \frac{\partial R_{\theta}}{\partial \theta_i} R_{\theta}^{-1} y \\
& & + 6 \frac{1}{n} y^t R_{\theta}^{-1} \diag\left( R_{\theta}^{-1} \right)^{-4}
\diag\left( R_{\theta}^{-1} \frac{\partial R_{\theta}}{\partial \theta_j} R_{\theta}^{-1} \right)
\diag\left( R_{\theta}^{-1} \frac{\partial R_{\theta}}{\partial \theta_i} R_{\theta}^{-1} \right) \nonumber
R_{\theta}^{-1} y \\
& & - 4 \frac{1}{n} y^t  R_{\theta}^{-1} \diag\left( R_{\theta}^{-1} \right)^{-3} \nonumber
\diag\left( R_{\theta}^{-1} \frac{\partial R_{\theta}}{\partial \theta_i} R_{\theta}^{-1} \frac{\partial R_{\theta}}{\partial \theta_j} R_{\theta}^{-1} \right)
R_{\theta}^{-1} y \\
& & + 2 \frac{1}{n} y^t  R_{\theta}^{-1} \diag\left( R_{\theta}^{-1} \right)^{-3} \diag\left( R_{\theta}^{-1}
\frac{\partial^2 R_{\theta}}{\partial \theta_i \partial \theta_j}  R_{\theta}^{-1} \right) \nonumber
R_{\theta}^{-1} y \\
& & + 2 \frac{1}{n} y^t  R_{\theta}^{-1} \diag\left( R_{\theta}^{-1} \right)^{-2} \nonumber
R_{\theta}^{-1} \frac{\partial R_{\theta}}{\partial \theta_j} R_{\theta}^{-1} \frac{\partial R_{\theta}}{\partial \theta_i} R_{\theta}^{-1} y \\
& & + 2 \frac{1}{n} y^t  R_{\theta}^{-1} \diag\left( R_{\theta}^{-1} \right)^{-2} \nonumber
R_{\theta}^{-1} \frac{\partial R_{\theta}}{\partial \theta_i} R_{\theta}^{-1} \frac{\partial R_{\theta}}{\partial \theta_j} R_{\theta}^{-1} y \\
& & -2 \frac{1}{n} y^t R_{\theta}^{-1} \diag\left( R_{\theta}^{-1} \right)^{-2} R_{\theta}^{-1}
\frac{\partial^2 R_{\theta}}{\partial \theta_i \partial \theta_j}  R_{\theta}^{-1}  y. \nonumber
\end{eqnarray*}

We then have, using $\EE\left(y^tAy |X\right) = {\rm{Tr}}\left(A R_{\theta_0}\right)$ and for matrices $D$, $M_1$ and $M_2$, with $D$ diagonal,
${\rm{Tr}}\left\{M_1 D \diag\left(M_2\right)\right\} = {\rm{Tr}}\left\{ M_2 D \diag\left(M_1\right) \right\}$
and ${\rm{Tr}}\left(DM_1\right) = {\rm{Tr}}\left(DM_1^t\right)$,

\begin{eqnarray} \label{eq: expression_esperance_derivee_seconde_CV}
 \EE \left( \frac{\partial^2}{\partial \theta_i \partial \theta_j} CV_{\theta_0} | X \right) & = &
- 8 \frac{1}{n} {\rm{Tr}} \left\{ \frac{\partial R_{\theta_0}}{\partial \theta_j} R_{\theta_0}^{-1}
\diag\left( R_{\theta_0}^{-1} \right)^{-3} \diag\left( R_{\theta_0}^{-1} \frac{\partial R_{\theta_0}}{\partial \theta_i} R_{\theta_0}^{-1} \right) R_{\theta_0}^{-1} \right\} \\
&  & + 2 \frac{1}{n} {\rm{Tr}} \left\{ \frac{\partial R_{\theta_0}}{\partial \theta_j} R_{\theta_0}^{-1} \diag\left( R_{\theta_0}^{-1} \right)^{-2} \nonumber
R_{\theta_0}^{-1} \frac{\partial R_{\theta_0}}{\partial \theta_i} R_{\theta_0}^{-1} \right\} \\
& & + 6 \frac{1}{n} {\rm{Tr}} \left\{ \diag\left( R_{\theta_0}^{-1} \right)^{-4}
\diag\left( R_{\theta_0}^{-1} \frac{\partial R_{\theta_0}}{\partial \theta_j} R_{\theta_0}^{-1} \right)
\diag \left( R_{\theta_0}^{-1} \frac{\partial R_{\theta_0}}{\partial \theta_i} R_{\theta_0}^{-1} \right) \nonumber
R_{\theta_0}^{-1} \right\} \\
& & - 4 \frac{1}{n} {\rm{Tr}} \left\{ \diag\left( R_{\theta_0}^{-1} \right)^{-3} \nonumber
\diag\left( R_{\theta_0}^{-1} \frac{\partial R_{\theta_0}}{\partial \theta_i} R_{\theta_0}^{-1} \frac{\partial R_{\theta_0}}{\partial \theta_j} R_{\theta_0}^{-1} \right)
R_{\theta_0}^{-1} \right\} \\
& & + 2 \frac{1}{n} {\rm{Tr}} \left\{ \diag\left( R_{\theta_0}^{-1} \right)^{-3} \diag\left( R_{\theta_0}^{-1}
\frac{\partial^2 R_{\theta_0}}{\partial \theta_i \partial \theta_j}  R_{\theta_0}^{-1} \right) \nonumber
R_{\theta_0}^{-1} \right\} \\
& & + 4 \frac{1}{n} {\rm{Tr}} \left\{ \diag\left( R_{\theta_0}^{-1} \right)^{-2} \nonumber
R_{\theta_0}^{-1} \frac{\partial R_{\theta_0}}{\partial \theta_i} R_{\theta_0}^{-1} \frac{\partial R_{\theta_0}}{\partial \theta_j} R_{\theta_0}^{-1} \right\} \\
& & -2 \frac{1}{n} {\rm{Tr}} \left\{ \diag\left( R_{\theta_0}^{-1} \right)^{-2} R_{\theta_0}^{-1}
\frac{\partial^2 R_{\theta_0}}{\partial \theta_i \partial \theta_j}  R_{\theta_0}^{-1}  \right\}. \nonumber
\end{eqnarray}

The fourth and sixth terms of \eqref{eq: expression_esperance_derivee_seconde_CV}
are opposite and hence cancel each other. Indeed,

\begin{eqnarray*}
& & {\rm{Tr}}\left\{  \diag\left( R_{\theta_0}^{-1} \right)^{-3}
\diag\left( R_{\theta_0}^{-1} \frac{\partial R_{\theta_0}}{\partial \theta_i}  R_{\theta_0}^{-1} \frac{\partial R_{\theta_0}}{\partial \theta_j}  R_{\theta_0}^{-1} \right)
R_{\theta_0}^{-1} \right\} \\
& = & \sum_{i=1}^n \left( R_{\theta_0}^{-1} \right)_{i,i}^{-3} 
\left( R_{\theta_0}^{-1} \frac{\partial R_{\theta_0}}{\partial \theta_i}  R_{\theta_0}^{-1} \frac{\partial R_{\theta_0}}{\partial \theta_j}  R_{\theta_0}^{-1} \right)_{i,i}
\left( R_{\theta_0}^{-1} \right)_{i,i}  \\
& = & \sum_{i=1}^n \left( R_{\theta_0}^{-1} \right)_{i,i}^{-2} 
\left( R_{\theta_0}^{-1} \frac{\partial R_{\theta_0}}{\partial \theta_i}  R_{\theta_0}^{-1} \frac{\partial R_{\theta_0}}{\partial \theta_j}  R_{\theta_0}^{-1} \right)_{i,i}\\
& = & {\rm{Tr}}\left\{  \diag\left( R_{\theta_0}^{-1} \right)^{-2}
R_{\theta_0}^{-1} \frac{\partial R_{\theta_0}}{\partial \theta_i}  R_{\theta_0}^{-1} \frac{\partial R_{\theta_0}}{\partial \theta_j}  R_{\theta_0}^{-1} \right\}.
\end{eqnarray*}
Similarly the fifth and seventh terms of \eqref{eq: expression_esperance_derivee_seconde_CV} cancel each other.

Hence, we show the expression of $\EE\left( \frac{\partial^2}{ \partial \theta_i \partial \theta_j} CV_{\theta_0} | X \right)$ of the proposition.

We use proposition \ref{prop: convergenceTrace} to show the existence of $\Sigma_{CV,1}$ and $\Sigma_{CV,2}$.
\end{proof}

\subsection{Proof of proposition \ref{prop: normaliteCV}}

\begin{proof}

We use proposition \ref{prop: lindeberg_presque_sur}, with $M^i_{\theta_0}$ the notation of proposition \ref{prop: gradients_CV} and
\[
N_i = - \left\{ M^i_{\theta_0} + \left(M^i_{\theta_0}\right)^t \right\} ,  
\]
together with propositions \ref{prop: convergenceTrace} and \ref{prop: gradients_CV}
to show that
\[
 \sqrt{n} \frac{ \partial }{ \partial \theta } CV_{\theta_0} \to_{\L} \N\left( 0 , \Sigma_{CV,1} \right).
\]
We have seen in the proof of proposition \ref{prop: gradients_CV} that
there exist matrices $P_{i,j}$ in $\mathcal{M}_{\theta_0}$ (proposition \ref{prop: convergenceTrace}), so that
$\frac{ \partial^2 }{ \partial \theta_i \partial \theta_j} CV_{\theta_0} = \frac{1}{n} y^t P_{i,j} y$, with
$ \frac{1}{n} {\rm{Tr}}\left(P_{i,j}R\right) \to \left(\Sigma_{CV,2}\right)_{i,j}$ $P_X$-almost surely.
Hence, using proposition \ref{prop: convergence_forme_quadratique}, $\frac{ \partial^2 }{ \partial \theta^2 } L_{\theta_0} $
converges to $\Sigma_{CV,2}$ in the mean square sense (on the product space).

Finally, $\frac{ \partial^3 }{ \partial \theta_i \partial \theta_j \partial \theta_k} CV_{\tilde{ \theta }} $
can be written as $\frac{1}{n} \left( z^t N^{i,j,k}_{\tilde{\theta}} z \right)$, where
$N^{i,j,k}_{\tilde{\theta}}$ are sums of matrices
of $\mathcal{M}_{\tilde{\theta}}$ (proposition \ref{prop: convergenceTrace}) and
$z$ depending on $X$ and $Y$ with $\L\left(z | X\right) = \N\left(0,I_n\right)$. The singular values of
$N^{i,j,k}_{\tilde{\theta}}$ are bounded uniformly in
$\tilde{\theta}$, $n$ and $x$ and so $ \sup_{i,j,k,\tilde{\theta}} \left( \frac{ \partial^3 }{ \partial \theta_i \partial \theta_j \partial \theta_k}
CV_{\tilde{ \theta }} \right)$ is bounded by $b \frac{1}{n} z^t z$, $b < + \infty$, and
is hence bounded in probability. We apply proposition \ref{prop: normaliteEstimateur} to conclude.

\end{proof}

\subsection{Proof of proposition \ref{prop: info>0_CV}}

\begin{proof}
 
We show the proposition in the case $p=1$, the generalization to the case $p >1$ being the same as in proposition \ref{prop: info>0_ML}.

Similarly to the proof of proposition \ref{prop: consistance_ML}, we show
that $ \left| \frac{\partial^2}{\partial \theta^2} CV_{\theta_0} - \EE\left( \frac{\partial^2}{\partial \theta^2} CV_{\theta_0} | X \right)\right| \to_p 0 $.
We will then show that there exists $C>0$ so that $P_X$-a.s.,
\begin{equation} \label{eq: condition_info>0_CV}
 \EE\left( \frac{\partial^2}{\partial \theta^2} CV_{\theta_0} | X \right)  \geq C \left| \frac{\partial R_{\theta}}{\partial \theta} \right|^2.
\end{equation}
The proof of the proposition will hence be carried out similarly as in the proof of proposition \ref{prop: consistance_ML}.

$ \frac{\partial^2}{\partial \theta^2} CV_{\theta_0} $ can be written as $z^t M z$ with
$z$ depending on $X$ and $Y$ and $\L\left(z | X \right) = \N\left(0,I_n\right)$, and $M$ a sum of matrices of
$\mathcal{M}_{\theta_0}$ (proposition \ref{prop: convergenceTrace}). Hence, using proposition \ref{prop: convergenceTrace}, uniformly in $n$,
$\sup_{\theta} \left| \frac{\partial^2}{\partial \theta^2} CV_{\theta} \right| \leq a \frac{1}{n} z^t z$
with $a < + \infty$. Hence, for fixed $n$, we can exchange derivatives and means conditionally to $X$ and so 
\begin{eqnarray*}
 \EE\left( \frac{\partial^2}{\partial \theta^2} CV_{\theta_0} | X \right) & = &  \frac{\partial^2}{\partial \theta^2} \EE\left( CV_{\theta_0} | X \right).
\end{eqnarray*}
Then, with $r_{i,\theta}$, $R_{-i,\theta} $ and $y_{-i}$ the notation of the proof of proposition  \ref{prop: consistance_CV},
\begin{eqnarray*}
 \EE\left( CV_{\theta} | X \right) & = & \frac{1}{n} \sum_{i=1}^n \left[  1 - r_{i,\theta_0}^t R_{-i,\theta_0}^{-1} r_{i,\theta_0} +
\EE\left\{ \left( r_{i,\theta_0}^t R_{-i,\theta_0}^{-1} y_{-i} - r_{i,\theta}^t R_{-i,\theta}^{-1} y_{-i} \right)^2 | X \right\}  \right]  \\
& = & \frac{1}{n} \sum_{i=1}^n \left(  1 - r_{i,\theta_0}^t R_{-i,\theta_0}^{-1} r_{i,\theta_0} \right)
+ \frac{1}{n} \sum_{i=1}^n \left( r_{i,\theta}^t R_{-i,\theta}^{-1} - r_{i,\theta_0}^t R_{-i,\theta_0}^{-1}\right) R_{-i,\theta_0} \left(  R_{-i,\theta}^{-1} r_{i,\theta} - R_{-i,\theta_0}^{-1} r_{i,\theta_0}\right).  
\end{eqnarray*}
By differentiating twice with respect to $\theta$ and taking the value at $\theta_0$ we obtain
\begin{eqnarray*}
\EE\left( \frac{\partial^2}{\partial \theta^2} CV_{\theta_0} | X \right) & = & \frac{1}{n} \sum_{i=1}^n
\left\{ \frac{\partial}{\partial \theta} \left( R_{-i,\theta_0}^{-1} r_{i,\theta_0}^t \right) \right\}^t
R_{-i,\theta_0} \left\{\frac{\partial}{\partial \theta} \left( R_{-i,\theta_0}^{-1} r_{i,\theta_0}^t \right) \right\} \\
& \geq & A \frac{1}{n} \sum_{i=1}^n \left|\left| \left\{\frac{\partial}{\partial \theta} \left( R_{-i,\theta_0}^{-1} r_{i,\theta_0}^t \right) \right\} \right|\right|^2
~ ~ \mbox{with $A = \inf_{n,i,x} \phi^2_i\left(R_{-i,\theta_0}\right)$, $A > 0$}, \\
\end{eqnarray*}

then, using the virtual CV formulas \citep{SS,CVKUN},

\begin{eqnarray*}
\EE\left( \frac{\partial^2}{\partial \theta^2} CV_{\theta_0} | X \right)
& \geq & A \frac{1}{n} \sum_{i=1}^n \sum_{j \neq i} \left[ \frac{\partial}{\partial \theta} \left\{  \frac{\left(R_{\theta_0}^{-1}\right)_{i,j}}{\left(R_{\theta_0}^{-1}\right)_{i,i}}  \right\}  \right]^2 \\
& = & A \left| \frac{\partial}{\partial \theta} \left\{ \diag\left( R_{\theta_0}^{-1} \right)^{-1} R_{\theta_0}^{-1} \right\} \right|^2 \\
& = & A \left|  \diag\left( R_{\theta_0}^{-1} \right)^{-1} \diag\left( R_{\theta_0}^{-1} \frac{\partial R_{\theta_0}}{\partial \theta} R_{\theta_0}^{-1}  \right)
\diag\left( R_{\theta_0}^{-1} \right)^{-1} R_{\theta_0}^{-1} - \diag\left( R_{\theta_0}^{-1} \right)^{-1} \left( R_{\theta_0}^{-1} \frac{\partial R_{\theta_0}}{\partial \theta} R_{\theta_0}^{-1} \right)  \right|^2  \\
& \geq & A^2 B \left|   \diag\left( R_{\theta_0}^{-1} \frac{\partial R_{\theta_0}}{\partial \theta} R_{\theta_0}^{-1}  \right)
\diag\left( R_{\theta_0}^{-1} \right)^{-1}  -   R_{\theta_0}^{-1} \frac{\partial R_{\theta_0}}{\partial \theta}    \right|^2
~ ~ \mbox{with $B = \inf_{i,n,x} \phi_i\left(R_{\theta_0}^{-1}\right)$, $B > 0$}  \\
& \geq & A^2 B \inf_{\lambda_1,...,\lambda_n} \left|   D_{\lambda} 
-   R_{\theta_0}^{-1} \frac{\partial R_{\theta_0}}{\partial \theta}    \right|^2  \\
& \geq & A^2 B^2 \inf_{\lambda_1,...,\lambda_n} \left| R_{\theta_0}  D_{\lambda} 
-    \frac{\partial R_{\theta_0}}{\partial \theta}    \right|^2.  \\
\end{eqnarray*}
Then, as $K_{\theta}\left(0\right)=1$ for all $\theta$, and hence $\frac{\partial}{\partial \theta} K_{\theta_0}\left(0\right) = 0$,
\begin{eqnarray*}
\EE\left( \frac{\partial^2}{\partial \theta^2} CV_{\theta_0} | X \right)
& \geq & A^2 B^2 \inf_{\lambda_1,...,\lambda_n} \frac{1}{n} \sum_{i=1}^n \left[ \lambda_i^2 + \sum_{j \neq i}
\left\{ \lambda_i \left( R_{\theta_0} \right)_{i,j} -  \left( \frac{\partial R_{\theta_0}}{\partial \theta} \right)_{i,j} \right\}^2 \right]\\
& = & A^2 B^2  \frac{1}{n} \sum_{i=1}^n \inf_{\lambda} \left[ \lambda^2 + \sum_{j \neq i}
\left\{ \lambda \left( R_{\theta_0} \right)_{i,j} -  \left( \frac{\partial R_{\theta_0}}{\partial \theta} \right)_{i,j} \right\}^2 \right].
\end{eqnarray*}

We then show, similarly to lemma \ref{lem: lambda_compensateur}, that
\begin{equation}  \label{eq: lambda_compensateur}
 \lambda^2 + \sum_{i=1}^n \left(a_i - \lambda b_i\right)^2 \geq \frac{ \sum_{i=1}^n a_i^2 }{ 1 + \sum_{i=1}^n b_i^2}.
\end{equation}
Hence, with $C \in [1, + \infty) $, by using \eqref{eq: controleCov} and lemma \ref{lem: sommabilité},
\begin{eqnarray*}
 \EE\left( \frac{\partial^2}{\partial \theta^2} CV_{\theta_0} | X \right)  & \geq & \frac{A^2 B^2}{C}  \frac{1}{n} \sum_{i=1}^n \sum_{j \neq i}
\left\{  \left( \frac{\partial R_{\theta_0}}{\partial \theta} \right)_{i,j} \right\}^2  \\
& = & \frac{A^2 B^2}{C} \left| \frac{\partial R_{\theta_0}}{\partial \theta} \right|^2 ~ ~ \mbox{because $\frac{\partial}{\partial \theta} K_{\theta_0(0)} = 0$}.
\end{eqnarray*}
We then showed \eqref{eq: condition_info>0_CV}, which concludes the proof in the case $p=1$.

\end{proof}

\section{Proofs for section \ref{section: numerical_study} }  \label{section: appendix_proof_echange_limite_derivee}

\subsection{Proof of proposition \ref{prop: deriveesSigma}}

\begin{proof}
It is enough to show the proposition for $\epsilon \in [0,\alpha]$ for all $\alpha < \frac{1}{2}$.
We use the following lemma.
\begin{lem} \label{lem: convergenceUnif}
 Let $f_n$ be a sequence of $C^2$ functions on a segment of $\RR$. We assume $f_n \to_{unif} f$, $f_n' \to_{unif} g$, $f_n'' \to_{unif} h$.
Then, $f$ is $C^2$, $f' = g$, and $f'' = h$.
\end{lem}

We denote $f_n\left(\epsilon\right) = \frac{1}{n} \EE \left\{ {\rm{Tr}}\left(M^{i,j}\right) \right\}$
where $(M^{i,j})_{n \in \NN^*}$ is a random matrix sequence defined on $(\Omega_X,\mathcal{F}_X,P_X)$
which belongs to $\mathcal{M}_{\theta}$ (proposition \ref{prop: convergenceTrace}).
We showed in proposition \ref{prop: convergenceTrace} that $f_n$ converges simply to $\Sigma_{i,j}$ on $[0,\alpha]$. 
We firstly use the dominated convergence theorem to show that $f_n$ is $C^2$ and that $f_n'$ and $f_n''$ are of the form
\begin{equation} \label{eq: formeTraceDerivees}
 \EE \left\{ \frac{1}{n} {\rm{Tr}}\left( N^{i,j} \right)  \right\},
\end{equation}
with $N^{i,j}$ a sum of random matrix sequences of $\tilde{\mathcal{M}}_{\theta_0}$. $\tilde{\mathcal{M}}_{\theta_0}$ is similar to $\mathcal{M}_{\theta_0}$
(proposition \ref{prop: convergenceTrace}), with the addition of the derivative matrices with respect to $\epsilon$.
We can then, using \eqref{eq: controleDeriveesK}, adapt proposition \ref{prop: convergenceTrace} to show that
$f_n'$ and $f_n''$ converge simply to some functions $g$ and $h$ on $[0,\alpha] $.

Finally, still adapting proposition \ref{prop: convergenceTrace}, the singular values of $N^{i,j}$ are bounded uniformly in $x$ and $n$.
Hence, using ${\rm{Tr}}\left(A\right) \leq n ||A||$, for a symmetric matrix $A$, the derivatives of $f_n$, $f_n'$ and $f_n''$ are bounded
uniformly in $n$, so that
the simple convergence implies the uniform convergence. The conditions of lemma \ref{lem: convergenceUnif} are hence fulfilled.
 
\end{proof}

\section{Proofs for section \ref{section: analysis_prediction} }  \label{section: proof_prediction_independente_estimation}

\subsection{Proof of proposition \ref{prop: infl_est_pred}}

\begin{proof}
Let us first show \eqref{eq: no_inf_est_on _pred}.
Consider a consistent estimator $\hat{\theta}$ of $\theta_0$.
Since $|\hat{\theta} - \theta_0| = o_p(1)$, it is sufficient to show that $\sup_{1 \leq i \leq p, \theta \in \Theta} | \frac{\partial}{\partial \theta_i} E_{\epsilon,\theta} | = O_p(1)$.

Consider a fixed $n$. Because the trajectory $Y(t)$
is almost surely continuous on $[0,N_{1,n}]^d$, because for every $\theta \in \Theta$, $1\leq i \leq p$,
$\frac{\partial}{\partial \theta_i} K_{\theta}(t)$
is continuous with respect to $t$ and because, from
\eqref{eq: controleCov},
$ \sup_{\theta \in \Theta,1\leq i \leq p} \left| \frac{\partial}{\partial \theta_i} K_{\theta}(t)\right|$
is bounded, we can almost surely exchange integration and derivation w.r.t. $\theta_i$ in the expression of $E_{\epsilon,\theta}$. Thus, we have almost surely
\begin{eqnarray*}
\frac{\partial}{\partial \theta_i}
E_{\epsilon,\theta} & = &
\frac{1}{N_{1,n}^d} \int_{[0,N_{1,n}]^d}
\frac{\partial}{\partial \theta_i} \left(  \left( Y(t) - \hat{Y}_{\theta}(t) \right)^2\right) dt  \\
& = & \frac{2}{N_{1,n}^d} \int_{[0,N_{1,n}]^d}
\left( Y(t) - \hat{Y}_{\theta}(t) \right)
\left( - \frac{\partial r_{\theta}^t(t) }{\partial \theta_i}  R_{\theta}^{-1} + r_{\theta}^t(t) R_{\theta}^{-1} \frac{\partial R_{\theta}}{\partial \theta_i} R_{\theta}^{-1} \right)y dt,
\end{eqnarray*}
with $\left(r_{\theta}(t)\right)_j = K_{\theta}(v_j + \epsilon x_j - t)$.
Then
\begin{flalign} \label{eq: in_proof_bounded_derivative_prediction_error}
& \EE \left( \sup_{\theta \in \Theta} \left| \frac{\partial}{\partial \theta_i} E_{\epsilon,\theta} \right| \right)
 &  \nonumber \\
& \leq  \frac{2}{N_{1,n}^d} \int_{[0,N_{1,n}]^d} 
\EE \left( \sup_{\theta \in \Theta} \left\{ \left| Y(t) - \hat{Y}_{\theta}(t) \right|
\left| \left( \frac{\partial r_{\theta}^t(t) }{\partial \theta_i}  R_{\theta}^{-1} - r_{\theta}^t(t) R_{\theta}^{-1} \frac{\partial R_{\theta}}{\partial \theta_i} R_{\theta}^{-1} \right) y \right|
\right\} \right) dt \nonumber & \\
& \leq 
\frac{2}{N_{1,n}^d} \sqrt{  \int_{[0,N_{1,n}]^d}
\EE \left( \sup_{\theta \in \Theta} \left\{ \left( Y(t) - \hat{Y}_{\theta}(t) \right)^2 \right\} \right) } & \nonumber \\
& ~ ~ \times 
\sqrt{ \int_{[0,N_{1,n}]^d} \EE \left(
\sup_{\theta \in \Theta} \left\{ \left( \left( \frac{\partial r_{\theta}^t(t) }{\partial \theta_i}  R_{\theta}^{-1} - r_{\theta}^t(t) R_{\theta}^{-1} \frac{\partial R_{\theta}}{\partial \theta_i} R_{\theta}^{-1} \right)y \right)^2 \right\} \right) dt } & \nonumber \\
& \leq    
\frac{2}{N_{1,n}^d}
\sqrt{ 2 \int_{[0,N_{1,n}]^d}
\EE \left(  \left\{ Y(t)^2 \right\} \right) +
2 \int_{[0,N_{1,n}]^d}
\EE \left( \sup_{\theta \in \Theta} \left\{ \left( \hat{Y}_{\theta}(t) \right)^2 \right\} \right)
} & \nonumber \\
& ~ ~ \times 
\sqrt{ \int_{[0,N_{1,n}]^d} \EE \left(
\sup_{\theta \in \Theta} \left\{ \left( \left( \frac{\partial r_{\theta}^t(t) }{\partial \theta_i}  R_{\theta}^{-1} - r_{\theta}^t(t) R_{\theta}^{-1} \frac{\partial R_{\theta}}{\partial \theta_i} R_{\theta}^{-1} \right)y \right)^2 \right\} \right) dt } & \nonumber  \\
& =    
\frac{2}{N_{1,n}^d}
\sqrt{ 2 K_{\theta_0}(0) N_{1,n}^d +
2 \int_{[0,N_{1,n}]^d}
\EE \left( \sup_{\theta \in \Theta} \left\{ \left( \hat{Y}_{\theta}(t) \right)^2 \right\} \right)
} & \nonumber \\
& ~ ~ \times 
\sqrt{ \int_{[0,N_{1,n}]^d} \EE \left(
\sup_{\theta \in \Theta} \left\{ \left( \left( \frac{\partial r_{\theta}^t(t) }{\partial \theta_i}  R_{\theta}^{-1} - r_{\theta}^t(t) R_{\theta}^{-1} \frac{\partial R_{\theta}}{\partial \theta_i} R_{\theta}^{-1} \right)y \right)^2 \right\} \right) dt }. &
\end{flalign}

In \eqref{eq: in_proof_bounded_derivative_prediction_error},
the two supremums can be written
\[
\EE \left( \sup_{\theta \in \Theta} \left\{   \left( w_{\theta}(t)^t y  \right)^2 \right\} \right), 
\]
with $w_{\theta}(t)$ a column vector of size $n$,
not depending on $y$.

Fix $t  \in [0,N_{1,n}]^d$. We now use Sobolev embedding theorem
on the space $\Theta$, equipped with the Lebesgue measure.
This theorem implies that for $f: \Theta \to \RR$,
$\sup_{\theta \in \Theta} | f(\theta) | \leq C_p
 \int_{\Theta}  \left( |f(\theta)|^p + \sum_{j=1}^p \left| \frac{\partial}{\partial \theta_j} f(\theta) \right|^p \right) d \theta$, with $C_p$ a finite constant depending only on $p$ and $\Theta$. By applying this inequality to the $C^1$ function of $\theta$,
$\left( w_{\theta}(t)^t y \right)^2$, we obtain
\begin{eqnarray*} 
\EE \left( \sup_{\theta \in \Theta} \left\{   \left( w_{\theta}(t)^t y  \right)^2 \right\}\right)
& \leq &
  \EE \left( C_p \int_{\Theta} \sum_{i=1}^p \left| \frac{\partial}{\partial \theta_i}  \left(   \left( w_{\theta}(t)^t y \right) \right)^2 \right|^p d \theta \right) 
+  \EE \left( C_p \int_{\Theta}  \left|   \left(   \left( w_{\theta}(t)^t y \right) \right)^2 \right|^p d \theta \right)  \nonumber \\
& = &
2 C_p \sum_{i=1}^p \int_{\Theta} \EE \left( \left|   \left( w_{\theta}(t)^t y \right)
\left(   \frac{\partial}{\partial \theta_i} \left( w_{\theta}(t)^t  \right) y \right)
 \right|^p \right) d \theta \nonumber 
 + C_p \int_{\Theta} \EE \left( \left\{    w_{\theta}(t)^t y 
 \right\}^{2p} \right) d \theta\\
& \leq  &
2 C_p \sum_{i=1}^p \sqrt{ \int_{\Theta} \EE \left( \left\{   \left( w_{\theta}(t)^t y \right)
 \right\}^{2p} \right) d \theta   }
\sqrt{ \int_{\Theta} \EE \left( \left\{   
\left(   \frac{\partial}{\partial \theta_i} \left( w_{\theta}(t)^t  \right) y \right)
 \right\}^{2p} \right) d \theta } \\
 & & + C_p \int_{\Theta} \EE \left( \left\{    w_{\theta}(t)^t y 
 \right\}^{2p} \right) d \theta.
\end{eqnarray*}
There exists a constant $C'_p$, depending only on $p$
so that, for $Z$ a centered Gaussian variable,
$\EE( Z^{2p} ) = C'_p \var(Z)^p$. Thus, we obtain
\begin{flalign}  \label{eq: after_sobolev_embedding}
& \EE \left( \sup_{\theta \in \Theta} \left\{   \left( w_{\theta}(t)^t y  \right)^2 \right\}\right) & \nonumber \\
& \leq 2 C_p C'_p \sum_{i=1}^p \sqrt{ \int_{\Theta} \left[ \EE \left( \left\{  w_{\theta}(t)^t y 
 \right\}^{2} \right) \right]^p  d \theta   }
\sqrt{ \int_{\Theta} \left[ \EE \left( \left\{   
   \frac{\partial}{\partial \theta_i} \left( w_{\theta}(t)^t  \right) y 
 \right\}^{2} \right) \right]^p d \theta } \nonumber \\
&  ~ ~ + C_p C_p' \int_{\Theta} \left[ \EE \left( \left\{    w_{\theta}(t)^t y 
 \right\}^{2} \right) \right]^p d \theta. &
\end{flalign}

Fix $1 \leq i \leq p $ in \eqref{eq: after_sobolev_embedding}.
By using \eqref{eq: controleCov}, a slight modification of lemma \ref{lem: sommabilité}
and lemma \ref{lem: controle_valeurs_propres}, $\sup_{\theta \in \Theta, t \in [0,N_{1,n}]^d} |w_{\theta}(t)|^2 \leq A$ and
$\sup_{\theta \in \Theta, t \in [0,N_{1,n}]^d} |\frac{\partial}{\partial \theta_i} w_{\theta}(t)|^2 \leq A$, independently of $n$ and $x$ and for a constant $A < + \infty$. Thus, in \eqref{eq: after_sobolev_embedding},
$\EE \left( \left\{  w_{\theta}(t)^t y 
 \right\}^{2} \right) = \EE_{X} \left( w_{\theta}(t)^t R_{\theta_0} w_{\theta}(t)^t \right) \leq AB$, with $B = \sup_{n,x} ||R_{\theta_0}||$. We show in the same way
$\EE \left( \left\{   
   \frac{\partial}{\partial \theta_i} \left( w_{\theta}(t)^t  \right) y 
 \right\}^{2} \right) \leq A B$.
Hence, from \eqref{eq: in_proof_bounded_derivative_prediction_error} and \eqref{eq: after_sobolev_embedding}, we have shown
that, for $1\leq i \leq p$, 
\[
 \EE \left( \sup_{\theta \in \Theta} \left| \frac{\partial}{\partial \theta_i}E_{\epsilon,\theta} \right| \right) 
\]
is bounded inpendently of $n$.
Hence $\sup_{\theta \in \Theta,1 \leq i \leq p} | \frac{\partial}{\partial \theta_i} E_{\epsilon,\theta} | = O_p(1)$,
which proves \eqref{eq: no_inf_est_on _pred}.

Let us now prove \eqref{eq: no_vanishing_pred_error}.

\begin{eqnarray*}
\EE \left( E_{\epsilon,\theta_0} \right) & = &
\EE \left( \frac{1}{(N_{1,n})^d}  \int_{[0,N_{1,n}]^d}
\left( Y(t) - \hat{Y}_{\theta_0}(t)   \right)^2 dt \right) \\
& = &  \frac{1}{(N_{1,n})^d}  \int_{[0,N_{1,n}]^d}
\EE_{X} \left( 1 - r_{\theta_0}^t(t) R^{-1}_{\theta_0}
r_{\theta_0}(t)  \right) dt.
\end{eqnarray*}
Now, let $\tilde{R}_{\theta_0}(t)$
be the covariance matrix of $(Y(t),y_1,...,y_n)^t$,
under covariance function $K_{\theta_0}$. Then, because of the virtual Leave-One-Out formulas \citep[ch.5.2]{SS},
\begin{eqnarray*}
\EE \left( E_{\epsilon,\theta_0} \right) & = &
\frac{1}{(N_{1,n})^d}  \int_{[0,N_{1,n}]^d}
\EE_{X} \left( \frac{1}{ ( \tilde{R}^{-1}_{\theta_0} )_{1,1} }  \right) dt \\
& \geq & \frac{1}{N_{1,n}^d} \sum_{i=1}^n
\int_{ \prod_{k=1}^d [ (v_i)_k + \epsilon + \frac{1}{2}(\frac{1}{2} - \epsilon) , (v_i)_k +1 - \epsilon - \frac{1}{2}(\frac{1}{2} - \epsilon) ] }
\EE_{X} \left( \frac{1}{ ( \tilde{R}^{-1}_{\theta_0} )_{1,1} }  \right) dt \\
\end{eqnarray*}
Now, for $t \in \prod_{k=1}^d [ (v_i)_k + \epsilon + \frac{1}{2}(\frac{1}{2} - \epsilon) , (v_i)_k +1 - \epsilon - \frac{1}{2}(\frac{1}{2} - \epsilon) ]$,
$ \inf_{n, 1 \leq j \leq n, x \in S_X^n} |t - v_j - \epsilon x_j|_{\infty} \geq  \frac{1}{2} \left( \frac{1}{2} - \epsilon \right)$. Thus, we can adapt proposition \ref{prop: minorationValeursPropres} to show that
the eigenvalues of $\tilde{R}_{\theta_0}(t)$
are larger than $A > 0$, independently of $n$, $x$
and $t \in \cup_{1\leq i \leq n} \prod_{k=1}^d [ (v_i)_k + \epsilon + \frac{1}{2}(\frac{1}{2} - \epsilon) , (v_i)_k +1 - \epsilon - \frac{1}{2}(\frac{1}{2} - \epsilon) ]$. This yields
\[
\EE \left( E_{\epsilon,\theta_0} \right)
\geq \frac{A}{ N_{1,n}^d } N_{1,n}^d \left( \frac{1}{2} - \epsilon \right)^d,
\] 
which concludes the proof.

\end{proof}

\section{Technical results}  \label{section: appendix_technical_results}

In subsection \ref{subsection: technical_results}
we state several technical results that are used in the proofs of the results of sections \ref{section: consistency_asymptotic_normality},
\ref{section: numerical_study}
and \ref{section: analysis_prediction}. Proofs are given in subsection \ref{subsection: proof_technical_results}.

\subsection{Statement of the technical results} \label{subsection: technical_results}

\begin{lem} \label{lem: sommabilité}
 Let $f: \RR^d \to \RR^+$, so that $f\left(t\right) \leq \frac{1}{1+|t|_{\infty}^{d+1}}$. Then, for all $i \in \NN^*$, $\epsilon \in (-\frac{1}{2},\frac{1}{2})$ and
$\left(x_i\right)_{i \in \NN^*} \in S_X^{\NN^*}$, 
\[
 \sum_{j \in \NN^*, j \neq i} f\left\{ v_i - v_j + \epsilon\left(x_i-x_j\right) \right\} \leq 2^d d \sum_{j \in \NN} \frac{\left(j+\frac{3}{2}\right)^{d-1}}{1+j^{d+1}}.
\]
\end{lem}

\begin{lem} \label{lem: sommabilité2}
 Let $f: \RR^d \to \RR^+$, so that $f\left(t\right) \leq \frac{1}{1+|t|_{\infty}^{d+1}}$.
We consider $\delta < \frac{1}{2}$. Then, for all $i \in \NN^*$, $a >0$, $\epsilon \in [-\delta,\delta]$ and $\left(x_i\right)_{i \in \NN^*} \in S_X^{\NN^*}$,
\[
 \sum_{j \in \NN^*, j \neq i} f\left[ a\left\{ v_i - v_j + \epsilon\left(x_i-x_j\right) \right\} \right] \leq 2^d d \sum_{j \in \NN} \frac{\left(j+\frac{3}{2}\right)^{d-1}}{1+a^{d+1} \left(j+1 - 2 \delta\right)^{d+1}}.
\]
\end{lem}
\begin{proof}
Similar to the proof of lemma \ref{lem: sommabilité}.
\end{proof}

\begin{lem} \label{lem: sommabilité3}
 Let $f: \RR^d \to \RR^+$, so that $f\left(t\right) \leq \frac{1}{1+|t|_{\infty}^{d+1}}$. Then, for all $i \in \NN^*$, $N \in \NN^*$ and $\left(x_i\right)_{i \in \NN^*} \in S_X^{\NN^*}$,
\[
 \sum_{j \in \NN^*, |v_i-v_j|_{\infty} \geq N} f\left\{ v_i - v_j + \epsilon\left(x_i-x_j\right) \right\} \leq 2^d d \sum_{j \in \NN, j \geq N-1} \frac{\left(j+\frac{3}{2}\right)^{d-1}}{1+j^{d+1}}.
\]
\end{lem}
\begin{proof}
Similar to the proof of lemma \ref{lem: sommabilité}.
\end{proof}

\begin{prop} \label{prop: minorationValeursPropres}

Assume that condition \ref{cond: Ktheta} is satisfied.

For all $0 \leq \delta < \frac{1}{2}$, there exists
$C_{\delta} >0$ so that for all $|\epsilon| \leq \delta $, for all $\theta \in \Theta$, for all $n \in \NN^*$ and for all $x \in \left(S_X\right)^n$,
the eigenvalues of $R_{\theta}$ are larger than $C_{\delta}$.
\end{prop}

\begin{lem}  \label{lem: controle_valeurs_propres}

Assume that condition \ref{cond: Ktheta} is satisfied.

For all $ |\epsilon| < \frac{1}{2}$ and for all $K \in \NN$, there exists $C_{\epsilon,K}$ so that the eigenvalues of $R_{\theta}^{-1}$ and of
$\frac{\partial^q R_{\theta}}{\partial \theta_{i_1},...,\partial \theta_{i_q}}$, $0 \leq q \leq K$, $1 \leq i_1,...,i_q \leq p$,
are bounded by $C_{\epsilon,K}$, uniformly in $n \in \NN$, $x \in \left(S_X\right)^n$ and $\theta \in \Theta$.
\end{lem}

\begin{proof}
 Using, proposition \ref{prop: minorationValeursPropres}, we control the eigenvalues of
$R_{\theta}^{-1}$ uniformly in $x$ and $\theta$.

With \eqref{eq: controleCov} and lemma \ref{lem: sommabilité}, and using Gershgorin circle theorem,
we control the eigenvalues of $\frac{\partial^q R_{\theta}}{\partial \theta_{i_1},...,\partial \theta_{i_q}}$.

\end{proof}

\begin{lem} \label{lem: equivalenceDiag}
For $M$ symmetric real non-negative matrix, $\inf_{i} \phi_i(\diag(M)) \geq \inf_i \phi_i(M)$
and $\sup_{i} \phi_i(\diag(M)) \leq \sup_i \phi_i(M)$.
Furthermore, if for two sequences of symmetric matrices $M_n$ and $N_n$, $M_n \sim N_n$, then $\diag\left(M_n\right) \sim \diag\left(N_n\right) $.
\end{lem}
\begin{proof}
We use $ M_{i,i}  = e_i^t M e_i$, where $(e_i)_{i=1...n}$ is the standard basis of $\RR^n$. Hence $\inf_i \phi_i(M) \leq M_{i,i} \leq \sup_i \phi_i(M)$ for a symmetric real non-negative matrix $M$.
We also use $\left|\diag\left(M\right)\right| \leq |M| $.
\end{proof}

The next proposition gives a law of large numbers for the matrices that can be written using only matrix multiplications, the matrix
$R_{\theta}^{-1}$, the matrices $\frac{\partial^k}{\partial \theta_{i_1} ,..., \partial \theta_{i_k}} R_{\theta}$, for $i_1,...,i_k \in \{1,...,p \}$,
the $\diag$ operator applied to the symmetric products of matrices $R_{\theta}$,
$R_{\theta}^{-1}$ and $\frac{\partial^k}{\partial \theta_{i_1} ,..., \partial \theta_{i_k}} R_{\theta}$, and the matrix $\diag\left(R_{\theta}^{-1}\right)^{-1}$.
Examples of sums of these matrices are the matrices $M_{ML}^{i,j}$, $M_{CV,1}^{i,j}$ and $M_{CV,2}^{i,j}$ of propositions \ref{prop: normaliteML}
and \ref{prop: gradients_CV}.

\begin{prop} \label{prop: convergenceTrace}

Assume that condition \ref{cond: Ktheta} is satisfied.

Let $\theta \in \Theta$.
We denote the set of multi-indexes $S_p := \cup_{k \in \{0,1,2,3 \}} \left\{ 1,...,p \right\}^k$.
For $ I = \left(i_1,...,i_k\right) \in  S_p$, we denote $m\left(I\right) = k$.
Then, we denote for $I \in  S_p \cup \left\{ -1 \right\}$, 
\[
R_{\theta}^I := 
 \begin{cases} \frac{ \partial^{m\left(I\right)}  }{ \partial \theta_{i_1},...,\partial \theta_{i_{m\left(I\right)}} } R_{\theta} &\mbox{if} ~ ~ I \in S_p \\
R_{\theta}^{-1} &\mbox{if} ~ ~ I=-1 \end{cases}.
\]

We then denote
\begin{itemize}
 \item $M_{nd}^I = R_{\theta}^I$ ~ ~ \mbox{for $I \in S_{nd} := \left( S_p \cup \{-1\} \right)$}
 \item $M_{sd}^1 = \diag\left(R_{\theta}^{-1}\right)^{-1}$ 
 \item $M_{bd}^I = \diag\left( R_{\theta}^{I_1}...R_{\theta}^{I_{m\left(I\right)}} \right) $ ~ ~ \mbox{for $I \in S_{bd} := \cup_{k \in \NN^*} S_{nd}^k$}
\end{itemize}

We then define $\mathcal{M}_{\theta}$ as the set of sequences of random matrices
(defined on $(\Omega_X,\mathcal{F}_X,P_X)$), indexed by $n \in \NN^*$, dependent on $X$, which can be written
$M_{d_1}^{I_1}...M_{d_K}^{I_K}$
with $\{d_1,I_1\},...,\{d_K,I_K\} \in \left(\{nd\} \times S_{nd}\right) \cup \left(\{sd \} \times \{1\}\right) \cup \left(\{ bd \} \times S_{bd}\right) $,
and so that, for the matrices $M_{d_j}^{I_j}$, so that $d_j = bd$, the matrix $R_{\theta}^{\left(I_j\right)_1}...R_{\theta}^{\left(I_j\right)_{m\left(I_j\right)}}$
is symmetric.

Then, for every matrix $M_{d_1}^{I_1}...M_{d_K}^{I_K}$ of $\mathcal{M}_{\theta}$, the singular values of
$M_{d_1}^{I_1}...M_{d_K}^{I_K}$ are bounded uniformly in $\theta$, $n$ and $x \in \left(S_X\right)^n$.
Then, denoting $S_n := \frac{1}{n} {\rm{Tr}}\left( M_{d_1}^{I_1}...M_{d_K}^{I_K} \right)$,
there exists a deterministic limit $S$, which only depends on $\epsilon$, $\theta$ and $\left(d_1,I_1\right),...,\left(d_K,I_K\right)$,
so that $S_n \to S$ $P_X$-almost surely. Hence $S_n \to S$ in quadratic mean and $\var\left(S_n\right) \to 0$ as $n \to + \infty$.
\end{prop}

\begin{prop} \label{prop: convergence_forme_quadratique}

Assume that condition \ref{cond: Ktheta} is satisfied.

 Let $M \in \mathcal{M}_{\theta}$ (proposition \ref{prop: convergenceTrace}). Then, $\frac{1}{n} y^t M y$ converges to
$\Sigma := \lim_{n \to +\infty} \frac{1}{n} {\rm{Tr}}\left(M R_{\theta_0}\right)$, in the mean square sense (on the product space).
\end{prop}

\begin{prop} \label{prop: lindeberg_presque_sur}

Assume that condition \ref{cond: Ktheta} is satisfied.

We recall $X \sim \mathcal{L}_X^{\otimes n}$ and $y_i = Y\left(i + \epsilon X_i\right)$, $1 \leq i \leq n$.
We consider symmetric matrix sequences $M_1,...,M_p$ and $N_1,...,N_p$
(defined on $(\Omega_X,\mathcal{F}_X,P_X)$), functions of $X$, so that the eigenvalues of $N_1,...,N_p$
are bounded uniformly in $n$ and $x \in \left(S_X\right)^n$,
${\rm{Tr}}\left( M_i + N_iR \right) = 0$ for $1 \leq i \leq p$ and there exists a $p \times p$ matrix $\Sigma$ so that
$ \frac{1}{n}{\rm{Tr}}\left(N_iRN_jR\right) \to \left(\Sigma\right)_{i,j}$
$P_X$-almost surely. 
Then the sequence of $p$-dimensional random vectors (defined on the product space)
$\left( \frac{1}{\sqrt{n}} \left\{ {\rm{Tr}}\left(M_i\right) + y^t N_i y\right\} \right)_{i=1...p}$
converges in distribution to a Gaussian random vector with mean zero and covariance matrix $2 \Sigma$.
\end{prop}

\begin{prop} \label{prop: normaliteEstimateur}
We recall $X \sim \mathcal{L}_X^{\otimes n}$ and $y_i = Y\left(i + \epsilon X_i\right)$, $1 \leq i \leq n$.
We consider a consistent estimator $\hat{\theta} \in \RR^p$ so that $ \mathbb{P} \left(c\left( \hat{\theta} \right) = 0\right) \to 1$,
for a function $c: \Theta \to \RR^p$, dependent on $X$ and $Y$, and twice differentiable in $\theta$.
We assume that $\sqrt{n} c\left(\theta_0\right) \to_{\L}
\N\left(0,\Sigma_1\right)$, for a $p \times p$ matrix $\Sigma_1$ and that the matrix $\frac{ \partial c\left(\theta_0\right) }{ \partial \theta} $
converges in probability to a $p \times p$ positive matrix $\Sigma_2$
(convergences are defined on the product space).
Finally we assume
that $\sup_{\tilde{\theta},i,j,k} \left| \frac{\partial^2}{\partial \theta_i \partial \theta_j} c_k\left(\tilde{\theta}\right) \right|$ is bounded in probability.

Then
\[
 \sqrt{n} \left( \hat{\theta} - \theta_0 \right) \to_{\L} \N\left(0, \Sigma_2^{-1} \Sigma_1 \Sigma_2^{-1} \right).
\]
 
\end{prop}

Proposition \ref{prop: normaliteEstimateur} can be proved using standard $M$-estimator techniques. In subsection \ref{subsection: proof_technical_results}
we give a short proof for consistency.

\subsection{Proof of the technical results} \label{subsection: proof_technical_results}

\paragraph{Proof of lemma \ref{lem: sommabilité}}

\begin{proof}
 \begin{eqnarray*}
\sum_{j \in \NN, j \neq i} f\left\{ v_i - v_j + \epsilon\left(x_i-x_j\right) \right\} & \leq & \sum_{v \in \ZZ^d, v \neq 0} \sup_{\delta_v \in [-1,1]^d} f\left( v+\delta_v \right) \\
& = & \sum_{j \in \NN} \sum_{ v \in \{-j-1,...,j+1\}^d \backslash \{-j,...,j\}^d } \sup_{\delta_v \in [-1,1]^d}  f\left(v + \delta_v\right).
\end{eqnarray*}
For $v \in \{-j-1,...,j+1\}^d \backslash \{-j,...,j\}^d$, $|v + \delta_v|_{\infty} \geq j$. The cardinality of the set
$\{-j-1,...,j+1\}^d \backslash \{-j,...,j\}^d$ is
\[
 \left(2j+3\right)^d - \left(2j+1\right)^d  = \int_{2j+1}^{2j+3} d.t^{d-1} dt \leq 2d \left(2j+3\right)^{d-1} = 2^d d \left(j + \frac{3}{2}\right)^{d-1}.
\]
Hence
\[
 \sum_{j \in \NN} f\left\{ v_i - v_j + \epsilon\left(x_i-x_j\right) \right\} \leq  \sum_{j \in \NN} 2^d d \left(j+\frac{3}{2}\right)^{d-1} \frac{1}{1+j^{d+1}}.
\]
\end{proof}

\paragraph{Proof of proposition \ref{prop: minorationValeursPropres}}
\begin{proof}

Let $h: \RR^d \to \RR$ so that $\hat{h}(f) = \prod_{i=1}^d \hat{h}_i(f_i)$, with
$\hat{h}_i(f_i) = 
\mathbf{1}_{f_i^2 \in [0,1]} \exp{\left( - \frac{1}{ 1 - f_i^2   } \right)}$.
For $1 \leq i \leq n$, $\hat{h}_i: \RR \to \RR$ is $C^{\infty}$, with compact support, so there exists $C>0$ so that $|h_i(t_i)| \leq \frac{C^{\frac{1}{d}}}{1+ |t_i|^{d+1}}$. Now, since the inverse Fourier transform of $\prod_{i=1}^d \hat{h}_i(f_i)$ is
$\prod_{i=1}^d h_i(t_i)$, we have
\[
|h(t)| \leq  C \prod_{i=1}^d \frac{1}{1+ |t_i|^{d+1}}
\leq   \frac{C}{1+ |t|_{\infty}^{d+1}}.
\]

Hence, from lemma \ref{lem: sommabilité2}, for all $i \in \NN$ and $a >0$,
\begin{equation} \label{eq: in_proof_minoration_eigen_values_1}
 \sum_{j \in \NN, j \neq i} \left| h\left[ a\left\{ v_i-v_j + \epsilon \left(x_i - x_j\right) \right\} \right] \right| \leq C 2^d d \sum_{j \in \NN} \frac{\left(j+\frac{3}{2} \right)^{d-1}}{1+a^{d+1} \left(j+1 - 2 \delta\right)^{d+1}}.
\end{equation}
The right-hand term in \eqref{eq: in_proof_minoration_eigen_values_1} goes to zero
when $a \to + \infty$. Also, $h(0)$ is positive,
because $\hat{h}$ is non-negative and is not almost
surely zero on $\RR^d$ with respect to the Lebesgue measure. Thus, there exists $ 0 < a < \infty$ so that for all $i \in \NN$,
\begin{equation} \label{eq: in_proof_minoration_eigen_values_2} 
 \sum_{j \in \NN, j \neq i} \left| h\left[ a\left\{ v_i-v_j + \epsilon \left(x_i - x_j\right) \right\} \right] \right| \leq  \frac{1}{2} h\left(0\right).
\end{equation}

Using Gershgorin circle theorem, for any $n \in \NN^*$, $x_1,...,x_n \in S_X$, the eigenvalues of
the symmetric matrices $\left( h\left[ a\left\{ v_i-v_j + \epsilon \left(x_i - x_j\right) \right\} \right] \right)_{1 \leq i,j \leq n}$
belong to the balls with center $h(0)$
and radius 
$\sum_{1 \leq j \leq n, j \neq i} \left| h\left[ a\left\{ v_i-v_j + \epsilon \left(x_i - x_j\right) \right\} \right] \right|$. Thus, because of \eqref{eq: in_proof_minoration_eigen_values_2}, these eigenvalues belong to the segment $[ h(0) - \frac{1}{2} h(0) , h(0) + \frac{1}{2} h(0)]$ and are larger than $\frac{1}{2} h(0)$. 

Hence, for all $n$, $t_1,...,t_n \in \RR$, $x_1,...,x_n \in S_X$,
\begin{eqnarray*}
 \frac{1}{2} h\left(0\right) \sum_{i=1}^n t_i^2 & \leq & \sum_{i,j=1}^n t_i  t_j h\left[ a  \left\{ v_i -v_j + \epsilon\left( x_i - x_j\right) \right\} \right]  \\
& = & \sum_{i,j=1}^n t_i  t_j \frac{1}{a^d} \int_{\RR^d} \hat{h}\left(\frac{f}{a}\right) e^{ \mathrm{i} f \cdot \left\{ v_i -v_j + \epsilon \left( x_i - x_j \right) \right\} } df  \\
& = &  \frac{1}{a^d} \int_{\RR^d} \hat{h}\left(\frac{f}{a}\right) \left| \sum_{i=1}^n t_i e^{ \mathrm{i} f \cdot \left(v_i + \epsilon x_i\right) } \right|^2 df.
\end{eqnarray*}
Hence, as $\hat{K}_{\theta}\left(f\right): (\Theta \times \RR^d) \to \RR$ is continuous and positive, using a compacity argument, there exists $C_2 > 0$ so that for all $\theta \in \Theta$, $f \in [-a,a]^d$,
$ \hat{K}_{\theta}\left(f\right) \geq C_2 \hat{h}\left(\frac{f}{a}\right)  $. Hence,
\begin{eqnarray*}
 \frac{1}{2} h\left(0\right) \sum_{i=1}^n t_i^2 & \leq & \frac{1}{a^d C_2}  \int_{\RR^d} \hat{K}_{\theta} \left(f\right) \left| \sum_{i=1}^n t_i e^{ \mathrm{i} f \cdot \left( v_i + \epsilon x_i\right) } \right|^2 df, \\
& = & \frac{1}{a^d C_2} \sum_{i,j=1}^n t_i  t_j K_{\theta}\left\{ v_i-v_j + \epsilon \left(x_i - x_j\right) \right\}.
\end{eqnarray*}

\end{proof}

\paragraph{Proof of proposition \ref{prop: convergenceTrace}}
\begin{proof}
Let $M_{d_1}^{I_1}...M_{d_K}^{I_K} \in \mathcal{M}_{\theta}$ be fixed in the proof.
 
The eigenvalues of $R_{\theta}^I$, $I \in S_{nd}$, are bounded uniformly with respect to $n$, $\theta$ and $x$ (lemma \ref{lem: controle_valeurs_propres}).
Then, using lemma Appendix D.5 on $R_{\theta}$
and $R_{\theta}^{-1}$ and using lemma \ref{lem: equivalenceDiag}, we show that the eigenvalues of $\diag\left(R_{\theta}^{-1}\right)^{-1}$
are bounded uniformly in $x$, $n$ and $\theta$.
Then, for $M_{bd}^I = \diag\left( R_{\theta}^{I_1}...R_{\theta}^{I_{m\left(I\right)}} \right)$, the eigenvalues of
$ R_{\theta}^{I_1}...R_{\theta}^{I_{m\left(I\right)}} $ are bounded by the product of the maximum eigenvalues of  $R_{\theta}^{I_1},...,R_{\theta}^{I_{m\left(I\right)}} $.
As the number of term in this product is fixed and independent of $n$, the resulting bound is finite
and independent of $n$.
Hence we use lemma \ref{lem: equivalenceDiag} to show that the eigenvalues of $M_{bd}^I$ are bounded uniformly in $n$, $\theta$ and $x$.
Finally we use $||A_1...A_K|| \leq ||A_1||...||A_K||$ to show that $||M_{d_1}^{I_1}...M_{d_K}^{I_K}||$ is bounded uniformly in
$n$, $\theta$ and $x$.

We decompose $n$ into $n = N_1^d n_2 + r$ with $N_1,n_2,r \in \NN$ and $r < N_1^d$. We define $C\left(v_i\right)$ as the unique $u \in \NN^d$ so that
$v_i \in E_u :=  \prod_{k=1}^d\{ N_1 u_k +1 ,...,N_1 \left(u_k+1\right) \}$. 

We then define the sequence of matrices
$\tilde{R_{\theta}}$ by
$\left(\tilde{R_{\theta}}\right)_{i,j} = \left(R_{\theta}\right)_{i,j} \mathbf{1}_{ C\left(v_i\right) = C\left(v_j\right) }$.
We denote $\tilde{M}_{d_1}^{I_1}...\tilde{M}_{d_K}^{I_K}$ the matrix built by replacing $R_{\theta}$ by $\tilde{R}_{\theta}$
in the expression of $M_{d_1}^{I_1}...M_{d_K}^{I_K}$ (we also make the substitution for the inverse and the partial derivatives). 

\begin{lem}  \label{lem: Mtilde}
 $\left| \tilde{M}_{d_1}^{I_1}...\tilde{M}_{d_K}^{I_K} - M_{d_1}^{I_1}...M_{d_K}^{I_K} \right|^2 \to 0$, uniformly in $x \in \left(S_X\right)^n$, when $N_1 , n_2 \to \infty$.
\end{lem}
\begin{proof}

Let $\delta >0$ and $N$ so that $T_N := C_0^2 2^{2d} d^2 \sum_{j \in \NN , j \geq N-1} \frac{ \left(j+\frac{3}{2}\right)^{2\left(d-1\right)}}{\left(1+j^{d+1}\right)^2} \leq \delta$. Then:
\begin{eqnarray*}
 \left| \tilde{R_{\theta}} - R_{\theta} \right|^2 & = & \frac{1}{n} \sum_{i,j=1}^n \left\{ \left(R_{\theta}\right)_{i,j} - \left(\tilde{R_{\theta}}\right)_{i,j} \right\}^2, \\
& = & \frac{1}{n} \sum_{ 1 \leq i,j \leq n, C\left(v_i\right) \neq C\left(v_j\right) } K_{\theta}^2\left\{ v_i-v_j + \epsilon\left(x_i-x_j\right) \right\},  \\
& \leq & \frac{1}{n} \sum_{i=1}^n \sum_{ j \in \NN^*,  C\left(v_i\right) \neq C\left(v_j\right) } K_{\theta}^2\left\{ v_i-v_j + \epsilon\left(x_i-x_j\right) \right\}. \\
\end{eqnarray*}
There exists a unique $a$ so that $\left(aN_1\right)^d \leq n < \left\{\left(a+1\right)N_1\right\}^d$. Among the $n$ observation points, $\left(a N_1\right)^d$ are in the $E_u$, for $u \in \{1,...,a\}^d$.
The number of remaining points is less than $d N_1 \left\{ \left(a+1\right)N_1 \right\}^{d-1}$. Therefore, using \eqref{eq: controleCov},

\begin{eqnarray*}
 \left| \tilde{R_{\theta}} - R_{\theta} \right|^2 & \leq &
\frac{1}{n} \sum_{u \in \{1,...,a\}^d} \sum_{1\leq i \leq n,  v_i \in E_u} \sum_{j \in \NN^*, C\left(v_j\right) \neq C\left(v_i\right)} K_{\theta}^2\left\{ v_i-v_j + \epsilon\left(x_i-x_j\right) \right\}
+ \frac{1}{n} d N_1 \left\{ \left(a+1\right)N_1 \right\}^{d-1} T_0, \\
& = & \frac{1}{n} \sum_{u \in \{1,...,a\}^d} \sum_{1\leq i \leq n, v_i \in E_u} \sum_{j \in \NN^*, C\left(v_j\right) \neq C\left(v_i\right)} K_{\theta}^2\left\{ v_i-v_j + \epsilon\left(x_i-x_j\right) \right\} + o\left(1\right).
\end{eqnarray*}
Then, for fixed $u$, the
cardinality of the set of the integers $i \in \{1,...,n\}$, so that $v_i \in E_u$ and there exists $j \in \NN^*$ so that $C(v_i) \neq C(v_j)$ and $|v_i - v_j|_{\infty} \leq N$
is $ N_1^d - \left(N_1 - 2N\right)^d $ and is less than $2 N d N_1^{d-1}$.
Hence, using \eqref{eq: controleCov}, lemmas \ref{lem: sommabilité} and \ref{lem: sommabilité3}, 
\begin{eqnarray*}
 \left| \tilde{R_{\theta}} - R_{\theta} \right|^2 & \leq & 
\frac{1}{n} \sum_{u \in \{1,...,a\}^d}  \left( 2 N d N_1^{d-1} T_0 + N_1^d T_N\right)   + o\left(1\right) \\
& \leq & \frac{1}{ a^d N_1^d } a^d \left\{ \left( 2 N d N_1^{d-1} T_0 + N_1^d T_N\right) \right\} + o\left(1\right).
\end{eqnarray*}

This last term is smaller than $2 \delta$ for $N_1$ and $n_2$ large enough. Hence we showed
$\left| \tilde{R_{\theta}} - R_{\theta} \right| \to 0$ uniformly in $x$, when $N_1 , n_2 \to \infty$. We can show the same result for
$\frac{\partial^k R_{\theta}}{\partial \theta_{i_1} ,..., \partial \theta_{i_k}} $ and
$\frac{\partial^k \tilde{R}_{\theta}}{\partial \theta_{i_1} ,..., \partial \theta_{i_k}} $.
Finally we use \cite{TCMR} theorem 1 to show that
$\left| \tilde{R}_{\theta}^{-1} - R_{\theta}^{-1} \right| \to 0$ uniformly in $x$, when $N_1 , n_2 \to \infty$.
Hence, using \cite{TCMR}, theorem 1 and lemma \ref{lem: equivalenceDiag},
$|\tilde{M}_{d}^I - M_{d}^I|$ converges to $0$ uniformly in $x$ when $N_1 , n_2 \to \infty$, for $d \in \{nd,sd,bd \}$
and $I \in S_{nd} \cup \{ 1 \} \cup S_{bd}$.
We conclude using \cite{TCMR}, theorem 1.

\end{proof}

We denote, for every $N_1, n_2$ and $r$, with $0 \leq r < N_1^d$, $n = N_1^d n_2 + r$ and
$S_{N_1,n_2} := \frac{1}{n} {\rm{Tr}}\left( \tilde{M}_{d_1}^{I_1}...\tilde{M}_{d_K}^{I_K} \right)$,
which is a sequence of real random variables defined on $(\Omega_X,\mathcal{F}_X,P_X)$ and indexed by $N_1$, $n_2$ and $r$. 
Using \cite{TCMR}, corollary 1 and lemma \ref{lem: Mtilde}, $|S_n - S_{N_1,n_2}| \to 0$ uniformly in $x$ when $N_1,n_2 \to \infty$
(uniformly in $r$). 
As the matrices in the expression of $S_{N_1,n_2}$ are block diagonal,
we can write $S_{N_1,n_2} = \frac{1}{n_2} \sum_{l=1}^{n_2} S_{N_1^d}^l + o\left( \frac{1}{n_2} \right)$, where the $S_{N_1^d}^l$ are
$iid$ random variables defined on $(\Omega_X,\mathcal{F}_X,P_X)$ with the distribution of $S_{N_1^d}$. 
We denote $\bar{S}_{N_1^d} := E_X \left( S_{N_1^d} \right)$.
Then, using the strong law of large numbers, for fixed $N_1$, $S_{N_1,n_2} \to \bar{S}_{N_1^d}$ $P_X$-almost surely when
$n_2 \to \infty$ (uniformly in $r$).

For every $N_1, p_{N_1},n_2 \in \NN^*$, there exist two unique $n_2',r \in \NN$ so that
$0 \leq r < N_1^d$ and
$\left(N_1 + p_{N_1}\right)^d n_2 = N_1^d n_2' + r$. 
Then we have
\begin{eqnarray} \label{eq: ABCDE}
 | \bar{S}_{\left(N_1\right)^d} - \bar{S}_{\left(N_1 + p_{N_1}\right)^d} | & \leq & | \bar{S}_{\left(N_1\right)^d} - S_{N_1,n_2'} | +  | S_{N_1,n_2'} - S_{N_1^d n_2'+r} | \\
&  & + | S_{N_1^d n_2'+r} - S_{ \left(N_1 + p_{N_1} \right)^d n_2  } | + | S_{ \left(N_1 + p_{N_1} \right)^d n_2  } - S_{ N_1 + p_{N_1} , n_2  } | 
 +  | S_{ N_1 + p_{N_1} , n_2  } - \bar{S}_{\left(N_1 + p_{N_1}\right)^d} |  \nonumber \\ 
& = & A + B + C + D + E. \nonumber
\end{eqnarray}
Because $n_2'$ and $r$ depend on $N_1$, $p_{N_1}$ and $n_2$, $A$, $B$, $C$, $D$ and $E$ are sequences of random variables
defined on $(\Omega_X,\mathcal{F}_X,P_X)$ and indexed by $N_1$, $p_{N_1}$ and $n_2$.
We have seen that there exists $\tilde{\Omega}_X \subset \Omega_X$, with $P_X(\tilde{\Omega}_X) = 1$ so that for $\omega_X \in \tilde{\Omega}_X$,
when $N_1,n_2 \to + \infty$, we also have $N_1+p_{N_1},n_2' \to + \infty$, and so $B$ and $D$ converge to zero.

Now, for every $N_1 \in \NN^*$, let $\Omega_{X,N_1}$ be so that $P_X( \Omega_{X,N_1} )=1$ and for all
$\omega_X \in \Omega_{X,N_1}$, $S_{N_1,n_2} \to_{n_2 \to +\infty} \bar{S}_{N_1^d}$.
Let $\tilde{\tilde{\Omega}}_X = \cap_{N_1 \in \NN^*} \Omega_{X,N_1}$. Then $P_X(\tilde{\tilde{\Omega}}_X) = 1$ and for all
$\omega_X \in \tilde{\tilde{\Omega}}_X$, for all $N_1 \in \NN^*$, $S_{N_1,n_2} \to_{n_2 \to + \infty} \bar{S}_{N_1^d}$.

We will now show that the $N_1$-indexed sequence $\bar{S}_{N_1^d}$ is a Cauchy sequence.
Let $\delta >0$. 
$P_X( \tilde{\tilde{\Omega}} \cap \tilde{\Omega} ) = 1$ so this set is non-empty. Let us fix
$\omega_X \in \tilde{\tilde{\Omega}} \cap \tilde{\Omega}$.
In \eqref{eq: ABCDE}, $C$ is null.
There exist $\bar{N}_1$ and $\bar{N_2}$ so that for every $N_1 \geq \bar{N}_1$, $n_2 \geq \bar{n}_2$, $p_{N_1} > 0$, $B$ and $D$ are smaller than $\delta$.
Let us now fix any $N_1 \geq \bar{N}_1$. Then, for every $p_{N_1} > 0$, with $n_2 \geq \bar{n}_2$ large enough,
$A$ and $E$ are smaller than $\delta$.

Hence, we showed, for the $\omega_X \in \tilde{\tilde{\Omega}} \times \tilde{\Omega}$ we were considering, that the $N_1$-indexed sequence $\bar{S}_{N_1^d}$ is a Cauchy sequence and we denote its limit by $S$.
Since this sequence is deterministic, $S$ is deterministic and $\bar{S}_{\left(N_1\right)^d} \to_{N_1 \to + \infty} S$.

Finally, let $n = N_1^d n_2 + r$ with $N_1,n_2 \to \infty$. Then
\[
 | S_n - S | \leq | S_n - S_{N_1,n_2} | + | S_{N_1,n_2} - \bar{S}_{N_1^d} | + | \bar{S}_{N_1^d} - S |.
\]
Using the same arguments as before, we show that, $P_X$-a.s., $|S_n-S| \to 0$ as $n \to + \infty$.

\end{proof}

\paragraph{Proof of proposition \ref{prop: convergence_forme_quadratique}}
\begin{proof}%

 $\EE\left( \frac{1}{n} y^t M y \right)= \EE\left\{ \EE\left( \frac{1}{n} y^t M y | X\right) \right\}
= \EE\left\{ \frac{1}{n} {\rm{Tr}}\left(M R_0\right) \right\} \to \Sigma$. 
Furthermore $\var\left( \frac{1}{n} y^t M y \right) = \EE\left\{ \var\left( \frac{1}{n} y^t M y | X \right) \right\}
+ \var\left\{ \EE\left(\frac{1}{n} y^t M y | X\right) \right\}$.
$\var\left( \frac{1}{n} y^t M y | X = x \right)$ is
a $O\left(\frac{1}{n}\right)$, uniformly in $x$, using proposition \ref{prop: convergenceTrace} and $||A+B|| \leq ||A||+||B||$.
Therefore $\var\left( \frac{1}{n} y^t M y | X \right)$ is bounded by $O(\frac{1}{n})$ $P_X$-a.s.
$\var\left\{ \EE\left( \frac{1}{n} y^t M y | X\right) \right\} =
\var \left\{ \frac{1}{n} {\rm{Tr}}\left( M R_{\theta_0} \right) \right\} \to 0$, using proposition \ref{prop: convergenceTrace}.
Hence $ \frac{1}{n} y^t M y $ converges to $\Sigma$ in the mean square sense.
 
\end{proof}

\paragraph{Proof of proposition \ref{prop: lindeberg_presque_sur}}
\begin{proof}
 Let $v_{\lambda} = \left(\lambda_1,...,\lambda_p\right)^t \in \RR^p$.
\begin{eqnarray*}
 \EE \left(  \exp{ \left[ \mathrm{i} \sum_{k=1}^p \lambda_k \frac{1}{\sqrt{n}} \left\{ {\rm{Tr}}\left(M_k\right) + y^t N_k y \right\} \right] } \right)
& = & \EE \left\{  \EE \left(  \exp{ \left[ \mathrm{i} \sum_{k=1}^p \lambda_k \frac{1}{\sqrt{n}} \left\{ {\rm{Tr}}\left(M_k\right) + y^t N_k y \right\} \right] } \big| X \right) \right\}. \\ 
\end{eqnarray*}

For fixed $x = \left(x_1,...,x_n\right)^t \in \left(S_X\right)^{n}$, denoting
$\sum_{k=1}^p \lambda_k R^{\frac{1}{2}} N_k R^{\frac{1}{2}} = P^t D P$, with $P^t P = I_n$ and $D$ diagonal,
$z_x = P R^{-\frac{1}{2}}y$ (which is a vector of $iid$ standard Gaussian variables,
conditionally to $X=x$), we have
\begin{eqnarray*}
\sum_{k=1}^p \lambda_k \frac{1}{\sqrt{n}} \left\{ {\rm{Tr}}\left( M_k \right) + y^t N_k y \right\}
& = & \frac{1}{\sqrt{n}} \left[ {\rm{Tr}}\left( \sum_{k=1}^p \lambda_k M_k \right) +
\sum_{i=1}^n \phi_i \left( \sum_{k=1}^p \lambda_k R^{\frac{1}{2}} N_k R^{\frac{1}{2}} \right) (z_x)_i^2 \right] \\
& = & \frac{1}{\sqrt{n}} \left[ \sum_{i=1}^n  \phi_i
\left( \sum_{k=1}^p \lambda_k R^{\frac{1}{2}} N_k R^{\frac{1}{2}} \right) \left\{(z_x)_i^2 -1\right\} \right]. \\
\end{eqnarray*}
Hence
\begin{eqnarray*}
\var \left[  \sum_{k=1}^p \lambda_k \frac{1}{\sqrt{n}} \left\{ {\rm{Tr}}\left( M_k \right) + y^t N_k y \right\}   \big| X \right] & = &
\frac{2}{n} \sum_{i=1}^n \phi_i^2\left( \sum_{k=1}^p \lambda_k R^{\frac{1}{2}} N_k R^{\frac{1}{2}}  \right)  \\
& = & \frac{2}{n} \sum_{k=1}^p \sum_{l=1}^p \lambda_k \lambda_l {\rm{Tr}}\left( R N_k R N_l \right)  \\
& \underset{n \to + \infty}{\to} & v_{\lambda}^t \left(2 \Sigma\right) v_{\lambda} ~ ~ \mbox{ for a.e. $\omega_X$ }. 
\end{eqnarray*}
Hence, for almost every $\omega_X$, we can apply Lindeberg-Feller criterion to the $\Omega_Y$-measurable variables
$ \frac{1}{\sqrt{n}}  \phi_i \left( \sum_{k=1}^p \lambda_k R^{\frac{1}{2}} N_k R^{\frac{1}{2}} \right) \left\{\left(z_x\right)_i^2 -1\right\}$, $1 \leq i \leq n$,
to show that $\sum_{k=1}^p \lambda_k \frac{1}{\sqrt{n}} \left\{ {\rm{Tr}}\left( M_k \right) + y^t N_k y \right\}$ converges
in distribution to $\mathcal{N}\left(0,v_{\lambda}^t \left(2 \Sigma\right) v_{\lambda}\right)$. 
Hence, $\EE \left(  \exp{ \left[ \mathrm{i} \sum_{k=1}^p \lambda_k \frac{1}{\sqrt{n}} \left\{ {\rm{Tr}}\left(M_k\right) +  y^t N_k y \right\} \right] } \big| X \right)$
converges for almost every $\omega_X$ to
$ \exp{ \left( - \frac{1}{2} v_{\lambda}^t \left(2 \Sigma\right) v_{\lambda} \right) } $. Using the dominated convergence theorem
on $(\Omega_X,\mathcal{F}_X,P_X)$,
$ \EE \left( \exp{ \left[ \mathrm{i} \sum_{k=1}^p \lambda_k \frac{1}{\sqrt{n}} \left\{ {\rm{Tr}}\left(M_k\right) +  y^t N_k y \right\}  \right] } \right)$ converges to
$ \exp{ \left\{ - \frac{1}{2} v_{\lambda}^t \left(2 \Sigma\right) v_{\lambda} \right\} } $.

\end{proof}

\paragraph{Proof of proposition \ref{prop: normaliteEstimateur}}
\begin{proof}
It is enough to consider the case $c\left( \hat{\theta}\right) = 0$, the case $ P \left\{c\left( \hat{\theta} \right) = 0\right\} \to 1$ being deduced from it by modifying $c$
on a set with vanishing probability measure, which does not affect the convergence in distribution. For all $1 \leq k \leq p$
\begin{eqnarray*} 
0 = c_k\left( \hat{\theta} \right)
 = c_k\left( \theta_0 \right) + \left\{\frac{\partial}{\partial \theta} c_k\left( \theta_0 \right)\right\}^t \left( \hat{\theta} - \theta_0\right) + r,
\end{eqnarray*}
with random $r$, so that $|r| \leq \sup_{\tilde{\theta},i,j,k} \left| \frac{\partial^2}{\partial \theta_i \partial \theta_j} c_k\left(\tilde{\theta}\right) \right| \times
| \hat{\theta} - \theta_0 |^2$. 
Hence $r = o_p\left( | \hat{\theta} - \theta_0 | \right)$. We then have
\begin{eqnarray*} 
 - c_k\left(\theta_0\right) & = & \left[ \left\{ \frac{\partial}{\partial \theta} c_k\left(\theta_0\right) \right\}^t + o_p\left(1\right) \right] \left( \hat{\theta} - \theta_0\right),
\end{eqnarray*}
and so
\begin{equation} \label{eq: DLprob} 
\left( \hat{\theta} - \theta_0 \right)  = - \left\{ \frac{\partial}{\partial \theta} c\left(\theta_0\right) + o_p\left(1\right) \right\}^{-1} c\left( \theta_0 \right). 
\end{equation}
We conclude using Slutsky lemma.

\paragraph{Remark}
One can show that, with probability going to one as $n \to + \infty$, the likelihood has a unique global minimizer.
Indeed, we first notice that the set of the minimizers is a subset of any open ball of center $\theta_0$ with probability going to one.
For a small enough open ball, the probability that the likelihood function is strictly convex in this open ball converges to one.
This is because of the third-order regularity
of the likelihood with respect to $\theta$, and because the limit of the second derivative matrix of the Likelihood at $\theta_0$ is positive.

\end{proof}

\section{Exact expressions of the asymptotic variances at $\epsilon=0$ for $d=1$}  \label{section: Toeplitz}

In this section we only address the case $d=1$ and $p=1$, where the observation points $v_i + \epsilon X_i$, $1 \leq i \leq n$, $n \in \NN^*$,
are the $i + \epsilon X_i$, where $X_i$
is uniform on $[-1,1]$,
and $\Theta = [\theta_{inf},\theta_{sup}]$.

We define the Fourier transform function $\hat{s}\left(.\right)$ of a sequence $s_n$ of $\ZZ$ by $\hat{s}\left(f\right) = \sum_{n \in \ZZ} s_n e^{ \mathrm{i} s_n f }$ as
in \cite{TCMR}.
This function is $2 \pi$ periodic on $[-\pi,\pi]$. Then
\begin{itemize}
\item The sequence of the $K_{\theta_0}\left(i\right)$, $i \in \ZZ$, has Fourier transform $f$ which is even and non-negative on $[-\pi,\pi]$.
\item The sequence of the $\frac{\partial}{\partial \theta}K_{\theta_0}\left(i\right)$, $i \in \ZZ$, has Fourier transform $f_{\theta}$ which is even on $[-\pi,\pi]$.
\item The sequence of the $\frac{\partial}{\partial t}K_{\theta_0}\left(i\right) \mathbf{1}_{i \neq 0}$, $i \in \ZZ$, has Fourier transform $\mathrm{i} ~ f_{t}$ which is
odd and imaginary on $[-\pi,\pi]$.
\item The sequence of the $\frac{\partial}{\partial t} \frac{\partial}{\partial \theta} K_{\theta_0}\left(i\right) \mathbf{1}_{i \neq 0}$, $i \in \ZZ$, has Fourier transform
$\mathrm{i} ~ f_{t,\theta}$ which
is odd and imaginary on $[-\pi,\pi]$.
\item The sequence of the $\frac{\partial^2}{\partial t^2}K_{\theta_0}\left(i\right) \mathbf{1}_{i \neq 0}$, $i \in \ZZ$, has Fourier transform $f_{t,t}$ which is even on $[-\pi,\pi]$.
\item The sequence of the $\frac{\partial^2}{\partial t^2} \frac{\partial }{\partial \theta} K_{\theta_0}\left(i\right) \mathbf{1}_{i \neq 0}$, $i \in \ZZ$, has Fourier transform
$f_{t,t,\theta}$ which is even on $[-\pi,\pi]$.
\end{itemize}

In this section we assume in condition \ref{cond: toeplitz} that all these sequences are dominated
by a decreasing exponential function, so that the Fourier transforms are $C^{\infty}$.
This condition could be weakened, but it simplifies the proofs, and it is satisfied in our framework.

\begin{cond} \label{cond: toeplitz}
 There exist $C < \infty$ and $a >0$  so that the sequences of general terms
$K_{\theta_0}\left(i\right)$, $\frac{\partial}{\partial \theta}K_{\theta_0}\left(i\right)$, $\frac{\partial}{\partial t}K_{\theta_0}\left(i\right) \mathbf{1}_{i \neq 0}$,
$\frac{\partial}{\partial t} \frac{\partial}{\partial \theta} K_{\theta_0}\left(i\right) \mathbf{1}_{i \neq 0}$,
$\frac{\partial^2}{\partial t^2}K_{\theta_0}\left(i\right) \mathbf{1}_{i \neq 0}$,
$\frac{\partial^2}{\partial t^2} \frac{\partial }{\partial \theta} K_{\theta_0}\left(i\right) \mathbf{1}_{i \neq 0}$, $i \in \ZZ$,
are bounded by $C e^{-a |i|}$.
\end{cond}

For a $2 \pi$-periodic function $f$ on $[-\pi,\pi]$, we denote by $M\left(f\right)$ the mean value of $f$ on $[-\pi,\pi]$.

Then, proposition \ref{prop: simplification_calcul_d2_toeplitz} gives the closed form expressions of $\Sigma_{ML}$, $\Sigma_{CV,1}$,
$ \Sigma_{CV,2}$ and $\left. \frac{\partial^2}{\partial \epsilon^2}  \Sigma_{ML} \right|_{\epsilon=0}$.

\begin{prop}  \label{prop: simplification_calcul_d2_toeplitz}

Assume that conditions \ref{cond: Ktheta} and \ref{cond: toeplitz} are verified.

For $\epsilon = 0$,
\[
 \Sigma_{ML} = \frac{1}{2} M\left( \frac{f_{\theta}^2 }{ f^2 } \right),
\]
\begin{eqnarray*}
 \Sigma_{CV,1}  & = & 8  M\left(\frac{1}{f}\right)^{-6} M\left( \frac{f_{\theta}}{f^2} \right)^2 M\left(\frac{1}{f^2}\right) \\
& & + 8 M\left(\frac{1}{f}\right)^{-4} M\left( \frac{f_{\theta}^2}{f^4} \right) \\
& & - 16 M\left(\frac{1}{f}\right)^{-5} M\left( \frac{f_{\theta}}{f^2} \right) M\left(\frac{f_{\theta}}{f^3}\right),
\end{eqnarray*}

\begin{eqnarray*}
 \Sigma_{CV,2}  & = & 2  M\left(\frac{1}{f}\right)^{-3} \left\{ M\left( \frac{f_{\theta}^2}{f^3} \right) M\left(\frac{1}{f}\right) - M\left( \frac{f_{\theta}}{f^2} \right)^2  \right\}, \\
\end{eqnarray*}

and
\begin{eqnarray*}
\left. \frac{\partial^2}{\partial \epsilon^2}  \Sigma_{ML} \right|_{\epsilon=0} & =  & 
\frac{2}{3} M\left( \frac{ f_{\theta} }{f^2} \right) M\left( \frac{ f_t^2 ~ f_{\theta}}{ f^2 } \right) \\
& - & \frac{4}{3} M\left( \frac{ 1 }{f} \right) M\left( \frac{ f_{t,\theta} ~ f_t ~ f_{\theta}}{ f^2 } \right)
- \frac{4}{3} M\left( \frac{ f_{\theta} }{f^2} \right) M\left( \frac{ f_{t,\theta} ~ f_{t}}{ f } \right) \\
& + & \frac{2}{3} M\left( \frac{ 1 }{f} \right) M\left( \frac{ f_t^2 ~ f_{\theta}^2}{ f^3 } \right)
+ \frac{2}{3} M\left( \frac{ f_{\theta}^2 }{f^3} \right) M\left( \frac{ f_t^2 }{ f } \right) \\
& - & \frac{2}{3} M\left( \frac{ f_{t,t} ~ f_{\theta}^2 }{f^3} \right) \\
& + & \frac{2}{3} M\left( \frac{ 1 }{f} \right) M\left( \frac{ f_{t,\theta}^2  }{ f } \right) \\
& + & \frac{2}{3} M\left( \frac{ f_{t,t,\theta} ~ f_{\theta} }{f^2} \right). 
\end{eqnarray*}
\end{prop}

Proposition \ref{prop: simplification_calcul_d2_toeplitz} is proved in the supplementary material.

An interesting remark can be made on $\Sigma_{CV,2}$. Using Cauchy-Schwartz inequality, we obtain
\[
 \Sigma_{CV,2} =  2  M\left(\frac{1}{f}\right)^{-3} \left[  M\left\{ \left(\frac{f_{\theta}}{f^{\frac{3}{2}}}\right)^2 \right\} M\left\{ \left(\frac{1}{f^{\frac{1}{2}}}\right)^2 \right\} - M \left\{  \frac{f_{\theta}}{f^{\frac{3}{2}}}   \frac{1}{f^{\frac{1}{2}}}   \right\}^2 \right]
\geq  0,
\]
so that the limit of the second derivative with respect to $\theta$ of the CV criterion at $\theta_0$ is indeed non-negative. Furthermore,
for the limit to be zero, it is necessary that $\frac{f_{\theta}}{f^{\frac{3}{2}}}$ be proportional to 
$\frac{1}{f^{\frac{1}{2}}}$, that is to say $f_{\theta}$ be proportional to $f$. This is equivalent to $\frac{\partial K_{\theta_0}}{\partial \theta}$
being proportional to $K_{\theta_0}$ on $\ZZ$, which happens only when around $\theta_0$,
$K_\theta\left(i\right) = \frac{\theta}{\theta_0} K_{\theta_0}\left(i\right)$, for $i \in \ZZ$. Hence around $\theta_0$, $\theta$ would be a global variance parameter.
Therefore, we have shown that for the regular grid in dimension one, the asymptotic variance is positive, as long as $\theta$ is not only a global variance parameter.

\bibliographystyle{model1b-num-names}
\bibliography{Biblio}

\end{document}